\documentclass[12pt,reqno]{amsart}

\usepackage{amsmath,amsfonts,amsthm,amssymb,amsxtra}
\usepackage{bbm} 
\usepackage{dsfont}
\usepackage[hidelinks]{hyperref} 
\usepackage{color}
\usepackage{comment}
\usepackage[numbers]{natbib}
\usepackage{enumitem}
\usepackage{moreenum}
\usepackage[greek,english]{babel}

\numberwithin{equation}{section}
\newtheorem{theorem}{Theorem}[section]
\newtheorem{proposition}[theorem]{Proposition}
\newtheorem{lemma}[theorem]{Lemma}
\newtheorem{corollary}[theorem]{Corollary}

\theoremstyle{definition}

\theoremstyle{remark}

\newtheorem{remark}[theorem]{Remark}
\newtheorem{remarks}[theorem]{Remarks}

\newcommand{\1}{\mathds{1}}

\renewcommand{\epsilon}{\varepsilon}

\newcommand{\N}{\mathbb{N}}

\renewcommand{\phi}{\varphi}
\newcommand{\R}{\mathbb{R}}

\newcommand{\Sph}{\mathbb{S}}

\newcommand{\Z}{\mathbb{Z}}
\newcommand{\Haus}{\mathcal{H}}

\DeclareMathOperator{\dist}{dist}
\DeclareMathOperator{\diam}{diam}

\DeclareMathOperator{\ran}{ran}

\DeclareMathOperator{\supp}{supp}

\DeclareMathOperator{\Tr}{Tr}

\newcommand{\limplus}{{\mathchoice{\vcenter{\hbox{$\scriptstyle +$}}}
  {\vcenter{\hbox{$\scriptstyle +$}}}
  {\vcenter{\hbox{$\scriptscriptstyle +$}}}
  {\vcenter{\hbox{$\scriptscriptstyle +$}}}
}}
\newcommand{\limminus}{{\mathchoice{\vcenter{\hbox{$\scriptstyle -$}}}
  {\vcenter{\hbox{$\scriptstyle -$}}}
  {\vcenter{\hbox{$\scriptscriptstyle -$}}}
  {\vcenter{\hbox{$\scriptscriptstyle -$}}}
}}
\newcommand{\limpm}{{\mathchoice{\vcenter{\hbox{$\scriptstyle \pm$}}}
  {\vcenter{\hbox{$\scriptstyle \pm$}}}
  {\vcenter{\hbox{$\scriptscriptstyle \pm$}}}
  {\vcenter{\hbox{$\scriptscriptstyle \pm$}}}
}}

\begin{document}

\title[Semiclassical inequalities on convex domains]{Semiclassical inequalities for Dirichlet and Neumann Laplacians on convex domains}

\author{Rupert L. Frank}
\address[Rupert L. Frank]{Mathe\-matisches Institut, Ludwig-Maximilians Universit\"at M\"unchen, The\-resienstr.~39, 80333 M\"unchen, Germany, and Munich Center for Quantum Science and Technology, Schel\-ling\-str.~4, 80799 M\"unchen, Germany, and Mathematics 253-37, Caltech, Pasa\-de\-na, CA 91125, USA}
\email{r.frank@lmu.de}

\author{Simon Larson}
\address{\textnormal{(Simon Larson)} Mathematical Sciences, Chalmers University of Technology and the University of Gothenburg, SE-41296 Gothenburg, Sweden}
\email{larsons@chalmers.se}

\newcommand\blfootnote[1]{%
  \begingroup
  \renewcommand\thefootnote{}\footnote{#1}%
  \addtocounter{footnote}{-1}%
  \endgroup
}

\thanks{\copyright\, 2025 by the authors. This paper may be reproduced, in its entirety, for non-commercial purposes.\\
	Partial support through US National Science Foundation grant DMS-1954995 (R.L.F.), the German Research Foundation grants EXC-2111-390814868 and TRR 352-Project-ID 470903074 (R.L.F.), the Knut and Alice Wallenberg foundation grant KAW 2017.0295 (S.L.), as well as the Swedish Research Council grant no.~2023-03985 (S.L.) is acknowledged.}

\begin{abstract}
	We are interested in inequalities that bound the Riesz means of the eigenvalues of the Dirichlet and Neumann Laplacians in terms of their semiclassical counterpart. We show that the classical inequalities of Berezin--Li--Yau and Kr\"oger, valid for Riesz exponents $\gamma\geq 1$, extend to certain values $\gamma<1$, provided the underlying domain is convex. We also study the corresponding optimization problems and describe the implications of a possible failure of P\'olya's conjecture for convex sets in terms of Riesz means. These findings allow us to describe the asymptotic behavior of solutions of a spectral shape optimization problem for convex sets.
\end{abstract}

\maketitle

\setcounter{tocdepth}{1}
\tableofcontents

\section{Introduction and main results}

\subsection{Weyl asymptotics and semiclassical inequalities}
\label{sec: asymptotics and semiclassical inequalities}

In this paper we are concerned with sharp inequalities for the eigenvalues of the Laplacian in bounded convex subsets of Euclidean space. Before turning to the particularities of convex domains in the next subsection, let us review some facts valid for general open sets $\Omega\subset\R^d$.

Let $-\Delta_\Omega^{\rm D}$ and $-\Delta_\Omega^{\rm N}$ denote the Dirichlet and Neumann realizations of the Laplace operator in an open set $\Omega\subset\R^d$, defined through quadratic forms; see, e.g., \cite[Section 3.1]{FrankLaptevWeidl}. When making statements that refer to either one of the two operators we write $-\Delta_\Omega^\sharp$. Under the assumption that the spectrum of $-\Delta_\Omega^\sharp$ is discrete, its eigenvalues listed in non-decreasing order and repeated according to their multiplicities are denoted by $\lambda_k^\sharp(\Omega)$, $k\in\N$.

The celebrated Weyl law states that
\begin{equation}\label{eq: Weyls law}
 	\#\{k \in \N: \lambda_k^\sharp(\Omega) <\lambda\} = L_{0,d}^{\rm sc}|\Omega|\lambda^{\frac{d}2} + o(\lambda^{\frac{d}2})\quad \mbox{as }\lambda\to \infty\,,
\end{equation}
where $L_{0,d}^{\rm sc}$ is a constant that depends only on $d$ and is given explicitly in \eqref{eq: semiclassical constant} below. It is remarkable that these asymptotics are valid for either choice of the boundary condition $\sharp$ and that the leading term depends on the underlying domain only through its measure. The asymptotics \eqref{eq: Weyls law} go back to~\cite{WeylAsymptotic} and are valid in the Dirichlet case under the sole assumption that $|\Omega|<\infty$ and in the Neumann case under mild additional conditions. We refer to \cite[Sections 3.2 and 3.3]{FrankLaptevWeidl} for a textbook proof.

It was famously conjectured by P\'olya~\cite{Polya_61} that the leading asymptotic term in~\eqref{eq: Weyls law} bounds the Dirichlet counting function from above and the Neumann counting function from below, that is,
\begin{equation}\label{eq: Polya conjecture}
 	\#\{k \in \N: \lambda_k^{\rm D}(\Omega) <\lambda\} \leq L_{0,d}^{\rm sc} |\Omega|\lambda^{\frac{d}2} \leq \#\{k \in \N: \lambda_k^{\rm N}(\Omega) <\lambda\}
\end{equation} 
for all open $\Omega \subset \R^d$ and all $\lambda \geq 0$. If $d=1$ the conjecture can be proved by explicitly computing the eigenvalues, however, for $d\geq 2$ this conjecture remains wide open. In fact, it was shown only very recently that P\'olya's conjectured inequalities \eqref{eq: Polya conjecture} are valid when $\Omega$ is a Euclidean ball; see \cite{Filonov_etal_Polya} for the Dirichlet case, as well as for the Neumann case in dimension $d=2$, and see \cite{Filonov_etal_Polya2} for the Neumann case for $d\geq 3$.

In addition to the spectral counting functions $\#\{k \in \N: \lambda_k^\sharp(\Omega) <\lambda\}$, an object of interest are their Riesz means. For a parameter $\gamma>0$ they are defined by
\begin{equation*}
	\Tr(-\Delta_\Omega^\sharp-\lambda)_\limminus^\gamma = \sum_{k\geq 1}(\lambda-\lambda_k^\sharp(\Omega))_\limplus^\gamma\quad \mbox{for } \lambda \geq 0\,.
\end{equation*}
Here and in what follows we write $x_\limpm = \frac{1}{2}(|x|\pm x)$. To simplify notation we set
\begin{equation*}
 	\Tr(-\Delta_\Omega^\sharp-\lambda)_\limminus^0 = \#\{k \in \N: \lambda_k^\sharp(\Omega)< \lambda\} \quad \mbox{for }\lambda \geq 0\,.
\end{equation*}
The Riesz means are averages of counting functions and typically, as we will explain below, the asymptotics and inequalities that we are interested in for $\gamma>0$ follow from those for $\gamma=0$. For instance, the Weyl asymptotics \eqref{eq: Weyls law} imply that
\begin{equation}\label{eq: Weyls law Riesz means}
	\Tr(-\Delta_\Omega^\sharp-\lambda)_\limminus^\gamma = L_{\gamma, d}^{\rm sc} |\Omega|\lambda^{\gamma+\frac{d}{2}} + o(\lambda^{\gamma+\frac{d}2})\quad \mbox{as }\lambda \to \infty
\end{equation}
with
\begin{equation}\label{eq: semiclassical constant}
	L_{\gamma, d}^{\rm sc} := \frac{\Gamma(1+\gamma)}{(4\pi)^{\frac{d}2}\Gamma(1+\gamma + \frac{d}{2})}\,.
\end{equation}
Also, if P\'olya's conjectured inequalities \eqref{eq: Polya conjecture} are valid for a given open set $\Omega\subset\R^d$ and all $\lambda\geq 0$, then for this set we have
\begin{equation}\label{eq: Polya conjecture gamma}
	\Tr(-\Delta_\Omega^{\rm D}-\lambda)_\limminus^\gamma \leq L_{\gamma, d}^{\rm sc} |\Omega| \lambda^{\gamma+\frac{d}{2}}\leq \Tr(-\Delta_\Omega^{\rm N}-\lambda)_\limminus^\gamma
\end{equation}
for all $\lambda\geq 0$. 

Both inequalities in \eqref{eq: Polya conjecture gamma} are valid for $\gamma=1$, any open set $\Omega\subset\R^d$ and any $\lambda\geq 0$. The inequalities for $\gamma=1$ are known as the Berezin--Li--Yau (for $\sharp=$ D) and Kr\"oger inequalities (for $\sharp=$ N), proved in~\cite{Berezin,LiYau_83} and~\cite{Kr92,Laptev2}, respectively. By the Aizenman--Lieb principle \cite{AizenmanLieb} this implies the validity of the inequalities in \eqref{eq: Polya conjecture gamma} for all $\gamma\geq 1$.


\subsection{Semiclassical inequalities on convex domains}

In the present paper we will investigate the validity of the conjectured inequalities \eqref{eq: Polya conjecture gamma} for \emph{convex} sets $\Omega$.

One of the main results in this paper is that in the context of convex sets $\Omega \subset \R^d$ both the first and second inequality in~\eqref{eq: Polya conjecture gamma} extends to a range of $\gamma<1$. In the setting of convex sets this provides an improvement of the Berezin--Li--Yau and Kr\"oger inequalities.

\begin{theorem}\label{thm: Extended range of semiclassical ineq Dir}
	Fix any $d\geq 1$. There exists $\gamma <1$ such that for all open, bounded, convex $\Omega \subset \R^d$ and $\lambda \geq 0,$
	\begin{equation*}
		\Tr(-\Delta_\Omega^{\rm D}-\lambda)_\limminus^\gamma \leq L_{\gamma, d}^{\rm sc} |\Omega| \lambda^{\gamma + \frac{d}{2}}\,.
	\end{equation*}
\end{theorem}

\begin{theorem}\label{thm: Extended range of semiclassical ineq Neu}
	Fix any $d\geq 1$. There exists $\gamma <1$ such that for all open, bounded, convex $\Omega \subset \R^d$ and $\lambda \geq 0,$
	\begin{equation*}
		\Tr(-\Delta_\Omega^{\rm N}-\lambda)_\limminus^\gamma \geq L_{\gamma, d}^{\rm sc} |\Omega| \lambda^{\gamma + \frac{d}{2}}\,.
	\end{equation*}
\end{theorem}

\begin{remarks}
    Several remarks are in order.
    \begin{enumerate}
        \item[(a)] According to the Aizenman--Lieb principle, mentioned in the previous subsection, the validity of one of the inequalities in Theorems \ref{thm: Extended range of semiclassical ineq Dir} and \ref{thm: Extended range of semiclassical ineq Neu} for some $\gamma<1$ implies its validity for all larger values of $\gamma$ as well.
        \item[(b)] The exponents $\gamma$ in Theorems \ref{thm: Extended range of semiclassical ineq Dir} and \ref{thm: Extended range of semiclassical ineq Neu} are explicit in the sense that no compactness argument is used in their proof. Since our proof relies on some of the results in our papers \cite{FrankLarson_Crelle20,FrankLarson_24,FrankLarson_heatkernels}, where we did not provide numerical values for the constants in the error terms, we refrain here from giving a numerical value. We emphasize that no compactness argument was used in \cite{FrankLarson_Crelle20,FrankLarson_24,FrankLarson_heatkernels}, so the constants there can, in principle, be tracked.
        \item[(c)] The proofs of Berezin--Li--Yau and Kr\"oger are all based in an essential way on the convexity of $E\mapsto(E-\lambda)_\limminus^\gamma$ for $\gamma=1$. This convexity breaks down for $\gamma<1$, while we show that the inequalities continue to hold, at least for convex sets $\Omega$. This should be compared to the results in \cite{Erdos_etal} and \cite{Frank_etal_09}, which concern the case of a constant magnetic field. In this case the validity of the analogous inequalities breaks down exactly at $\gamma=1$. 
    \end{enumerate}
\end{remarks}


\subsection{The optimization problems}

The semiclassical inequalities discussed in the previous subsection lead in a natural way to certain optimization problems. In this and the next subsection we will state these problems and describe our results.

We need to introduce some notation. For $d\geq 1$ let $\mathcal{C}_d$ denote the collection of all open, bounded, non-empty, and convex subsets of $\R^d$. 
Define the critical exponents
\begin{equation*}
		\gamma_d^{\rm D} :=\inf\Bigl\{\gamma\geq 0: \Tr(-\Delta_\Omega^{\rm D}-\lambda)_\limminus^\gamma \leq L^{\rm sc}_{\gamma,d}|\Omega|\lambda^{\gamma+\frac{d}2} \mbox{ for all }\Omega \in \mathcal{C}_d, \lambda \geq 0\Bigr\} \,\,
\end{equation*}
and
\begin{equation*}
		\gamma_d^{\rm N} :=\inf\Bigl\{\gamma\geq 0: \Tr(-\Delta_\Omega^{\rm N}-\lambda)_\limminus^\gamma \geq L^{\rm sc}_{\gamma,d}|\Omega|\lambda^{\gamma+\frac{d}2} \mbox{ for all }\Omega \in \mathcal{C}_d, \lambda \geq 0\Bigr\} \,.
\end{equation*}
The exponent $\gamma_d^\sharp$ can be characterized as the smallest number so that if $\gamma \geq \gamma_d^{\rm D}$, then
\begin{equation*}
	\Tr(-\Delta_\Omega^{\rm D}-\lambda)_\limminus^\gamma \leq L_{\gamma,d}^{\rm sc} |\Omega| \lambda^{\gamma+\frac{d}2}
\end{equation*}
for all bounded, open, and convex $\Omega \subset \R^d$ and $\lambda \geq 0$.
If $\gamma \geq \gamma_d^{\rm N}$ then, similarly,
\begin{equation*}
	\Tr(-\Delta_\Omega^{\rm N}-\lambda)_\limminus^\gamma \geq L_{\gamma,d}^{\rm sc} |\Omega| \lambda^{\gamma+\frac{d}2}
\end{equation*}
for all bounded, open, and convex $\Omega \subset \R^d$ and $\lambda \geq 0$. With this definition P\'olya's conjecture restricted to convex domains, can be stated as $\gamma_d^{\sharp} =0$. In particular, the validity of P\'olya's conjecture in the one-dimensional case entails that $\gamma_1^\sharp=0$. Our Theorems \ref{thm: Extended range of semiclassical ineq Dir} and \ref{thm: Extended range of semiclassical ineq Neu} can be stated succinctly as
$$
\gamma_d^\sharp < 1\,.
$$

The following result gives a lower bound on the critical Riesz exponent $\gamma_d^\sharp$ in $d$ dimensions in terms of the critical Riesz exponent $\gamma_{d-1}^\sharp$ in $d-1$ dimensions. 

\begin{proposition}\label{dimred}
    Fix $d\geq 2$ and $\sharp\in\{ \rm D, \rm N\}$. Then
    $$
    \gamma_d^{\sharp} \geq \Bigl(\gamma_{d-1}^{\sharp}-\frac12\Bigr)_{\!\limplus} \,.
    $$
\end{proposition}

We now consider the optimization problems associated to the inequalities in \eqref{eq: Polya conjecture gamma} at the critical exponent $\gamma=\gamma_d^\sharp$. As we have just shown, $\gamma_d^\sharp\geq (\gamma_{d-1}^{\sharp}-\frac12)_\limplus$ and the behavior that we describe depends on whether this inequality is strict or not.

\begin{theorem}\label{thm: Main conclusions gamma_d}
	Fix $d\geq 2$ and $\sharp \in \{{\rm D}, {\rm N}\}$. It holds that 
	\begin{enumerate}[label=\textup{(}\hspace{-0.4pt}\alph*\textup{)}]
		\item\label{itm: Main thm existence sup} if $\gamma_d^{\sharp}>(\gamma_{d-1}^{\sharp}-\frac12)_\limplus$, then there is a bounded convex open non-empty $\Omega_*\subset\R^d$ and a $\lambda_*>0$ such that
	$$
	\Tr(-\Delta_{\Omega_*}^{\sharp}-\lambda_*)_\limminus^{\gamma_d^{\sharp}} = L_{\gamma_d^{\sharp},d}^{{\rm sc}} |\Omega_*| \lambda_*^{\gamma_d^{\sharp}+\frac d2} \,.
	$$

	\item\label{itm: Main thm cylinder sequence} if $\gamma_d^{\sharp}= \gamma_{d-1}^{\sharp}-\frac12\geq 0$, then there is a bounded convex open non-empty $\omega_*\subset \R^{d-1}$ and a $\lambda_*>0$ such that
    $$
        \Tr(-\Delta_{\omega_*}^{\sharp}-\lambda_*)_\limminus^{\gamma_{d-1}^{\sharp}}
    = {L_{\gamma_{d-1}^{\sharp},d-1}^{{\rm sc}} |\omega_*| \lambda_*^{\gamma_{d-1}^{\sharp}+\frac {d-1}2}}
    $$
    and, if
	\begin{equation*}
		\Omega(\lambda) := \bigl(\bigl(\tfrac{\lambda_*}{\lambda}\bigr)^{\frac{1}2}\omega_*\bigr)\times \bigl(0, \bigl(\tfrac{\lambda}{\lambda_*}\bigr)^{\frac{d-1}2}\bigr) \subset\R^d \,,
	\end{equation*}
	then 
	$$
	\lim_{\lambda\to\infty} \frac{\Tr(-\Delta_{\Omega(\lambda)}^{\sharp}-\lambda)_\limminus^{\gamma_d^{\sharp}}}{L_{\gamma_d^{\sharp},d}^{{\rm sc}} |\Omega(\lambda)| \lambda^{\gamma_d^{\sharp}+\frac d2}} = 1 \,.
	$$
    \end{enumerate}
\end{theorem}

We emphasize that Theorem \ref{thm: Main conclusions gamma_d} is void when P\'olya's conjecture holds for convex sets, that is, when $\gamma_d^\sharp =0$ for all $d\geq 1$. Nevertheless, we believe that it is of interest and, in fact, provides a possible route towards proving P\'olya's conjecture among convex sets. If P\'olya's conjecture fails within the class of convex subsets of $\R^d$, then we have $\gamma_d^\sharp >0$ by the continuity of the Riesz means with respect to the exponent $\gamma$. Thus, either \ref{itm: Main thm existence sup} or~\ref{itm: Main thm cylinder sequence} in Theorem \ref{thm: Main conclusions gamma_d} proves the existence of an optimizing convex set for a shape optimization problem, namely $\Omega_*$ or $\omega_*$. For these optimizing sets one might hope to deduce additional properties by studying the associated Euler--Lagrange equations. If by utilizing these additional properties of $\Omega_*$ and $\omega_*$ one could conclude that the corresponding semiclassical inequality is strict, one would reach a contradiction and prove the validity of P\'olya's conjecture. What one has gained by passing to positive $\gamma$ compared to the original optimization problem of P\'olya (which corresponds to $\gamma=0$) are the following three items: 1) One has existence of an optimizing set, 2) one knows the value of the optimum, and 3) one has a smoother variational problem for $\gamma>0$ than for $\gamma=0$. Carrying out this strategy remains an open problem.


\subsection{Consequences in asymptotic shape optimization}

In this subsection we turn our attention to the following two shape optimization problems, depending on parameters $\gamma\geq 0$ and $\lambda\geq 0$, given by
\begin{align*}
	M^{\rm D}_\gamma(\lambda)&:=\sup\bigl\{\Tr(-\Delta^{\rm D}_\Omega-\lambda)_\limminus^\gamma: \Omega \in \mathcal{C}_d, |\Omega|=1\bigr\}\,,\\
    M^{\rm N}_\gamma(\lambda)&:=\inf\bigl\{\Tr(-\Delta^{\rm N}_\Omega-\lambda)_\limminus^\gamma: \Omega \in \mathcal{C}_d, |\Omega|=1\bigr\}\,.
\end{align*}	
More precisely, for fixed $\gamma\geq 0$, we shall be interested in the numbers $M^\sharp_\gamma(\lambda)$ and the corresponding minimizers or maximizers in the limit $\lambda\to\infty$.

These problems have been studied in the Dirichlet case for $\gamma\geq 1$ in \cite{LarsonJST,FrankLarson_Crelle20} (see also \cite{Freitas_17} for related results in the case $\gamma=1, \sharp = \rm D$). Our goal here is to enlarge the parameter regime to $\gamma\geq 0$ and to extend the results to the Neumann case.

Let us explain the connection between these asymptotic shape optimization problems and the rest of this paper. It is known that for $\gamma>0$ there is a subleading term in the Weyl asymptotics \eqref{eq: Weyls law Riesz means}, namely, as $\lambda\to\infty$
\begin{equation}\label{eq: Weyls law Riesz means second}
    \begin{aligned}
        \Tr(-\Delta^{\rm D}_\Omega-\lambda)_\limminus^\gamma = L_{\gamma,d}^{\rm sc} |\Omega| \lambda^{\gamma+\frac d2} - \frac14 L_{\gamma,d-1}^{\rm sc} \mathcal H^{d-1}(\partial\Omega) \lambda^{\gamma+\frac{d-1}{2}} + o(\lambda^{\gamma+\frac{d-1}{2}}) \,,\\
        \Tr(-\Delta^{\rm N}_\Omega-\lambda)_\limminus^\gamma = L_{\gamma,d}^{\rm sc} |\Omega| \lambda^{\gamma+\frac d2} + \frac14 L_{\gamma,d-1}^{\rm sc} \mathcal H^{d-1}(\partial\Omega) \lambda^{\gamma+\frac{d-1}{2}} + o(\lambda^{\gamma+\frac{d-1}{2}}) \,.
    \end{aligned}
\end{equation}
Here $\mathcal H^{d-1}(\partial\Omega)$ denotes the surface area of the boundary of $\Omega$ and we emphasize the different signs of the subleading term depending on the boundary condition. 

Now if the asymptotics \eqref{eq: Weyls law Riesz means second} were uniform with respect to all convex sets $\Omega$ with unit measure, then the asymptotic shape optimization problem would reduce, in both cases, to the problem of minimizing $\mathcal H^{d-1}(\partial\Omega)$ among all convex open sets $\Omega$ with $|\Omega|=1$. This is the classical isoperimetric problem, which is solved (precisely) by Euclidean balls. Thus, one might arrive at the expectation that optimizers of the problems $M_\gamma^\sharp(\lambda)$ converge, as $\lambda\to\infty$, to balls. As we will see in this subsection, whether or not this expectation is correct depends on the validity of P\'olya's conjecture for convex sets and the relation between $\gamma$ and a certain critical Riesz exponent.

Our first result concerns the asymptotics of $M_\gamma^\sharp(\lambda)$ as $\lambda\to\infty$. By comparison with a fixed convex set of unit measure it is easy to see that
\begin{align*}
    \liminf_{\lambda \to \infty}\frac{M_{\gamma}^{\sharp}(\lambda)}{L_{\gamma,d}^{\rm sc}\lambda^{\gamma+\frac{d}2}} \geq 1\quad \mbox{if }\sharp = {\rm D}\quad \mbox{and}\quad
    \limsup_{\lambda \to \infty}\frac{M_{\gamma}^{\sharp}(\lambda)}{L_{\gamma,d}^{\rm sc}\lambda^{\gamma+\frac{d}2}} \leq 1\quad \mbox{if }\sharp = {\rm N}\,.
\end{align*}
Moreover, equality holds for $\gamma\geq\gamma_d^\sharp$ by definition of the latter number. This condition for equality may not be necessary, however. The following result provides a necessary and sufficient condition in terms of the number $\gamma_{d-1}^\sharp-\frac12$, which is $\leq\gamma_d^\sharp$ by Proposition~\ref{dimred}. We find it remarkable that it is the critical Riesz exponent in dimension $d-1$ that is relevant for the shape optimization problem in dimension $d$.

\begin{proposition}\label{prop:shapeoptasymp}
    Let $d\geq 2$, $\gamma\geq 0$ and $\sharp \in \{{\rm D}, {\rm N}\}$. If $\gamma\geq \gamma_{d-1}^\sharp-\frac12$, then
    \begin{equation*}
        \lim_{\lambda \to \infty}\frac{M_{\gamma}^{\sharp}(\lambda)}{L_{\gamma,d}^{\rm sc}\lambda^{\gamma+\frac{d}2}} =1 \,,
    \end{equation*}
    while if $\gamma < \gamma_{d-1}^\sharp -\frac12$, then
    \begin{equation*}
    \begin{aligned}
        \lim_{\lambda \to \infty}\frac{M_{\gamma}^{\sharp}(\lambda)}{L_{\gamma,d}^{\rm sc}\lambda^{\gamma+\frac{d}2}}>1 \quad  \mbox{if }\sharp ={\rm D} \quad \mbox{and}\quad
        \lim_{\lambda \to \infty}\frac{M_{\gamma}^{\sharp}(\lambda)}{L_{\gamma,d}^{\rm sc}\lambda^{\gamma+\frac{d}2}}<1\quad  \mbox{if }\sharp ={\rm N}\,.
    \end{aligned}
    \end{equation*}
\end{proposition}

Note that if $d=2$, then for any $\gamma\geq 0$
$$
    \lim_{\lambda \to \infty}\frac{M_{\gamma}^{\sharp}(\lambda)}{L_{\gamma,2}^{\rm sc}\lambda^{\gamma+1}} = 1 \,.
$$
This follows from $\gamma_1^\sharp=0$.

A consequence of this proposition is that the expectation that optimizing sets converge to balls necessarily must fail for $0\leq \gamma < \gamma_{d-1}^\sharp -\frac12$.

Now we turn to the finer question of describing the asymptotic behavior of sets realizing the extremum in $M^\sharp_\gamma(\lambda)$ as $\lambda\to\infty$. That sets realizing the supremum defining $M_\gamma^{\rm D}(\lambda)$ exist is proved, for instance, in~\cite[Lemma 3.1]{LarsonJST} and we modify the proof in Lemma \ref{exshapeoptneumann} to handle the Neumann case. We want to stress, however, that our results do not need the existence of an optimizer, but rather are valid for any sequence of almost optimizing sets.

To formulate our results we need to introduce some terminology. For $\Omega, \Omega' \in \mathcal{C}_d$ the (complementary) Hausdorff distance between these sets is defined by
\begin{equation*}
    d^{H}(\Omega, \Omega') = \max\Bigl\{\sup_{x \in K\setminus\Omega}\dist(x, K\setminus\Omega'), \sup_{x \in K\setminus\Omega'}\dist(x, K\setminus\Omega)\Bigr\}\,,
\end{equation*}
where $K\subset \R^d$ is a compact set with $\Omega, \Omega' \subset K$ (the distance is independent of the choice of $K$). It is a classical result, often attributed to Blaschke, that the metric space $(\mathcal{C}_d, d^H)$ is locally compact in the sense that, if $\{\Omega_j\}_{j\geq 1}\subset \mathcal{C}_d$ and there is a compact set $K\subset \R^d$ so that $\Omega_j \subset K$ for each $j\geq 1$, then there exists a subsequence along which $\Omega_j$ converges to an open convex set $\Omega$ (possibly empty) with respect to the distance~$d^H$ (see, e.g.,~\cite[Corollary 2.2.26 \& eq.~(2.19)]{HenrotPierre_18}). As it will be important for us later on we mention that $\Omega \mapsto |\Omega|$ is continuous in $(\mathcal{C}_d, d^H)$. Also the map $\Omega \mapsto \Haus^{d-1}(\partial\Omega)$ is continuous in $(\mathcal{C}_d, d^H)$ away from the empty set. For basic properties concerning the Hausdorff distance we refer to \cite{HenrotPierre_18}.

Define the inradius of a set $\Omega \subset \R^d$ as the radius of the largest ball contained in $\Omega$, that is,
\begin{equation*}
	r_{\rm in}(\Omega) := \sup_{x\in \Omega}\dist(x, \Omega^c)\,.
\end{equation*}
Let $\{\Omega_j\}_{j\geq 1}$ be a sequence of open convex subsets of $\R^d$ with $|\Omega_j|=1$ for each~$j$. If $\liminf_{j\to \infty}r_{\rm in}(\Omega_j)>0$ then, modulo translation, the Blaschke selection theorem implies that $\{\Omega_j\}_{j\geq 1}$ is precompact in $(\mathcal{C}_d, d^H)$. However, if $\liminf_{j \to \infty}r_{\rm in}(\Omega_j)=0$ then along some subsequence $\Omega_j$ collapses onto a lower dimensional set. This collapse can happen in many different manners and here we shall mostly be interested in the particular case when the convex sets become elongated in one direction. Precisely, the regime of interest is when $\{\Omega_j\}_{j\geq1}$ is such that $|\Omega_j|=1$, $\lim_{j\to \infty}r_{\rm in}(\Omega_j)=0$, and there exists a constant $C>0$ such that after appropriate translations and rotations
\begin{equation*}
	\Omega_j \subset \{x = (x', x_d) \in \R^d: |x'| < Cr_{\rm in}(\Omega_j)\} \quad \mbox{for all }j \geq 1\,.
\end{equation*}
We shall refer to this scenario by saying that the sequence $\Omega_j$ \emph{collapses in a spaghetti-like manner}.

The following theorem describes the asymptotic behavior of (almost) optimizers to the problems $M_\gamma^\sharp(\lambda)$ as $\lambda\to\infty$.

\begin{theorem}\label{thm: shape optimization convex}
	Let $d\geq 2$, $\gamma\geq 0$ and $\sharp \in \{{\rm D}, {\rm N}\}$. Let $\{\lambda_j\}_{j\geq 1}\subset (0, \infty)$ be a sequence with $\lim_{j\to \infty}\lambda_j = \infty$ and let $\{\Omega_j\}_{j\geq 1}$ be a sequence of open convex sets with $|\Omega_j|=1$ satisfying
	\begin{equation*}
		\lim_{j \to \infty} \frac{\Tr(-\Delta^{\sharp}_{\Omega_j}-\lambda_j)_\limminus^\gamma-M_\gamma^{\sharp}(\lambda_j)}{\lambda_j^{\gamma + \frac{d-1}{2}}} = 0\,.
	\end{equation*}
	\begin{enumerate}[label=\textup{(}\hspace{-0.3pt}\roman*\textup{)}]
		\item\label{itm: Shape opt thm super critical} If $\gamma >(\gamma_{d-1}^{\sharp}-\tfrac{1}2)_\limplus$, then up to translations the sequence $\{\Omega_j\}_{j\geq 1}$ converges with respect to the Hausdorff distance to a ball $B\subset\R^d$ with $|B|=1$ and
        \begin{equation*}
            M_\gamma^{\sharp}(\lambda) = \Tr(-\Delta_B^{\sharp}-\lambda)_\limminus^\gamma + o(\lambda^{\gamma+ \frac{d-1}2})
            \qquad\text{as}\ \lambda\to\infty \,.
        \end{equation*}
		\item\label{itm: Shape opt thm counting super critical} If $\gamma=0$ and $\gamma_{d-1}^{\sharp}<\frac{1}{2}$, then 
		\begin{equation*}
			\lim_{j\to \infty}r_{\rm in}(\Omega_j)\sqrt{\lambda_j}=\infty \,.
		\end{equation*}
        \item\label{itm: Shape opt thm sub critical case} If $0 \leq \gamma \leq  \gamma_{d-1}^{\sharp}-\frac12$ and (at least) one of these inequalities is strict, then
		\begin{equation*}
			\limsup_{j\to \infty}r_{\rm in}(\Omega_j)\sqrt{\lambda_j}<\infty  \,.
		\end{equation*}
		and the sequence $\{\Omega_j\}_{j\geq 1}$ collapses in a spaghetti-like manner.  
	\end{enumerate}
\end{theorem}

Note that this theorem covers the full range of parameters $\gamma\geq 0$, except for the special case $\gamma=0$ when $\gamma_{d-1}^\sharp=\frac12$.

Note that for $d=2$ we get a rather complete solution of the asymptotic shape optimization problem for any $\gamma>0$ (since $\gamma_1^\sharp=0$). In particular, we prove the expected convergence to a disk. In higher dimensions the result is conditional on the value of the (unknown) parameter $\gamma_{d-1}^\sharp$, but since $\gamma_{d-1}^\sharp<1$ by Theorems \ref{thm: Extended range of semiclassical ineq Dir} and \ref{thm: Extended range of semiclassical ineq Neu} the theorem improves substantially our previous result in \cite{LarsonJST,FrankLarson_Crelle20}. (We use this opportunity to mention a typo in \cite[Remark 1.4]{FrankLarson_Crelle20}: $\limsup$ there should be $\liminf$.)

Shape optimization has long been an active topic in spectral theory, indeed its history can be traced as far back as Rayleigh~\cite{Rayleigh_1877}. For a modern review of the topic we refer the reader to~\cite{Henrot_17}, and references therein. However, asymptotic problems in the spirit of those considered here have only rather recently received much attention. This new boost in interest was largely motivated by the connection between P\'olya's conjecture and the asymptotic behavior of optimal shapes for $\lambda_k^\sharp(\Omega)$ as $k \to \infty$ highlighted in~\cite{ColboisElSoufi_14} (see also~\cite{Freitas_etal_21}). Note that the optimization of $\lambda_k^\sharp$ is essentially the same as the $\gamma=0$ case of the problems studied here. Similarly the optimization of partial sums of the $\lambda_k^\sharp$ is essentially the same as the $\gamma=1$ case of the problems studied here. For further results of this nature we mention \cite{AntunesFreitas_13,BucurFreitas_13, vdBerg_15, vdBergBucurGittins_16, vdBergGittins_17, GittinsLarson_17, Freitas_17, LarsonAFM, Lagace20, BuosoFreitas_20}. Most of these results concern the case of optimizing individual eigenvalues (or the corresponding counting function) within rather small collections of sets. The fact that our results are valid in the comparatively large class of convex sets is one of the reasons we find it interesting. However, our results are also valid for certain subclasses of convex sets as in \cite{LarsonJST}, for instance for the subclass of planar convex polygons with an upper bound on the number of edges. We omit the details.


\subsection{Weyl asymptotics for Riesz means on collapsing convex sets}

The next theorem is an important ingredient in the proof of the results in the previous two subsections. It concerns the asymptotic behavior of Riesz means along collapsing sequences of convex domains. In this regime the ordinary Weyl law is not valid but, as we will show, a variant of it is, where only the non-collapsing directions become semiclassical while the collapsing ones retain their `quantum nature'.

\begin{theorem}
	\label{thm: Asymptotics degenerating convex sets}
	Let $\{\lambda_j\}_{j\geq 1}\subset (0, \infty)$ and $\{\Omega_j\}_{j\geq 1}$ be a sequence of bounded open convex sets in $\R^d$ satisfying
	\begin{equation*}
	 	0< \inf_{j\geq 1}r_{\rm in}(\Omega_j)\sqrt{\lambda_j}\leq \sup_{j\geq 1}r_{\rm in}(\Omega_j)\sqrt{\lambda_j} <\infty
	 	\quad \mbox{and}\quad
	 	\lim_{j\to \infty}|\Omega_j|\lambda_j^{\frac{d}2}=\infty\,.
	 \end{equation*} 
	 Fix $\gamma\geq 0$ and $\sharp \in \{{\rm D}, {\rm N}\}$. If the limit
	 \begin{equation*}
	 		\lim_{j \to \infty} \frac{\Tr(-\Delta_{\Omega_j}^\sharp-\lambda_j)_\limminus^\gamma}{L_{\gamma,d}^{\rm sc}|\Omega_j|\lambda_j^{\gamma+\frac{d}2}} 
	 \end{equation*}
	 exists, then there exists an open bounded convex non-empty set $\Omega_* \subset \R^d$ and an integer $1\leq m \leq d-1$ such that
	 \begin{equation*}
	 	\lim_{j \to \infty} \frac{\Tr(-\Delta_{\Omega_j}^\sharp-\lambda_j)_\limminus^\gamma}{L_{\gamma,d}^{\rm sc}|\Omega_j|\lambda_j^{\gamma+\frac{d}2}} = \frac{1}{L^{\rm sc}_{\gamma+ \frac{d-m}{2}, m}|\Omega_*|}\int_{P^\perp\Omega_*} \Tr(-\Delta_{\Omega_*(y)}^\sharp-1)_\limminus^{\gamma + \frac{d-m}{2}}\,dy \,,
	 \end{equation*}
	 where
	 \begin{align*}
	 	P^\perp\Omega_* &:= \{y\in \R^{d-m}: (x,y) \in \Omega_* \mbox{ for some } x \in \R^m\}\,,\\
	 	\Omega_*(y) &:= \{x \in \R^{m}: ( x,y) \in \Omega_*\}\,.
	 \end{align*}
\end{theorem}

\begin{remark} A couple of remarks:

\begin{enumerate}
    \item The integer $m$ in the statement corresponds to the number of directions in which every subsequence of $\{\Omega_j\}_{j\geq 1}$ collapses at the rate proportional to $1/\!\sqrt{\lambda_j}$. In these directions one does not observe semiclassical behavior and the limiting spectral problem retains a `quantum nature'. In particular, the case $m=d-1$ corresponds to the sequence collapsing in a spaghetti-like manner.

    \item In our applications of the theorem, the assumption that the limit exists is not a severe one. It will always be satisfied along a subsequence along which either the limsup or the liminf is attained.

	\item If $\lim_{j \to \infty}r_{\rm in}(\Omega_j)\sqrt{\lambda_j} = \infty$, then
\begin{equation*}
	\lim_{j \to \infty} \frac{\Tr(-\Delta^\sharp_{\Omega_j}-\lambda_j)_\limminus^\gamma}{L^{\rm sc}_{\gamma,d}|\Omega_j|\lambda_j^{\gamma+\frac{d}2}} = 1\,.
\end{equation*}
Indeed, for $\gamma>0$ this follows from \cite[Theorem 1.2]{FrankLarson_24} and for $\gamma=0$ we refer to Theorem~\ref{thm: quantitative Weyl law} below.
As will be discussed later, the quantity on the right-hand side of the theorem is generally not equal to $1$.

\item If $\lim_{j\to \infty}r_{\rm in}(\Omega_j)\sqrt{\lambda_j}=0$, then
$$
\lim_{j\to \infty}\frac{\Tr(-\Delta^{\rm D}_{\Omega_j}-\lambda_j)_\limminus^\gamma}{|\Omega_j|\lambda_j^{\gamma+\frac{d}2}} = 0
\qquad\text{and}\qquad
\lim_{j\to \infty}\frac{\Tr(-\Delta^{\rm N}_{\Omega_j}-\lambda_j)_\limminus^\gamma}{|\Omega_j|\lambda_j^{\gamma+\frac{d}2}} = \infty\,.
$$
Indeed, in the Dirichlet case this follows from the Hersch--Protter inequality~\cite{Hersch_60,Protter_81} (see also~\cite{FrankLaptevWeidl}), which says that if $r_{\rm in}(\Omega_j)\sqrt{\lambda_j}<\frac{\pi^2}{4}$ then $\Tr(-\Delta^{\rm D}_{\Omega_j}-\lambda_j)_\limminus^\gamma = 0$. In the Neumann case the assertion follows from Lemma~\ref{lem: small energy improved Kroger} below.

\item By Fubini's theorem we have
\begin{equation*}
	\int_{P^\perp \Omega_*} \frac{\Haus^{m}(\Omega_*(y))}{|\Omega_*|}\,dy  = \int_{P^\perp\Omega_*} \int_{\Omega_*(y)}\frac{dxdy}{|\Omega_*|} = \int_{\Omega_*}\frac{dz}{|\Omega_*|} = 1\,,
\end{equation*}
so the integral can be thought of as a weighted average with respect to $y\in P^\perp \Omega_*$ of the traces
$$
\frac{\Tr(-\Delta_{\Omega_*(y)}^\sharp-1)_\limminus^{\gamma + \frac{d-m}{2}}}{L^{\rm sc}_{\gamma+ \frac{d-m}{2}, m}\Haus^{m}(\Omega_*(y))}  \,.
$$
\end{enumerate}
\end{remark}

Theorem \ref{thm: Asymptotics degenerating convex sets} is related to results about partially semiclassical limits. These arise typically in spectral asymptotics where the standard Weyl term is infinite and where an asymptotic separation of variables occurs, such that some of the variables becomes semiclassical while effective lower-dimensional model operators emerge that act with respect to the non-semiclassical variables and depend parametrically on the semiclassical ones. Correspondingly the asymptotics involve integrals with respect to the eigenvalues of these model operators. In the present setting the variables $y\in\R^{d-m}$ become semiclassical and the lower dimensional model operators are the Laplacians on the slices $\Omega_*(y)$.

Such partially semiclassical asymptotics have been investigated for a long time and we refer to \cite[Chapter~5, Section:~Commentary and references to the literature]{BiSo} for many references. Those include, in particular, results by Solomyak and Vulis concerning a power-like degeneration of the coefficients of an operator close to the boundary of a domain; see \cite[Theorem 5.19]{BiSo}. Another relevant reference is \cite{Ta} concerning Laplace operators on a class of unbounded sets. Similar results have appeared in the setting of Schr\"odinger operators and we refer to \cite{Ro,Simon_83,So,CarlenFrankLarson} for representative results. A recent work where a similar phenomenon appears is \cite{Read_24}. It is also worth mentioning that the operator-valued Lieb--Thirring inequalities in \cite{LaptevWeidl_00} are in the spirit of these works.

We emphasize, however, that all of these previous papers concern spectral asymptotics for a single operator, while we consider asymptotics for a \emph{family} of operators.


\subsection{Some comments on the proofs}

We end this introduction by highlighting some of the ideas that we employ in proving the results that we have mentioned so far. An important ingredient will be the semiclassical asymptotics that we have shown in our recent works~\cite{FrankLarson_24,FrankLarson_heatkernels}. For what we do here the main point of those results is their uniformity in the geometry of the underlying set and, specifically, that the semiclassical approximation is valid for convex sets $\Omega$ as soon as $r_{\rm in}(\Omega)\sqrt\lambda \gg 1$.

The main focus of the present paper is to study the regime $r_{\rm in}(\Omega)\sqrt\lambda \lesssim 1$. To deal with this regime we will prove both \emph{asymptotic} and \emph{non-asymptotic bounds}. The non-asymptotic bounds are behind Theorems \ref{thm: Extended range of semiclassical ineq Dir} and \ref{thm: Extended range of semiclassical ineq Neu}, while the asymptotic bounds are behind Theorems~\ref{thm: Main conclusions gamma_d} and \ref{thm: shape optimization convex}.

The key \emph{non-asymptotic results} that we establish are universal improvements over the semiclassical inequalities of Berezin--Li--Yau and Kr\"oger. Those appear in Propositions~\ref{prop: improved Berezin} and \ref{prop: improved Kroger}. In contrast to earlier results in a similar spirit, which we recall in Subsection \ref{sec:improvedblyk}, our new results give an improvement that is only exponentially small in $\sqrt\lambda$ times the inradius, but this is sufficient for our purposes. The main technical ingredient in the proof of these improvements is a quantitative version of the Amrein--Berthier uncertainty principle; see Lemma \ref{uncertainty}.

The basic \emph{asymptotic result} in the regime $r_{\rm in}(\Omega)\sqrt\lambda\lesssim 1$ is Theorem \ref{thm: Asymptotics degenerating convex sets}, where we discover a partially semiclassical limit. We refer to the beginning of Section \ref{sec:asymptoticscollapsing} for a more detailed explanation of the mechanism behind our proof.

A crucial idea, related to our asymptotic analysis, that we would like to mention is that of \emph{dimensional reduction}. For instance, in the case of spaghetti-like degeneration a single dimension disappears and we are left with a $(d-1)$-dimensional set (essentially the cross-section of the original $d$-dimensional set). More generally, in the setting of Theorem \ref{thm: Asymptotics degenerating convex sets} $(d-m)$-directions may disappear and an $m$-dimensional set remains. We observe that such a loss of $d-m$ directions comes accompanied by an increase of $\gamma$ by $\frac{d-m}{2}$. One of the ideas behind our proofs of Theorems~\ref{thm: Main conclusions gamma_d} and \ref{thm: shape optimization convex} is that it is unfavorable to loose more than a single dimension. 

Some of these results can be understood from the point of view of compactness of sequences of convex sets. Loss of compactness can be characterized in terms of the asymptotic behavior of the semiaxes of the corresponding John ellipsoids. Two facts about bounded convex sets $\Omega\subset\R^d$ that we will use frequently in our analysis of sequences of convex sets are the inequalities
\begin{equation}\label{eq: inradius bound}
	\frac{|\Omega|}{\Haus^{d-1}(\partial\Omega)} \leq r_{\rm in}(\Omega) \leq d\frac{|\Omega|}{\Haus^{d-1}(\partial\Omega)}
 \quad \mbox{and}\quad \diam(\Omega) \leq C_d \frac{|\Omega|}{r_{\rm in}(\Omega)^{d-1}}\,,
\end{equation}
see, for instance,~\cite{LarsonJST}.

With the exception of the question of whether P\'olya's conjecture holds within the class of convex sets, our Theorems \ref{thm: Extended range of semiclassical ineq Dir}, \ref{thm: Extended range of semiclassical ineq Neu}, \ref{thm: Main conclusions gamma_d} and \ref{thm: shape optimization convex} provide a rather complete picture in the setting of convex sets. To which extent analogues of these theorems hold without the convexity assumption remains a formidable \emph{open problem}.

\smallskip

Throughout the paper we shall use the asymptotics notation $\lesssim, \gtrsim, \sim, o,$ and $O$. In using this notation the dependence of the implicit constants on the parameters of the problem varies from line to line. If the particular dependence of these implicit constants is of importance, we often note in terms of which quantities they can be bounded.


\section{Improved inequalities}

\subsection{Improved Berezin--Li--Yau and Kr\"oger inequality}
\label{sec:improvedblyk}

In this subsection we turn our attention to one of the main ingredients in our overall strategy, namely improved versions of the Berezin--Li--Yau and Kr\"oger inequalities. Proofs of the original inequalities can be found in~\cite{Berezin,LiYau_83} and~\cite{Kr92,Laptev2}, respectively; see also \cite[Sections 3.5 and 3.6]{FrankLaptevWeidl} and, for further related inequalities, \cite{Vo}. Improved versions of the Berezin--Li--Yau and Kr\"oger inequalities have been obtained by various authors, the earliest results in this direction go back Melas \cite{Melas_03} and Freericks, Lieb, and Ueltschi \cite{Freericks_etal_02}. The latter paper concerns a discrete setting, but the underlying problem is analogous. For further results in this direction we additionally refer to \cite{Ueltschi_04,LiTang_06,Weidl,KoVuWe,GeLaWe,KoWe,LarsonPAMS,HarrellStubbe_18,LarsonJST,Harrell_etal_21,FrankLarsonPfeiffer,GanJiangLin_25}.

We define the width of a set $\Omega\subset\R^d$ by
$$
w(\Omega) := \inf_{\omega\in\Sph^{d-1}} \left( \sup_{x\in\Omega} \omega\cdot x - \inf_{x\in\Omega} \omega\cdot x \right).
$$

We emphasize that the results of this subsection do not require the set $\Omega$ to be convex. The bounds we obtain are valid for arbitrary open sets $\Omega$ with finite measure and width. For our applications the crucial feature of the bounds is that they improve as $w(\Omega)$ becomes small relative to the natural length scale $1/\sqrt{\lambda}$.

\begin{proposition}\label{prop: improved Berezin}
    There are constants $c, c'>0$ so that for any $d\in\N$, any open set $\Omega \subset \R^d$ of finite measure and finite width and for any $\lambda \geq 0$, we have
    \begin{equation*}
        \Tr(-\Delta^{\rm D}_\Omega-\lambda)_\limminus \leq L_{1,d}^{\rm sc}|\Omega|\lambda^{1+\frac{d}2}(1 -c e^{-c' w(\Omega) \sqrt{\lambda}})\,.
    \end{equation*}
    In particular, if $\Omega\subset \R^d$ is convex and bounded, then
    \begin{equation*}
        \Tr(-\Delta^{\rm D}_\Omega-\lambda)_\limminus \leq L_{1,d}^{\rm sc}|\Omega|\lambda^{1+\frac{d}2}(1 -c e^{-c'' r_{\rm in}(\Omega) \sqrt{\lambda}})\,,
    \end{equation*}
    where $c''>0$ is a constant depending only on $d$.
\end{proposition}

\begin{proposition}\label{prop: improved Kroger}
    There are constants $c, c'>0$ so that for any $d\in\N$, any open set $\Omega \subset \R^d$ of finite measure and finite width and for any $\lambda \geq 0$, we have
    \begin{equation*}
        \Tr(-\Delta^{\rm N}_\Omega-\lambda)_\limminus \geq L_{1,d}^{\rm sc}|\Omega|\lambda^{1+\frac{d}2}(1 +c e^{-c' w(\Omega) \sqrt{\lambda}})\,.
    \end{equation*}
    In particular, if $\Omega\subset \R^d$ is convex and bounded, then
    \begin{equation*}
        \Tr(-\Delta^{\rm N}_\Omega-\lambda)_\limminus \geq L_{1,d}^{\rm sc}|\Omega|\lambda^{1+\frac{d}2}(1 +c e^{-c'' r_{\rm in}(\Omega) \sqrt{\lambda}})\,,
    \end{equation*}
    where $c''>0$ is a constant depending only on $d$.
\end{proposition}

In a subsequent work with Pfeiffer \cite{FrankLarsonPfeiffer} the bounds in Propositions~\ref{prop: improved Berezin} and \ref{prop: improved Kroger} have been generalized in two directions; firstly to arbitrary open sets of finite measure, and secondly in two dimensions we allow for the presence of a constant magnetic field.

The main ingredient in the proof of these propositions is a quantitative version of the Amrein--Berthier uncertainty principle, due to Nazarov \cite{Nazarov_93}. For a more precise version, which is not needed here, however, see \cite{Kovrijkine_01}. The use of bounds of this type in connection with improved Kr\"oger inequalities has been suggested by Li and Tang~\cite{LiTang_06}. 

\begin{lemma}\label{uncertainty}
    There are constants $c,c'>0$ such that for every $d\in\N$, every $\ell,\lambda>0$ and every function $f\in L^2(\R^d)$ with $\supp f\subset [-\ell,\ell]\times\R^{d-1}$ one has
    $$
    \int_{\xi_1^2>\lambda} |\hat f(\xi)|^2\,d\xi \geq c e^{-c' \ell\sqrt\lambda} \int_{\R^d} |\hat f(\xi)|^2\,d\xi \,.
    $$  
\end{lemma}

\begin{proof}
    For $d=1$ this is a special case of the result in \cite{Nazarov_93}. For $d>1$ we fix $x'\in\R^{d-1}$ and apply the $d=1$ version of the lemma to the function $x_1\mapsto f(x_1,x')$. Integrating the resulting inequality with respect to $x'$ and applying Plancherel's theorem we obtain the claimed bound.
\end{proof}

\begin{proof}[Proof of Proposition \ref{prop: improved Berezin}]
    As the spectrum of $-\Delta_\Omega^{\rm D}$ agrees with that of $-\Delta^{\rm D}_{T\Omega+a}$ for any $T\in O(d)$ and $a\in\R^d$, we may without loss of generality assume that
    $
    \Omega \subset (-\ell,\ell) \times\R^{d-1}
    $
    where $\ell:=w(\Omega)/2$.

    Since $|\Omega|<\infty$ the spectrum of $-\Delta_\Omega^{\rm D}$ is discrete. Let $\{\lambda_k^{\rm D}(\Omega)\}_{k\geq 1}, \{u_k\}_{k\geq 1}$ be the eigenvalues of $-\Delta_\Omega^{\rm D}$ ordered in non-decreasing order and counted according to multiplicity and a corresponding sequence of $L^2$-orthonormal eigenfunctions.

   By extension by zero we consider $\{u_k\}_{k\geq 1}$ as elements of $H^1(\R^d)$. By Plancherel's identity and by the fact that $\sum_{k\geq 1}|(u_k, e^{i(\cdot)\cdot \xi})_{L^2(\Omega)}|^2 = \|e^{i(\cdot)\cdot \xi}\|_{L^2(\Omega)}^2 =|\Omega|$, we find
    \begin{align*}
        \Tr(-\Delta_\Omega^{\rm D}-\lambda)_\limminus 
        &=
        \sum_{\lambda_k^{\rm D}(\Omega)<\lambda}(\lambda-\lambda_k^{\rm D}(\Omega))\int_{\R^d}|\hat u_k(\xi)|^2\,d\xi\\
        &=
        \int_{\R^d}(\lambda-|\xi|^2)\sum_{\lambda_k^{\rm D}(\Omega)<\lambda}|\hat u_k(\xi)|^2\,d\xi\\
         &=
        (2\pi)^{-d}\int_{|\xi|^2\leq \lambda}(\lambda-|\xi|^2)\sum_{k\geq 1}|(u_k, e^{i(\cdot)\cdot\xi})_{L^2(\Omega)}|^2\,d\xi\\
        &\quad -
        \int_{|\xi|^2\leq \lambda}(\lambda-|\xi|^2)\sum_{\lambda_k^{\rm D}(\Omega)\geq \lambda}|\hat u_k(\xi)|^2\,d\xi\\
        &\quad -
        \int_{|\xi|^2>\lambda}(|\xi|^2-\lambda)\sum_{\lambda_k^{\rm D}(\Omega)<\lambda}|\hat u_k(\xi)|^2\,d\xi\\
        &=
       L_{1,d}^{\rm sc}|\Omega|\lambda^{1+\frac{d}2}\\
        &\quad -
        \int_{|\xi|^2\leq \lambda}(\lambda-|\xi|^2)\sum_{\lambda_k^{\rm D}(\Omega)\geq \lambda}|\hat u_k(\xi)|^2\,d\xi\\
        &\quad -
        \int_{|\xi|^2>\lambda}(|\xi|^2-\lambda)\sum_{\lambda_k^{\rm D}(\Omega)<\lambda}|\hat u_k(\xi)|^2\,d\xi\,.
    \end{align*}
    Note that each one of the three terms on the right side is non-negative.

    By dropping the second term and writing $(|\xi|^2-\lambda)$ with the aid of the layer-cake formula, we obtain that
    \begin{align*}
        \Tr(-\Delta_\Omega^{\rm D}-\lambda)_\limminus 
        &\leq
       L_{1,d}^{\rm sc}|\Omega|\lambda^{1+\frac{d}2}
        -
        \int_{|\xi|^2>\lambda}(|\xi|^2-\lambda)\sum_{\lambda_k^{\rm D}(\Omega)<\lambda}|\hat u_k(\xi)|^2\,d\xi\\
        &=
        L_{1,d}^{\rm sc}|\Omega|\lambda^{1+\frac{d}2}
         -
        \int_\lambda^\infty \int_{|\xi|^2>\eta}\sum_{\lambda_k^{\rm D}(\Omega)<\lambda}|\hat u_k(\xi)|^2\,d\xi \, d\eta\\
        &\leq
       L_{1,d}^{\rm sc}|\Omega|\lambda^{1+\frac{d}2}-\#\{k : \lambda_k^{\rm D}(\Omega)<\lambda\}
       \!\int_\lambda^\infty\! \Bigl(\min_{j: \lambda_j^{\rm D}(\Omega) <\lambda} \Bigl\{\int_{|\xi|^2>\eta}\!\!\!|\hat u_j(\xi)|^2\,d\xi\Bigr\}\!\Bigr) d\eta\,.
    \end{align*}
    According to Lemma \ref{uncertainty}, we have
    \begin{equation*}
        \int_{|\xi|^2>\eta} |\hat u_j(\xi)|^2\,d\xi \geq 
        \int_{\xi_1^2>\eta} |\hat u_j(\xi)|^2\,d\xi \geq 
        c e^{-c' \ell \sqrt{\eta}} \|\hat u_j\|_{L^2(\R^d)}^2= c e^{-c' \ell \sqrt{\eta}}
    \end{equation*}
    with absolute constants $c,c'>0$. Inserting this into the above bound, and using that
    \begin{equation*}
        \int_\lambda^\infty e^{-c'\ell\sqrt{\eta}}\,d\eta
        \geq \int_\lambda^{4\lambda} e^{-c'\ell\sqrt{\eta}}\,d\eta \geq 3\lambda e^{-2c'\ell\sqrt{\lambda}} \,,
    \end{equation*}
    leads to
    \begin{equation*}
        \Tr(-\Delta_\Omega^{\rm N}-\lambda)_\limminus 
        \leq
       L_{1,d}^{\rm sc}|\Omega|\lambda^{1+\frac{d}2}
       -3c\lambda e^{-2c'\ell \sqrt{\lambda}}\#\{k: \lambda_k^{\rm D}(\Omega)<\lambda\}\,.
    \end{equation*}
    The claimed bound follows if we assume that $\#\{k: \lambda_k^{\rm D}(\Omega)<\lambda\} \geq \frac{L_{1,d}^{\rm sc}}{2} |\Omega|\lambda^{\frac{d}2}$. 
    
    If instead $\#\{k: \lambda_k^{\rm D}(\Omega)<\lambda\} \leq \frac{L_{1,d}^{\rm sc}}{2} |\Omega|\lambda^{\frac{d}2}$
    , then by the Aizenman--Lieb identity we have that
    \begin{align*}
        \Tr(-\Delta_\Omega^{\rm D}-\lambda)_\limminus &= \int_0^\lambda \#\{k: \lambda_k^{\rm D}(\Omega)<\tau\} \,d\tau
        \leq 
        \int_0^\lambda \frac{L_{1,d}^{\rm sc}}{2}|\Omega|\lambda^{\frac{d}2} \,d\tau
        =
        \frac{L_{1,d}^{\rm sc}}{2}|\Omega|\lambda^{1+\frac{d}2}
    \end{align*}
    which proves again the claimed inequality.

    The bound in the convex case follows immediately via the inequality $w(\Omega)\lesssim_d r_{\rm in}(\Omega)$. The sharp version of the latter inequality (which also shows that the implied constant grows like $\sqrt d$) is due to Steinhagen \cite{Steinhagen_1922}.
\end{proof}

\begin{proof}[Proof of Proposition \ref{prop: improved Kroger}]
    As in the proof of Proposition \ref{prop: improved Berezin}, we may without loss of generality assume that
    $
    \Omega \subset (-\ell,\ell) \times\R^{d-1}
    $
    where $\ell:=w(\Omega)/2$.

    We assume first that the spectrum of $-\Delta_\Omega^{\rm N}$ is discrete. Let $\{\lambda_k^{\rm N}(\Omega)\}_{k\geq 1}, \{u_k\}_{k\geq 1}$ be the eigenvalues of $-\Delta_\Omega^{\rm N}$ ordered in non-decreasing order and counted according to multiplicity and a corresponding sequence of $L^2$-orthonormal eigenfunctions.

   Since $\{u_k\}_{k\geq 1}$ is an ON-basis for $L^2(\Omega)$ and $x\mapsto e^{i x\cdot \xi}$ belongs to the domain of $\sqrt{-\Delta_\Omega^{\rm N}+1}$ we compute that
   \begin{align*}
    (2\pi)^d\sum_{k\geq 1}|\hat u_k(\xi)|^2 &=\sum_{k\geq 1} |(u_k, e^{i(\cdot)\cdot\xi})_{L^2(\Omega)}|^2 = \| e^{i(\cdot)\cdot\xi}\|_{L^2(\Omega)}^2 = |\Omega|\,,\\
    (2\pi)^d\sum_{k\geq 1}\lambda_k^{\rm N}(\Omega)|\hat u_k(\xi)|^2&=\sum_{k\geq 1}\lambda_k^{\rm N}(\Omega) |(u_k, e^{i(\cdot)\cdot\xi})_{L^2(\Omega)}|^2 = \|\nabla e^{i(\cdot)\cdot\xi}\|_{L^2(\Omega)}^2 = |\Omega||\xi|^2\,.
   \end{align*}
   When combined with Plancherel's identity (the $u_k$'s being extended by zero), we find
    \begin{align*}
        \Tr(-\Delta_\Omega^{\rm N}-\lambda)_\limminus 
         &=
        \int_{|\xi|^2\leq \lambda}\sum_{k\geq 1}(\lambda-\lambda_k^{\rm N}(\Omega))|\hat u_k(\xi)|^2\,d\xi\\
        &\quad +
        \int_{|\xi|^2\leq \lambda}\sum_{\lambda_k^{\rm N}(\Omega)\geq \lambda}(\lambda_k^{\rm N}(\Omega)-\lambda)|\hat u_k(\xi)|^2\,d\xi\\
        &\quad +
        \int_{|\xi|^2>\lambda}\sum_{\lambda_k^{\rm N}(\Omega)<\lambda}(\lambda-\lambda_k^{\rm N}(\Omega))|\hat u_k(\xi)|^2\,d\xi\\
        &=
       L_{1,d}^{\rm sc}|\Omega|\lambda^{1+\frac{d}2}\\
        &\quad  +
        \int_{|\xi|^2\leq \lambda}\sum_{\lambda_k^{\rm N}(\Omega)\geq \lambda}(\lambda_k^{\rm N}(\Omega)-\lambda)|\hat u_k(\xi)|^2\,d\xi\\
        &\quad +
        \int_{|\xi|^2>\lambda}\sum_{\lambda_k^{\rm N}(\Omega)<\lambda}(\lambda-\lambda_k^{\rm N}(\Omega))|\hat u_k(\xi)|^2\,d\xi\,.
    \end{align*}
    Note that each of the three terms on the right side is non-negative.

    By dropping the second term and using Kr\"oger's inequality \cite{Kr92,Laptev2} we obtain that
    \begin{align*}
        \Tr(-\Delta_\Omega^{\rm N}-\lambda)_\limminus 
        &\geq
       L_{1,d}^{\rm sc}|\Omega|\lambda^{1+\frac{d}2}
         +
        \int_{|\xi|^2>\lambda}\sum_{\lambda_k^{\rm N}(\Omega)<\lambda}(\lambda-\lambda_k^{\rm N}(\Omega))|\hat u_k(\xi)|^2\,d\xi\\
        &\geq
       L_{1,d}^{\rm sc}|\Omega|\lambda^{1+\frac{d}2}
         +
        \min_{j: \lambda_j^{\rm N}(\Omega) <\lambda}\Bigl\{\int_{|\xi|^2>\lambda}|\hat u_j(\xi)|^2\,d\xi\Bigr\} \sum_{\lambda_k^{\rm N}(\Omega)<\lambda}(\lambda-\lambda_k^{\rm N}(\Omega))\\
        &\geq
       L_{1,d}^{\rm sc}|\Omega|\lambda^{1+\frac{d}2}\Bigl(1+\min_{j: \lambda_j^{\rm N}(\Omega) <\lambda} \Bigl\{\int_{|\xi|^2>\lambda}|\hat u_j(\xi)|^2\,d\xi\Bigr\}\Bigr) \,.
    \end{align*}
    According to Lemma \ref{uncertainty}, we have
    \begin{equation*}
        \int_{|\xi|^2>\lambda} |\hat u_j(\xi)|^2\,d\xi \geq 
        \int_{\xi_1^2>\lambda} |\hat u_j(\xi)|^2\,d\xi \geq 
        c e^{-c' \ell \sqrt{\lambda}} \|\hat u_j\|_{L^2(\R^d)}^2= c e^{-c' \ell \sqrt{\lambda}}
    \end{equation*}
    with absolute constants $c,c'>0$. Inserting this into the above bound, we obtain the first assertion of the proposition under the assumption of discreteness of the spectrum.

    The case where $\Lambda:=\inf \sigma_{\rm ess}(-\Delta_\Omega^{\rm N})<\infty$ follows by essentially the same argument, except that we need the spectral measure. Note that we may assume that $\lambda\leq\Lambda$, for otherwise the inequality is trivially correct, the left side being infinite. We now take $\{u_k\}_{k\geq 1}$ to be an orthonormal system of eigenfunctions that is complete in the subspace $\ran\1(-\Delta_\Omega^{\rm N}<\Lambda)$. Then we add and subtract the quantity
    $$
    (2\pi)^{-d} \int_{|\xi|^2\leq\lambda} \left( \int_{[\Lambda,\infty)} (\lambda - \mu) \,d( e^{i(\cdot)\cdot\xi}, \1(-\Delta_\Omega^{\rm N} <\mu) \, e^{i(\cdot)\cdot\xi} ) \right) d\xi
    $$
    to the above expression for $\Tr(-\Delta_\Omega^{\rm N}-\lambda)_\limminus$. Note that this quantity is non-positive since $\lambda - \mu\leq 0$ for $\lambda\leq\Lambda\leq\mu$. Meanwhile, similarly as before,
    \begin{align*}
        & (2\pi)^{-d}\int_{|\xi|^2\leq \lambda}\sum_{k\geq 1}(\lambda-\lambda_k^{\rm N}(\Omega))|(u_k, e^{i(\cdot)\cdot\xi})_{L^2(\Omega)}|^2\,d\xi\\
        &
        + (2\pi)^{-d} \int_{|\xi|^2\leq\lambda} \left( \int_{[\Lambda,\infty)} (\lambda - \mu) \,d( e^{i(\cdot)\cdot\xi}, \1(-\Delta_\Omega^{\rm N} <\mu) \, e^{i(\cdot)\cdot\xi} ) \right) d\xi \\
        &\qquad  = (2\pi)^{-d} \int_{|\xi|^2\leq\lambda} \left( \int_{[0,\infty)} (\lambda - \mu) \,d( e^{i(\cdot)\cdot\xi}, \1(-\Delta_\Omega^{\rm N} <\mu) \, e^{i(\cdot)\cdot\xi} ) \right) d\xi \\
        &\qquad  = (2\pi)^{-d} \int_{|\xi|^2\leq\lambda} \left( \lambda \|e^{i(\cdot)\cdot\xi} \|_{L^2(\Omega)}^2 - \|\nabla e^{i(\cdot)\cdot\xi} \|_{L^2(\Omega)}^2 \right) d\xi \\
        &\qquad =
       L_{1,d}^{\rm sc}|\Omega|\lambda^{1+\frac{d}2} \,.
    \end{align*}
    The rest of the argument is the same as in the case of discrete spectrum.    

    The proof of the bound in the convex case can now be completed as in the proof of Proposition~\ref{prop: improved Kroger}
\end{proof}


\subsection{A priori bounds in the Neumann case}

We begin by recalling the following bound, which is due to Li and Yau~\cite[Theorem 5.3]{LiYau_86}:
\begin{lemma}\label{lem: upper bound N counting}
    Fix $d \geq 1$. There exists a constant $C_{d}$ such that for all open, bounded, and convex $\Omega\subset \R^d$ and $\lambda \geq 0$ it holds that
	\begin{equation*}
	 	\Tr(-\Delta_\Omega^{\rm N} -\lambda)_\limminus^0 \leq
		 C_{d}\diam(\Omega)^d\lambda^{\frac{d}{2}}+1\,.
	 \end{equation*} 
\end{lemma}

The next bound is a lower bound on Riesz means of the Neumann Laplacian with the important property that it incorporates that the trace becomes large if the convex set degenerates.

\begin{lemma}\label{lem: small energy improved Kroger}
	Fix $d \geq 1, \gamma \geq 0$. There exists a positive constant $C_{\gamma,d}$ such that for all $\Omega\subset \R^d$ open, bounded, and convex and $\lambda \geq 0$ it holds that
	\begin{equation*}
	 	\Tr(-\Delta_\Omega^{\rm N} -\lambda)_\limminus^\gamma \geq
		 C_{\gamma,d}\frac{|\Omega|}{r_{\rm in}(\Omega)}\lambda^{\gamma+ \frac{d-1}{2}}\,.
	 \end{equation*} 
\end{lemma}

We note that this inequality is a direct consequence of Kr\"oger's inequality \cite{Kr92,Laptev2} if $\gamma \geq 1$ and $r_{\rm in}(\Omega)\sqrt{\lambda} \gtrsim 1$, but the point is that it improves Kr\"oger's result in the regime when $r_{\rm in}(\Omega)\sqrt{\lambda}$ becomes small. In fact, this inequality is better than that corresponding to P\'olya's conjecture if $r_{\rm in}(\Omega)\sqrt{\lambda} < \frac{C_{\gamma,d}}{L_{\gamma, d}^{\rm sc}}$.

\begin{proof}
If $d=1$ then the inequality is trivially true as $|\Omega|=2r_{\rm in}(\Omega)$ and $\Tr(-\Delta_\Omega^{\rm N}-\lambda)_\limminus^\gamma \geq \lambda^\gamma$ since $\lambda_1^{\rm N}(\Omega)=0$. Therefore, without loss of generality we in the remainder of the proof assume that $d\geq 2$.

Let $\{l_j\}_{j=1}^d$ be the semiaxes of the John ellipsoid of $\Omega$. After a translation and an orthogonal map (which leaves the spectrum of $-\Delta_\Omega^{\rm N}$ unchanged) we may by John's theorem~\cite[Theorem 10.12.2]{Schneider_book14} assume that
$
	E \subseteq \Omega \subseteq dE,
$
where 
\begin{equation*}
	E := \Bigl\{(x_1, \ldots, x_d) \in \R^d: \sum_{j=1}^d\frac{x_j^2}{l_j^2}<1\Bigr\}\,.
\end{equation*}
Note that this implies that
$
	Q \subset \Omega \subset d^{3/2}Q
$
where
\begin{equation*}
	Q := \{(x_1, \ldots, x_d)\in \R^d: |x_j|< l_j/\sqrt{d}, j=1, \ldots, d\}.
\end{equation*}

Let $\{\mu_j\}_{j\geq 1}$ be the eigenvalues of $-\Delta_Q^{\rm N}$, in non-decreasing order and with multiplicity taken into account, and let $\{\psi_j\}_{j\geq 1}$ the corresponding eigenfunctions. Note that $\psi_j$ can be explicitly written in terms of cosines and can be extended to all of $\R^d$ as a lattice-periodic function in the sense that
\begin{equation*}
	\psi_j(x) = \psi_j\Bigl(x+ \sum_{i=1}^d l_i v_i e_i\Bigr) \quad \mbox{for all }v_i \in \Z\,.
\end{equation*}
Going forward we will not distinguish between $\psi_j$ and their periodic extensions.

Note that the restrictions to $\Omega$ of $\psi_j$ are linearly independent; if a non-trivial linear combination of these functions were to vanish in $\Omega$ it would vanish in $Q$ and contradict the linear independence of the functions in $L^2(Q)$. 

Consequently, the variational principle implies that for any $j \geq 1$
\begin{equation*}
	\lambda_j^{\rm N}(\Omega) \leq \max_{u \in \mathrm{span}\{\psi_{l}: l=1, \ldots, j\}} \frac{\int_\Omega |\nabla u(x)|^2\,dx}{\int_\Omega |u(x)|^2\,dx}\,.
\end{equation*}
Let $M \in 2\N+1$ be the unique odd integer so that $d^{3/2}<M \leq d^{3/2}+2$. By the inclusion $Q \subset \Omega \subset MQ$ and the periodicity of each of the $\psi_{k}$ it follows that if $u \in \mathrm{span}\{\psi_{l}: l=1, \ldots, j\}$ we have
\begin{equation*}
	\frac{\int_\Omega |\nabla u(x)|^2\,dx}{\int_\Omega |u(x)|^2\,dx} \leq\frac{\int_{MQ} |\nabla u(x)|^2\,dx}{\int_Q |u(x)|^2\,dx} = \frac{M^d\int_Q |\nabla u(x)|^2\,dx}{\int_Q |u(x)|^2\,dx} \leq M^d\mu_j\,.
\end{equation*}

Consequently, for any $\gamma\geq 0, \lambda \geq 0$ we observe that

\begin{align*}
	\Tr(-\Delta_\Omega^{\rm N}-\lambda)_\limminus^\gamma 
	&\geq 
		\sum_{j: M^d\mu_j<\lambda}(\lambda-M^d\mu_j)^\gamma\\
	&=
		\sum_{\bar k\in (\N\cup \{0\})^d}\Bigl(\lambda - M^d\sum_{i=1}^d \frac{\pi^2dk_i^2}{4l_i^2}\Bigr)^\gamma\\
	&\geq
		\sum_{\bar k\in (\N\cup \{0\})^{d-1}}\Bigl(\lambda - M^d\sum_{i=2}^d \frac{\pi^2dk_i^2}{4l_i^2}\Bigr)^\gamma\\
		&= M^{d\gamma}\Tr(-\Delta_{\tilde Q}^{\rm N}-\lambda/M^d)^\gamma_\limminus \,,
\end{align*}
where
$
\tilde Q := \prod_{i=2}^d (-l_i/\sqrt{d}, l_i/\sqrt{d})\subset \R^{d-1}
$ 
is the projection of $Q$ onto the last $d-1$ coordinates. Since P\'olya's conjecture is valid for cuboids in all dimensions and since $\Haus^{d-1}(\tilde Q) = \sqrt{d}|Q|/(2l_1) \gtrsim_d |\Omega|r_{\rm in}(\Omega)^{-1}$, we deduce that there is a constant $C_{\gamma, d}$ so that
\begin{align*}
	\Tr(-\Delta_\Omega^{\rm N}-\lambda)_\limminus^\gamma 
	&\geq
		 C_{\gamma,d}\frac{|\Omega|}{r_{\rm in}(\Omega)}\lambda^{\gamma+ \frac{d-1}{2}}\,.
\end{align*}
This completes the proof.
\end{proof}

We end this subsection with an application of Lemma \ref{lem: small energy improved Kroger} to shape optimization problems. We recall that for any $\gamma\geq 0$ and $\lambda\geq 0$ the extrema $M_\gamma^\sharp(\lambda)$ were defined in the introduction. In the case $\sharp=$ D the existence of an extremizer is shown in \cite{LarsonJST}. Here we use Lemma \ref{lem: small energy improved Kroger} to derive the corresponding result for $\sharp=$ N.

\begin{lemma}\label{exshapeoptneumann}
    Let $\gamma\geq 0$ and $\lambda\geq 0$. Then there is an $\Omega_*\in\mathcal C_d$ with $|\Omega_*|=1$ that attains the infimum defining $M_\gamma^{\rm N}(\lambda)$.
\end{lemma}

\begin{proof}
    We may assume that $\lambda>0$, for otherwise the statement is trivial. Let $\{\Omega_j\}_{j\geq 1}\subset\mathcal C_d$ with $|\Omega_j|=1$ for all $j$ be a minimizing sequence for $M_\gamma^{\rm N}(\lambda)$. According to Lemma~\ref{lem: small energy improved Kroger} we find
    $$
    \liminf_{j\to\infty} r_{\rm in}(\Omega_j) \geq C_{\gamma,d} \frac{\lambda^{\gamma+\frac{d-1}2}}{M_\gamma^{\rm N}(\lambda)} >0 \,.
    $$
    We can now argue as in the proof of \cite[Lemma 3.1]{LarsonJST}. By $|\Omega_j|=1, r_{\rm in}(\Omega_j)\geq \delta>0$ it follows from~\eqref{eq: inradius bound} that the sequence $\{\Omega_j\}_{j\geq 1}$ is uniformly bounded. Therefore, the Blaschke selection theorem implies that we can extract a subsequence that converges in the Hausdorff sense up to translations to a bounded, open, convex set $\Omega_*$ with $|\Omega_*|=1$. Recall that the Neumann eigenvalues are continuous with respect to Hausdorff convergence of convex sets (see \cite{Ross_04} or Proposition~\ref{prop: DN continuity of eigenvalues} below). Therefore, using dominated convergence and the bound in Lemma~\ref{lem: upper bound N counting} one deduces that $\liminf_{j\to \infty}\Tr(-\Delta_{\Omega_j}^{\rm N}-\lambda)_\limminus^\gamma \geq \Tr(-\Delta_{\Omega_*}^{\rm N}-\lambda)_\limminus^\gamma$. Since $\{\Omega_j\}_{j\geq 1}$ was a minimizing sequence $\Omega_*$ is a minimizer for the variational problem defining $M_\gamma^{\rm N}(\lambda)$.
\end{proof}


\section{Extrapolating semiclassical inequalities for Riesz means}

In this section we discuss an idea that will be used frequently in our arguments.

\begin{proposition}\label{prop: general extrapolation}
    Let $\Omega \subset \R^d$ be an open set and $\gamma_1>\gamma_0\geq 0$. 
    \begin{enumerate}[label=(\arabic*)]
    \item\label{itm: extrap D} If there exist $\Lambda_0>0, \Lambda_1>\Lambda_0, c<1$, such that
    \begin{equation*}
        \Tr(-\Delta_\Omega^{\rm D}-\lambda)_\limminus^{\gamma'} \leq L_{\gamma',d}^{\rm sc}|\Omega|\lambda^{\gamma'+ \frac{d}{2}}\quad \mbox{for all }\lambda \geq \Lambda_0, \gamma' \in [\gamma_0, \gamma_1]\,,
    \end{equation*}
    and
    \begin{equation*}
        \Tr(-\Delta_\Omega^{\rm D}-\lambda)_\limminus^{\gamma_1}\leq c L_{\gamma_1, d}^{\rm sc}|\Omega| \lambda^{\gamma_1+\frac{d}2}\,, \quad \mbox{for all }\lambda \leq \Lambda_1\,,
    \end{equation*}
    then there exists $\gamma\in [\gamma_0, \gamma_1)$ depending only on $d, \gamma_0, \gamma_1, c$ and an upper bound on $\frac{\Lambda_1}{\Lambda_1-\Lambda_0}$ so that
    \begin{equation*}
        \Tr(-\Delta_\Omega^{\rm D}-\lambda)_\limminus^{\gamma}\leq  L_{\gamma, d}^{\rm sc}|\Omega| \lambda^{\gamma+\frac{d}2}\,, \quad \mbox{for all }\lambda \geq 0\,.
    \end{equation*}

    \item\label{itm: extrap N} If there exist $\Lambda_0>0, c>1$, such that
    \begin{equation*}
       \Tr(-\Delta_\Omega^{\rm N}-\lambda)_\limminus^{\gamma'} \geq L_{\gamma',d}^{\rm sc}|\Omega|\lambda^{\gamma'+ \frac{d}{2}}\quad \mbox{for all }\lambda \geq \Lambda_0, \gamma' \in [\gamma_0, \gamma_1]
    \end{equation*}
    and
    \begin{equation*}
        \Tr(-\Delta_\Omega^{\rm N}-\lambda)_\limminus^{\gamma_1}\geq c L_{\gamma_1, d}^{\rm sc}|\Omega| \lambda^{\gamma_1+\frac{d}2}\,, \quad \mbox{for all }\lambda \leq \Lambda_0\,,
    \end{equation*}
    then there exists $\gamma \in [\gamma_0, \gamma_1)$ depending only on $d, \gamma_0, \gamma_1, c$ so that
    \begin{equation*}
        \Tr(-\Delta_\Omega^{\rm N}-\lambda)_\limminus^{\gamma}\geq  L_{\gamma, d}^{\rm sc}|\Omega| \lambda^{\gamma+\frac{d}2}\,, \quad \mbox{for all }\lambda \geq 0\,.
    \end{equation*}
    \end{enumerate}
\end{proposition}

Our proof of Proposition~\ref{prop: general extrapolation} relies on an idea that in a special case appears in Laptev~\cite{Laptev2} and generally in \cite{Frank_etal_09}.

\begin{lemma}\label{lem: Laptev extrapolation}
	Fix $0< \gamma' <\gamma$. For any $\lambda, \mu, \delta>0$ it holds that
	\begin{equation*}
		\lambda^{\gamma'-\gamma}(\lambda-\mu)_\limplus^\gamma\leq (\lambda-\mu)_\limplus^{\gamma'} \leq \delta^{\gamma'-\gamma}\frac{{\gamma'}^{\gamma'} (\gamma-\gamma')^{\gamma-\gamma'}}{{\gamma}^{\gamma}}  (\lambda-\mu+\delta)_\limplus^{\gamma}\,.
	\end{equation*}
	and
	\begin{equation*}
		\lambda^{-\gamma}(\lambda-\mu)_\limplus^\gamma \leq \1_{[\mu, \infty)}(\lambda) \leq \delta^{-\gamma} (\lambda-\mu+\delta)_\limplus^{\gamma}\,.
	\end{equation*}
\end{lemma}

The proof of the lemma is elementary and is omitted. We recall that Laptev~\cite{Laptev2} used the second pair of inequalities (with $\gamma=1$) to deduce bounds for counting functions from the Berezin--Li--Yau and Kr\"oger inequalities. This was generalized to general $\gamma\in[0,1)$ (and still $\gamma'=1$) in \cite{Frank_etal_09} and, remarkably, the resulting eigenvalue inequalities were shown to be sharp in the presence of a homogeneous magnetic field.

The consequences of the inequalities in Lemma~\ref{lem: Laptev extrapolation} that we shall require are recorded in the following two lemmas. 

\begin{lemma}\label{lem: extrapolating improved inequality small lambda Dir}
	Fix $d\geq 2, \gamma>0$ and let $\Omega \subset \R^d$ be an open set with $|\Omega|<\infty$. Assume that there exist $c>0, \Lambda \in (0, \infty]$  so that
	\begin{equation*}
		\Tr(-\Delta_\Omega^{\rm D} - \lambda)_\limminus^{\gamma} \leq c L_{\gamma, d}^{\rm sc} |\Omega| \lambda^{\gamma+\frac{d}2}\qquad \mbox{for all } \lambda \leq \Lambda\,.
	\end{equation*}	
	Then, for any $\gamma' \in [0, \gamma]$
	\begin{equation*}
		\Tr(-\Delta_\Omega^{\rm D} - \lambda)_\limminus^{\gamma'} \leq c f^{\rm D}_{\gamma}(\gamma') L^{\rm sc}_{\gamma', d} |\Omega| \lambda^{\gamma'+\frac{d}2} \qquad \mbox{for all } \lambda \leq \frac{\gamma'+\frac{d}2}{\gamma+\frac{d}2}\Lambda\,,
	\end{equation*} 
	with $f^{\rm D}_{\gamma}\colon [0, \gamma]\to [1, \infty)$ the decreasing function defined by
	\begin{equation*}
		f^{\rm D}_{\gamma}(\gamma') = \frac{(\gamma+\frac{d}2)^{\gamma+\frac{d}2}}{(\gamma'+\frac{d}2)^{\gamma'+\frac{d}2}}\frac{\Gamma(1+\gamma'+\frac{d}2)\Gamma(1+\gamma)}{\Gamma(1+\gamma+\frac{d}2)\Gamma(1+\gamma')}\,.
	\end{equation*}
\end{lemma}

\begin{lemma}\label{lem: extrapolating improved inequality small lambda Neu}
	Fix $d\geq 2, \gamma>0$ and let $\Omega \subset \R^d$ be an open set with $|\Omega|<\infty$. Assume that there exist $c>0, \Lambda \in (0, \infty]$  so that
	\begin{equation*}
		\Tr(-\Delta_\Omega^{\rm N} - \lambda)_\limminus^{\gamma} \geq c L_{\gamma, d}^{\rm sc} |\Omega| \lambda^{\gamma+\frac{d}2}\qquad \mbox{for all } \lambda \leq \Lambda\,.
	\end{equation*}
	
	Then, for any $\gamma' \in [0, \gamma]$
	\begin{equation*}
		\Tr(-\Delta_\Omega^{\rm N} - \lambda)_\limminus^{\gamma'} \geq c f^{\rm N}_{\gamma}(\gamma') L^{\rm sc}_{\gamma', d} |\Omega| \lambda^{\gamma'+\frac{d}2} \qquad \mbox{for all } \lambda \leq \Lambda\,, 
	\end{equation*} 
	with $f^{\rm N}_{\gamma}\colon [0, \gamma]\to (0, 1]$ the increasing function defined by
	\begin{equation*}
		f^{\rm N}_{\gamma}(\gamma') = \frac{L_{\gamma,d}^{\rm sc}}{L_{\gamma',d}^{\rm sc}} = \frac{\Gamma(1+\gamma'+\frac{d}2)\Gamma(1+\gamma)}{\Gamma(1+\gamma+\frac{d}2)\Gamma(1+\gamma')}\,.
	\end{equation*}
\end{lemma}

\begin{remark}
    We note that the functions $f^{\sharp}$ appearing in the above lemmas inherit the semigroup structure of Riesz means; if $0\leq \gamma''<\gamma'<\gamma$ then
    \begin{equation*}
        f^{\sharp}_{\gamma}(\gamma'')=f^{\sharp}_{\gamma}(\gamma')f^{\sharp}_{\gamma'}(\gamma'')\,.
    \end{equation*}
\end{remark}

\begin{proof}[Proof of Lemma~\ref{lem: extrapolating improved inequality small lambda Dir}]
	We write out the proof of the proposition under the assumption $\gamma' >0$. The argument extends to the case $\gamma'=0$ by utilizing the second inequality in Lemma~\ref{lem: Laptev extrapolation} instead of the first or by, after proving the statement for $\gamma'>0$, using continuity of the involved quantities and taking the limit $\gamma' \to 0$.  

    We use the assumed bound for $\gamma>\gamma'$ and Laptev's extrapolation argument. For any $\tau>0$ by Lemma~\ref{lem: Laptev extrapolation} with $\delta = \tau \lambda$ we deduce that
	\begin{align*}
		\Tr(-\Delta_\Omega^{\rm D}-\lambda)_\limminus^{\gamma'}  
        = \sum_{k\geq 1}(\lambda- \lambda_k^{\rm D}(\Omega))_\limplus^{\gamma'}
		\leq 
		\frac{{\gamma'}^{\gamma'}}{{\gamma}^{\gamma}}(\gamma-\gamma')^{\gamma-\gamma'} (\tau \lambda)^{\gamma'-\gamma}\Tr(-\Delta_\Omega^{\rm D}-\lambda(1+\tau))_\limminus^{\gamma}\,.
	\end{align*}
	If $(1+\tau) \lambda \leq \Lambda$, then by the assumed bound for $\Tr(-\Delta_\Omega^{\rm D}-\lambda)_\limminus^\gamma$
	\begin{align*}
		\Tr(-\Delta_\Omega^{\rm D}-\lambda)_\limminus^{\gamma'} & \leq
		\frac{{\gamma'}^{\gamma'} (\gamma-\gamma')^{\gamma-\gamma'} }{{\gamma}^{\gamma}}(\tau \lambda)^{\gamma'-\gamma}cL^{\rm sc}_{\gamma,d}|\Omega|(1+\tau)^{\gamma+\frac{d}2}\lambda^{\gamma+\frac{d}2}\\
		& =
		\frac{{\gamma'}^{\gamma'} (\gamma-\gamma')^{\gamma-\gamma'} }{{\gamma}^{\gamma}}\frac{(1+\tau)^{\gamma+\frac{d}2}}{\tau^{\gamma-\gamma'}}\frac{\Gamma(1+\gamma'+\frac{d}2)\Gamma(1+\gamma)}{\Gamma(1+\gamma+\frac{d}2)\Gamma(1+\gamma')}cL^{\rm sc}_{\gamma',d}|\Omega|\lambda^{\gamma'+\frac{d}2}\,.
	\end{align*}

	We minimize with respect to $\tau$, setting $\tau = \frac{2(\gamma-\gamma')}{d+2\gamma'}$, which leads to the bound
	\begin{align*}
		\Tr(-\Delta_\Omega^{\rm D}-\lambda)_\limminus^{\gamma'} & \leq
		cf^{\rm D}_{\gamma}(\gamma')L^{\rm sc}_{\gamma',d}|\Omega|\lambda^{\gamma'+\frac{d}2}
	\end{align*}
	for all $\lambda \leq \frac{\gamma'+\frac{d}2}{\gamma+\frac{d}2}\Lambda$.

    As it is clear that $f_\gamma^{\rm D}(\gamma)=1$ it remains to prove that $f_\gamma^{\rm D}$ is decreasing. 
    Define
    $$
       g\colon [0, \infty) \to \R \,,\quad s \mapsto \frac{(s+\frac{d}{2})^{s+ \frac{d}{2}} \ \Gamma(1+s)}{\Gamma(1+s+ \frac{d}2)}
    $$
    so that $f_\gamma^{\rm D}(\gamma') = \frac{g(\gamma)}{g(\gamma')}$. As $(0, \infty)\ni s\mapsto\Gamma(s)$ is smooth and positive it follows that $g$ is smooth and positive. Therefore, the function $f^{\rm D}_\gamma$ is decreasing if and only if $g$ is increasing. We shall prove this by arguing that $s \mapsto \ln(g(s))$ is increasing. 
    
    Let $\psi$ denote the digamma function defined by $\psi(s) = \frac{\Gamma'(s)}{\Gamma(s)}$ and recall that $\psi$ is increasing and strictly concave (this follows by differentiating the integral representation in~\cite[Eq.~5.9.16]{NIST}).

    In terms of $\psi$ we find that
    \begin{equation*}
        \frac{d}{ds}\ln(g(s)) = 1+ \ln\Bigl(s+ \frac{d}{2}\Bigr) + \psi(1+s)-\psi\Bigl(1+s+\frac{d}{2}\Bigr)\,,
    \end{equation*}
    and
    \begin{equation*}
        \frac{d^2}{ds^2}\ln(g(s)) = \frac{1}{s+ \frac{d}{2}} + \psi'(1+s)-\psi'\Bigl(1+s+\frac{d}{2}\Bigr)\,.
    \end{equation*}
    As $\psi$ is concave $\psi'$ is increasing and thus $\frac{d^2}{ds^2}\ln(g(s))>0$. Consequently, it follows that $g$ is increasing if we show that $\frac{d}{ds}\ln(g(s))$ is non-negative at $0$, that is, if
    \begin{equation}\label{eq: derivative sign}
        1+ \ln\Bigl(\frac{d}{2}\Bigr) + \psi(1)-\psi\Bigr(1+\frac{d}{2}\Bigr)\geq 0\,.
    \end{equation}

    To prove \eqref{eq: derivative sign}, we distinguish whether $d$ is even or odd. When $d=2k$ with $k\in \N$, then according to \cite[Eq. 5.4.14]{NIST}
    $$
        \psi\Bigr(1+\frac{d}{2}\Bigr) - \psi(1) = \sum_{j=1}^k \frac 1j \leq 1+ \sum_{j=2}^k \int_{j-1}^k \frac{dx}{x} = 1+ \ln k = 1+ \ln\Bigl(\frac d2\Bigr) \,,
    $$
    which is equivalent to \eqref{eq: derivative sign}.
    If instead $d=2k-1$ with $k\in \N$, then by \cite[Eq. 5.4.15]{NIST}
    $$
        \psi\Bigr(1+\frac{d}{2}\Bigr) - \psi(1) = -2\ln 2 + \sum_{j=1}^k \frac 2{2j-1}  \,.
    $$
    By convexity we have $\int_{2j-2}^{2j}x^{-1} dx \geq 2(2j-1)^{-1}$ for $j\geq 2$. Thus,
    \begin{align*}
        \psi\Bigr(1+\frac{d}{2}\Bigr) - \psi(1) &\leq - 2 \ln 2 + 2 + \sum_{j=2}^k \int_{2j-2}^{2j} \frac{dx}{x}\\
        &= - 2\ln 2 + 2 + \ln(k)\\
        &= 1+\ln\Bigl(k-\frac{1}2\Bigr)+\Bigl(1-\ln\Bigl(4-\frac{2}k\Bigr)\Bigr)\\
        &\leq 1+\ln\Bigl(k-\frac12\Bigr) \,, 
    \end{align*}
    as $\ln(4-\frac{2}k)\geq \ln(3) >\ln(e)=1$. Since $k -\frac{1}{2}= \frac{d}{2}$ this proves \eqref{eq: derivative sign}, and thus completes the proof of Lemma~\ref{lem: extrapolating improved inequality small lambda Dir}.
 \end{proof}

\begin{proof}[Proof of Lemma~\ref{lem: extrapolating improved inequality small lambda Neu}]
	We use the assumed bound for $\Tr(-\Delta_\Omega^{\rm N}-\lambda)_\limminus^\gamma$ and the lower bounds in Lemma~\ref{lem: Laptev extrapolation}. By Lemma~\ref{lem: Laptev extrapolation} we deduce that
	\begin{align*}
		\Tr(-\Delta_\Omega^{\rm N}-\lambda)_\limminus^{\gamma'} & = \sum_{k\geq 1}(\lambda- \lambda_k^{\rm N}(\Omega))_\limplus^{\gamma'}
		\geq 
		\lambda^{\gamma'-\gamma}\Tr(-\Delta_\Omega^{\rm N}-\lambda)_\limminus^{\gamma}\,.
	\end{align*}
	If $\lambda \leq \Lambda$ then by the assumed bound for $\gamma$
	\begin{align*}
		\Tr(-\Delta_\Omega^{\rm N}-\lambda)_\limminus^{\gamma'} 
		 \geq
			c L_{\gamma,d}^{\rm sc}|\Omega|\lambda^{\gamma'+ \frac{d}{2}}
		=
			c f^{\rm N}_\gamma(\gamma')L_{\gamma',d}^{\rm sc}|\Omega|\lambda^{\gamma'+ \frac{d}{2}} \,.
	\end{align*}

    Clearly $f^{\rm N}_\gamma(\gamma)=1$ and so to complete the proof it only remains to prove that $f^{\rm N}_\gamma$ is increasing. Since $f^{\rm N}_\gamma(\gamma') = \frac{L_{\gamma,d}^{\rm sc}}{L_{\gamma',d}^{\rm sc}}$ the claimed monotonicity of $f^{\rm N}_\gamma$ follows if we show that $\gamma \mapsto L_{\gamma,d}^{\rm sc}$ is decreasing. That this is the case can be seen by using the integral representation
    \begin{equation*}
        L_{\gamma, d}^{\rm sc} = (2\pi)^{-d}\int_{\R^d}(1-|\xi|^2)_\limplus^\gamma \,d\xi
    \end{equation*}
    and noting that $\gamma \mapsto (1-|\xi|^2)_\limminus^\gamma$ is decreasing for each fixed $\xi\in \R^d$ with $|\xi|<1$ and zero otherwise. This completes the proof.
\end{proof}

\begin{proof}[Proof of Proposition~\ref{prop: general extrapolation}]
{\it Case 1: Dirichlet boundary conditions.} By the first assumption we need to prove that for some $\gamma\in [\gamma_0, \gamma_1]$ we have that
\begin{equation}\label{eq: goal extrapolation Dir}
    \Tr(-\Delta_\Omega^{\rm D}-\lambda)_\limminus^\gamma \leq L_{\gamma, d}^{\rm sc} |\Omega|\lambda^{\gamma + \frac{d}2} \quad \mbox{for all }\lambda < \Lambda_0\,.
\end{equation}

By Lemma~\ref{lem: extrapolating improved inequality small lambda Dir} and the second assumption we have for any $\gamma' \leq \gamma_1$ and all $\lambda \leq \frac{\gamma'+ \frac{d}{2}}{\gamma_1 + \frac{d}{2}}\Lambda_1$ that
\begin{equation*}
    \Tr(-\Delta_\Omega^{\rm D}-\lambda)_\limminus^{\gamma'} \leq c f_{\gamma_1}^{\rm D}(\gamma')  L_{\gamma', d}^{\rm sc}|\Omega|\lambda^{\gamma'+ \frac{d}2}\,.
\end{equation*}
Since $f_{\gamma_1}^{\rm D}$ is continuous and decreasing with $\lim_{\gamma'\to \gamma_1}f_{\gamma_1}^{\rm D}(\gamma') =1$ and since $c<1$ and $\Lambda_1>\Lambda_0$, we find that \eqref{eq: goal extrapolation Dir} holds for all $\gamma$ satisfying
\begin{equation*}
    \max\Bigl\{(f_{\gamma_1}^{\rm D})^{-1}\Bigl(\frac{1}{c}\Bigr), \gamma_1 -\frac{\Lambda_1-\Lambda_0}{\Lambda_1}\Bigl(\gamma_1+\frac{d}2\Bigr)\Bigr\}\leq \gamma \leq \gamma_1\,.
\end{equation*}

\medskip
{\it Case 2: Neumann boundary conditions.} By the first assumption we need to prove that for some $\gamma\in [\gamma_0, \gamma_1]$ we have that
\begin{equation}\label{eq: goal extrapolation Neu}
    \Tr(-\Delta_\Omega^{\rm N}-\lambda)_\limminus^\gamma \leq L_{\gamma, d}^{\rm sc} |\Omega|\lambda^{\gamma + \frac{d}2} \quad \mbox{for all }\lambda < \Lambda_0\,.
\end{equation}

By Lemma~\ref{lem: extrapolating improved inequality small lambda Neu} and the second assumption we have for any $\gamma' <\gamma_1$ and all $\lambda \leq \Lambda_0$ that
\begin{equation*}
    \Tr(-\Delta_\Omega^{\rm N}-\lambda)_\limminus^{\gamma'} \geq c f_{\gamma_1}^{\rm N}(\gamma') L_{\gamma', d}^{\rm sc}|\Omega|\lambda^{\gamma'+ \frac{d}2}\,.
\end{equation*}
Since $f_{\gamma_1}^{\rm N}$ is continuous and increasing with $\lim_{\gamma'\to \gamma_1}f_{\gamma_1}^{\rm N}(\gamma') =1$ and since $c>1$, we find that \eqref{eq: goal extrapolation Neu} holds for all $\gamma$ satisfying
\begin{equation*}
    (f_{\gamma_1}^{\rm N})^{-1}\Bigl(\frac{1}{c}\Bigr)\leq \gamma \leq \gamma_1\,. \qedhere
\end{equation*}
\end{proof}


\section{Semiclassical inequalities in convex domains} 

\subsection{Proof of Theorems \ref{thm: Extended range of semiclassical ineq Dir} and \ref{thm: Extended range of semiclassical ineq Neu}}

In the setting of convex sets Proposition~\ref{prop: general extrapolation} provides the following results, which essentially say that the strict inequality $\gamma>\gamma_d^{\sharp}$ holds if in a certain sense an improved semiclassical inequality holds for Riesz means of order $\gamma$, uniformly among convex sets.

\begin{proposition}\label{prop: extrapolation of improved inequality Dir}
	Let $d\geq 1$. Assume that $\gamma>0$ and that for each $B>0$ there exists a constant $c<1$ such that for all bounded convex open sets $\Omega \subset \R^d$ and all $\lambda\leq \frac{B}{r_{\rm in}(\Omega)^2}$ we have
		\begin{equation*}
			\Tr(-\Delta_\Omega^{\rm D}-\lambda)_\limminus^{\gamma} \leq c L_{\gamma,d}^{{\rm sc}} |\Omega| \lambda^{\gamma+\frac d2} \,.
		\end{equation*}
	Then $\gamma>\gamma_d^{\rm D}$.
\end{proposition}

\begin{proposition}\label{prop: extrapolation of improved inequality Neu}
	Let $d\geq 1$. Assume that $\gamma>0$ and that for each $B>0$ there exists a constant $c>1$ such that for all bounded convex open sets $\Omega \subset \R^d$ and all $\lambda\leq \frac{B}{r_{\rm in}(\Omega)^2}$ we have
		\begin{equation*}
			\Tr(-\Delta_\Omega^{\rm N}-\lambda)_\limminus^{\gamma} \geq c L_{\gamma,d}^{{\rm sc}} |\Omega| \lambda^{\gamma+\frac d2} \,.
		\end{equation*}
	Then $\gamma>\gamma_d^{\rm N}$.
\end{proposition}

\begin{remark}
	By tracking constants in our proofs (and those in~\cite{FrankLarson_24,FrankLarson_heatkernels,FrankLarson_Crelle20}) the gap $\gamma-\gamma_d^{\sharp}$ can be quantified.
\end{remark}

We emphasize that Propositions \ref{prop: extrapolation of improved inequality Dir} and \ref{prop: extrapolation of improved inequality Neu} will be superceded by Theorems \ref{thm: improved inequality above critical gamma Dir} and~\ref{thm: improved inequality above critical gamma Neu}. We have opted to state the propositions separately since they lead to a quick proof of our first pair of main results, Theorems \ref{thm: Extended range of semiclassical ineq Dir} and \ref{thm: Extended range of semiclassical ineq Neu}.

\begin{proof}[Proof of Propositions~\ref{prop: extrapolation of improved inequality Dir} and \ref{prop: extrapolation of improved inequality Neu}]
	By the Berezin--Li--Yau and Kr\"oger inequalities we have $\gamma_d^\sharp \leq 1$, so without loss of generality we may assume that $\gamma\leq 1$. By \cite[Theorem~1.2]{FrankLarson_24}, for any $\gamma' >0$ there exists a $B>0$ such that for all bounded, convex open sets $\Omega\subset\R^d$ and all $\lambda \geq \frac{B}{r(\Omega)^2}$ we have
	\begin{equation*}
		\Tr(-\Delta_\Omega^{\rm D}-\lambda)_\limminus^{\gamma'}\leq L^{\rm sc}_{\gamma',d}|\Omega|\lambda^{\gamma'+\frac{d}2} \leq \Tr(-\Delta_\Omega^{\rm N}-\lambda)_\limminus^{\gamma'} \,.
	\end{equation*} 
	In fact, $B$ can be chosen uniformly bounded for $\gamma'$ in any compact subset of $(0, 1]$.
	Indeed, the proof of \cite[Theorem~1.2]{FrankLarson_24} rests on using an argument of Riesz to prove inequalities for all $\gamma \in (\gamma_0, 1)$ by interpolating between corresponding inequalities at the endpoints and the constant produced in this interpolation depends continuously on the interpolation parameter. Alternatively, uniformity of the error estimate for $\gamma \in [\gamma_0, 1]$  can also be deduced by the Aizenman--Lieb argument deducing two-term asymptotics for Riesz means of order $\gamma$ by integrating those known at $\gamma_0$. Therefore, the first assumption in parts \ref{itm: extrap D} and \ref{itm: extrap N} of Proposition \ref{prop: general extrapolation} is satisfied. The second assumption in parts \ref{itm: extrap D} and \ref{itm: extrap N} of Proposition \ref{prop: general extrapolation} holds by assumption. Thus, the assertion follows from Proposition \ref{prop: general extrapolation}.
\end{proof}

We are now in position to prove the first main result of our paper.

\begin{proof}[Proof of Theorems~\ref{thm: Extended range of semiclassical ineq Dir} and \ref{thm: Extended range of semiclassical ineq Neu}]
    We apply Propositions \ref{prop: extrapolation of improved inequality Dir} and \ref{prop: extrapolation of improved inequality Neu} with $\gamma=1$. The assumptions of those propositions are immediate consequences of Propositions~\ref{prop: improved Berezin} and \ref{prop: improved Kroger}, respectively.
\end{proof}

We end this subsection with another application of Proposition \ref{prop: general extrapolation}.

\begin{corollary}
    Let $d\geq 1$ and let $\Omega\subset\R^d$ be a Lipschitz set of finite measure. Then there are $\gamma_\Omega^{\rm D}<1$ and $\gamma_\Omega^{\rm N}<1$ such that for all $\gamma\geq\gamma_\Omega^{\rm D}$ and all $\lambda>0$ we have
    $$
    \Tr(-\Delta_\Omega^{\rm D}-\lambda)_\limminus^{\gamma} \leq L_{\gamma,d}^{{\rm sc}} |\Omega| \lambda^{\gamma+\frac d2}
    $$
    and for all $\gamma\geq\gamma_\Omega^{\rm N}$ and all $\lambda>0$ we have
    $$
    \Tr(-\Delta_\Omega^{\rm N}-\lambda)_\limminus^{\gamma} \geq L_{\gamma,d}^{{\rm sc}} |\Omega| \lambda^{\gamma+\frac d2} \,.
    $$
\end{corollary}

\begin{proof}
    With $\gamma_1=1$ and any $0<\gamma_0<1$ the assumptions of both \ref{itm: extrap D} and \ref{itm: extrap N} in Proposition \ref{prop: general extrapolation} are satisfied. Indeed, the first assumption is true for some $\Lambda_0$ by the two-term asymptotics proved in \cite{FrankLarson_24}. To see that the second assumption is valid we argue as follows. The proofs of the Berezin--Li--Yau and Kr\"oger inequalities show that equality can never hold in the inequalities. (This is quantified in our Propositions \ref{prop: improved Berezin} and \ref{prop: improved Kroger}. Note the Lipschitz sets of finite measure are necessarily bounded.) Since
    \begin{equation*}
        g^\sharp(\lambda)= \frac{\Tr(-\Delta_\Omega^{\sharp}-\lambda)_\limminus}{\lambda^{1+\frac{d}2}}
    \end{equation*}
    is continuous on $(0, \Lambda]$ for any $\Lambda>0$ and $\lim_{\lambda \to 0^\limplus}g^{\rm D}(\lambda)=0, \lim_{\lambda \to 0^\limplus}g^{\rm N}(\lambda)=\infty$ the functions $g^{\rm D}$ and $g^{\rm N}$ attain a maximal and minimal value on $[0, \Lambda]$, respectively. By the strictness of the Berezin--Li--Yau and Kr\"oger inequalities these maximal and minimal values are less than resp.\ larger than $L_{1,d}^{\rm cl}|\Omega|$. This shows that also the second assumptions are true.    
\end{proof}


\subsection{A characterization of the critical Riesz exponent}

An important ingredient in our proof of Theorem~\ref{thm: Main conclusions gamma_d}
is provided by passing via improved versions of semiclassical inequalities. In Propositions~\ref{prop: extrapolation of improved inequality Dir} and \ref{prop: extrapolation of improved inequality Neu} we showed that the validity of an improved version of the semiclassical inequalities for some $\gamma>0$ and $\lambda \lesssim r_{\rm in}(\Omega)^{-2}$ implied that $\gamma> \gamma_d^{\sharp}$. In this section we shall prove that a stronger form of improved semiclassical inequalities hold for convex sets as soon as $\gamma>\gamma_d^{\sharp}$. We refer to these stronger improved inequalities as two-term semiclassical inequalities as they improve upon the inequalities in~\eqref{eq: Polya conjecture gamma} by either subtracting or adding a non-negative second term that matches the second term in the corresponding two-term asymptotic expansion of the Riesz means. Bounds of this form and sufficiently large values of $\gamma$ have previously been obtained for the Dirichlet case, see \cite{GeLaWe,LarsonPAMS,LarsonJST,FrankLarson_Crelle20}.

Remarkably, the validity of such two-term semiclassical inequalities characterizes the critical Riesz exponents $\gamma_d^\sharp$.

More specifically, we shall prove the following two theorems. 

\begin{theorem}\label{thm: improved inequality above critical gamma Dir}
    Let $d\geq 1$ and $\gamma>0$. The following are equivalent:
    \begin{enumerate}[label=\textup{(\roman*)}]
        \item\label{itm: equiv 1D} $\gamma>\gamma_d^{\rm D}$.
        \item\label{itm: equiv 2D} For each $B>0$ there is a constant $c<1$ such that for all bounded convex open sets $\Omega\subset\R^d$ and all $\lambda \leq \frac{B}{r_{\rm in}(\Omega)^2}$ we have
        $$
        \Tr(-\Delta_\Omega^{\rm D}-\lambda)_\limminus^{\gamma} \leq c L_{\gamma,d}^{{\rm sc}} |\Omega| \lambda^{\gamma+\frac d2} \,.
        $$
        \item\label{itm: equiv 3D} There is a constant $c_{\gamma,d}>0$ such that for all bounded convex open sets $\Omega \subset \R^d$ and $\lambda\geq 0$ we have
	\begin{equation*}
		\Tr(-\Delta_\Omega^{\rm D}-\lambda)_\limminus^\gamma \leq \Bigl(L^{\rm sc}_{\gamma,d} |\Omega| \lambda^{\gamma+\frac d2} -c_{\gamma,d}\Haus^{d-1}(\partial\Omega)\lambda^{\gamma+ \frac{d-1}{2}}\Bigr)_\limplus \,.
	\end{equation*}		
    \end{enumerate}
\end{theorem}

\begin{theorem}\label{thm: improved inequality above critical gamma Neu}
    Let $d\geq 1$ and $\gamma>0$. The following are equivalent:
    \begin{enumerate}[label=\textup{(\roman*)}]
        \item\label{itm: equiv 1N} $\gamma>\gamma_d^{\rm N}$.
        \item\label{itm: equiv 2N} For each $B>0$ there is a constant $c>1$ such that for all bounded convex open sets $\Omega\subset\R^d$ and all $\lambda \leq \frac{B}{r_{\rm in}(\Omega)^2}$ we have
        $$
        \Tr(-\Delta_\Omega^{\rm N}-\lambda)_\limminus^{\gamma} \geq c L_{\gamma,d}^{{\rm sc}} |\Omega| \lambda^{\gamma+\frac d2} \,.
        $$
        \item\label{itm: equiv 3N} There is a constant $c_{\gamma,d}>0$ such that for all bounded convex open sets $\Omega \subset \R^d$ and $\lambda\geq 0$ we have
	\begin{equation*}
		\Tr(-\Delta_\Omega^{\rm N}-\lambda)_\limminus^\gamma \geq L^{\rm sc}_{\gamma,d} |\Omega| \lambda^{\gamma+\frac d2} +c_{\gamma,d}\Haus^{d-1}(\partial\Omega)\lambda^{\gamma+ \frac{d-1}{2}} \,.
	\end{equation*}		
    \end{enumerate}
\end{theorem}


The main thrust of Theorems~\ref{thm: improved inequality above critical gamma Dir} and \ref{thm: improved inequality above critical gamma Neu} is the implication (i)$\implies$(iii). The proofs in the Dirichlet and Neumann cases are based on the same ideas:
\begin{enumerate}
	\item if $\lambda r_{\rm in}(\Omega)^2$ is sufficiently large (depending only on $d$), then the quantitative asymptotics proved in~\cite{FrankLarson_24} imply the claimed bound.
	\item if $\lambda r_{\rm in}(\Omega)^2$ is smaller than some number (depending only on $d$), then the claimed inequality can be obtained rather directly (for any value of $\gamma$).
	\item for intermediate values of $\lambda r_{\rm in}(\Omega)^2$ we use the Aizenman--Lieb identity (or Riesz iteration) to write the relevant trace as an integral of $\Tr(-\Delta_\Omega^{\sharp}-\tau)_\limminus^{\gamma_d^\sharp}$ with respect to $\tau$. Using the improved inequality of the second step and the fact that by definition of $\gamma_d^\sharp$ the Riesz mean that is integrated satisfies a semiclassical inequality allows us to bridge the remaining gap. 
\end{enumerate}
The main difference in the proofs between the Dirichlet and the Neumann case is the bound that is applied for small values of $\lambda r_{\rm in}(\Omega)^2$.

\begin{proof}[Proof of Theorem~\ref{thm: improved inequality above critical gamma Dir}]
    The implication \ref{itm: equiv 2D}$\implies$\ref{itm: equiv 1D} follows from Proposition \ref{prop: extrapolation of improved inequality Dir}.
    
    To show \ref{itm: equiv 3D}$\implies$\ref{itm: equiv 2D}, given $B>0$ we set
    $$
    c:= (1 - c_{\gamma,d} (L_{\gamma,d}^{\rm sc})^{-1} B^{-1/2} )_+ \,.
    $$
    Then for $\lambda\geq \frac{B}{r_{\rm in}(\Omega)^2}$, according to \eqref{eq: inradius bound},
    $$
    \left( L_{\gamma,d}^{\rm sc} |\Omega| \lambda^{\gamma+\frac d2} - c_{\gamma,d} \mathcal H^{d-1}(\partial\Omega) \lambda^{\gamma+\frac{d-1}2} \right)_+
    \leq c L_{\gamma,d}^{\rm sc} |\Omega| \lambda^{\gamma+\frac d2} \,.
    $$
    Thus the inequality in \ref{itm: equiv 2D} follows from that in \ref{itm: equiv 3D}.

    It thus remains to prove the implication \ref{itm: equiv 1D}$\implies$\ref{itm: equiv 3D}. If $d\geq 2$ then by \cite[Theorem 1.2]{FrankLarson_24}, for any $\gamma >0$ and $A < \frac{L_{\gamma,d-1}^{\rm sc}}{4}$ there exists a constant $B> 0$ so that for all bounded convex open sets $\Omega\subset\R^d$ and all $\lambda \geq \frac{B}{r_{\rm in}(\Omega)^2}$ we have
	\begin{equation*}
	 	 \Tr(-\Delta_\Omega^{\rm D}-\lambda)_\limminus^\gamma \leq L_{\gamma, d}^{\rm sc} |\Omega|\lambda^{\gamma+\frac{d}{2}}- A \Haus^{d-1}(\partial\Omega)\lambda^{\gamma+ \frac{d-1}{2}} \,.
	\end{equation*}
    For $d=1$ the corresponding inequality can be proved by explicit calculations using that the eigenvalues of the Laplacian on an interval are explicit and the Aizenman--Lieb procedure.
	Consequently, it suffices to prove that the claimed inequality is true for $0\leq \lambda < \frac{B}{r_{\rm in}(\Omega)^2}$ when $\gamma >\gamma_d^{\rm D}$.

	To deal with $\lambda$ close to zero we recall the Hersch--Protter inequality
	$$\lambda_1(\Omega)\geq \frac{\pi^2}{4r_{\textup{in}}(\Omega)^2}$$ 
	for any open and bounded convex set $\Omega \subset \R^d$; see \cite{Hersch_60,Protter_81} and also \cite[Lemma 3.9]{FrankLaptevWeidl}. Therefore, for $0\leq \lambda \leq \frac{\pi^2}{4r_{\rm in}(\Omega)^2}$ the claimed inequality is trivially true as the left-hand side is zero.

	If $B \leq \frac{\pi^2}{4}$ we are done. If $B > \frac{\pi^2}{4}$ it remains to prove the inequality for $\lambda \in \bigl(\frac{\pi^2}{4r_{\rm in}(\Omega)^2}, \frac{B}{r_{\rm in}(\Omega)^2}\bigr)$. 
	To this end we recall the Aizenman--Lieb identity, if $\gamma>\gamma'$,
	\begin{equation}\label{eq: Aizenman-Lieb}
		\Tr(-\Delta_\Omega^{\sharp} -\lambda)_\limminus^\gamma = B(1+\gamma', \gamma-\gamma')^{-1}\int_0^\lambda (\lambda-\tau)^{\gamma-\gamma'-1}\Tr(-\Delta_\Omega^{\sharp}-\tau)_\limminus^{\gamma'}\,d\tau\,,
	\end{equation}
	where $B$ denotes the Euler Beta function $B(x, y) = \frac{\Gamma(x)\Gamma(y)}{\Gamma(x+y)}$.

	Applying \eqref{eq: Aizenman-Lieb} with $\gamma'= \gamma_d^{\rm D}$ and using the fact that the trace in the integral is zero if $\tau \leq \lambda_1^{\rm D}(\Omega)$ together with the fact that by the definition of $\gamma_d^{\rm D}$ the trace in the integral can be bounded by the Weyl term we conclude that
	\begin{align*}
		\Tr(-\Delta_\Omega^{\rm D} -\lambda)_\limminus^\gamma &\leq L_{\gamma,d}^{\rm sc}|\Omega|\lambda^{\gamma+\frac{d}{2}}- \frac{L_{\gamma_d^{\rm D},d}^{\rm sc}|\Omega|}{B(1+\gamma_d^{\rm D}, \gamma-\gamma_d^{\rm D})}\int_0^{\lambda_1(\Omega)} (\lambda-\tau)^{\gamma-\gamma_d^{\rm D}-1}\tau^{\gamma_d^{\rm D}+\frac{d}{2}}\,d\tau\\
		&= L_{\gamma,d}^{\rm sc}|\Omega|\lambda^{\gamma+\frac{d}{2}}- \frac{L_{\gamma_d^{\rm D},d}^{\rm sc}|\Omega|\lambda^{\gamma+\frac{d}{2}}}{B(1+\gamma_d^{\rm D}, \gamma-\gamma_d^{\rm D})}\int_0^{\lambda_1(\Omega)/\lambda} (1-s)^{\gamma-\gamma_d^{\rm D}-1}s^{\gamma_d^{\rm D}+\frac{d}{2}}\,ds\,.
	\end{align*}
	For $\lambda \in\bigl(\frac{\pi^2}{4r_{\rm in}(\Omega)^2}, \frac{B}{r_{\rm in}(\Omega)^2}\bigr)$ the Hersch--Protter inequality implies that $\frac{\lambda_1(\Omega)}{\lambda} \geq \frac{\pi^2}{4B}$ and therefore
	\begin{align*}
		\Tr(-\Delta_\Omega^{\rm D} -\lambda)_\limminus^\gamma &\leq  L_{\gamma,d}^{\rm sc}|\Omega|\lambda^{\gamma+\frac{d}{2}}- \frac{L_{\gamma_d^{\rm D},d}^{\rm sc}|\Omega|\lambda^{\gamma+\frac{d}{2}}}{B(1+\gamma_d^{\rm D}, \gamma-\gamma_d^{\rm D})}\int_0^{\frac{\pi^2}{4B}} (1-s)^{\gamma-\gamma_d^{\rm D}-1}s^{\gamma_d^{\rm D}+\frac{d}{2}}\,ds\\
		&\leq  L_{\gamma,d}^{\rm sc}|\Omega|\lambda^{\gamma+\frac{d}{2}}- \frac{\pi L_{\gamma_d^{\rm D},d}^{\rm sc}\Haus^{d-1}(\partial\Omega)\lambda^{\gamma+\frac{d-1}{2}}}{2 dB(1+\gamma_d^{\rm D}, \gamma-\gamma_d^{\rm D})}\int_0^{\frac{\pi^2}{4B}} (1-s)^{\gamma-\gamma_d^{\rm D}-1}s^{\gamma_d^{\rm D}+\frac{d}{2}}\,ds\, \,,
	\end{align*}
	where in the second step we used the fact that, by~\eqref{eq: inradius bound} and the assumption $\lambda \geq\frac{\pi^2}{4r_{\rm in}(\Omega)^2}$,
	\begin{equation*}
		|\Omega|\sqrt{\lambda} \geq \frac{\pi|\Omega|}{2r_{\rm in}(\Omega)} \geq \frac{\pi}{2d} \Haus^{d-1}(\partial\Omega)\,.
	\end{equation*}

	Therefore, we conclude that for all $\lambda \geq 0$ the inequality in \ref{itm: equiv 3D} of Theorem~\ref{thm: improved inequality above critical gamma Dir} holds with
	\begin{align*}
		c_{\gamma,d} &= \min\biggl\{A, \frac{\pi L_{\gamma_d^{\rm D},d}^{\rm sc}}{2 dB(1+\gamma_d^{\rm D}, \gamma-\gamma_d^{\rm D})}\int_0^{\frac{\pi^2}{4B}} (1-s)^{\gamma-\gamma_d^{\rm D}-1}s^{\gamma_d^{\rm D}+\frac{d}{2}}\,ds\biggr\}\,,
	\end{align*}
    which completes the proof.
\end{proof}

For the Neumann case, instead of the Hersch--Protter bound, we shall use the inequality from Lemma \ref{lem: small energy improved Kroger}.

\begin{proof}[Proof of Theorem~\ref{thm: improved inequality above critical gamma Neu}]
    The implications \ref{itm: equiv 2N}$\implies$\ref{itm: equiv 1N} and \ref{itm: equiv 3N}$\implies$\ref{itm: equiv 2N} can be shown as in the Dirichlet case.

    It thus remains to prove the implication \ref{itm: equiv 1N}$\implies$\ref{itm: equiv 3N}. By the same argument as in the Dirichlet case, \cite[Theorem 1.2]{FrankLarson_24} implies that for any $A< \frac{L_{\gamma,d-1}^{\rm sc}}{4}$ there exists $B$ so that it suffices to prove the claimed inequality in the range $0\leq \lambda < \frac{B}{r_{\rm in}(\Omega)^2}$ when $\gamma >\gamma_d^{\rm N}$. To accomplish this we apply Lemma~\ref{lem: small energy improved Kroger} together with the Aizenman--Lieb identity. 
    
    Let $c_1 := \frac{1}2C_{\gamma_d^{\rm N},d}/L_{\gamma_d^{\rm N}, d}^{\rm sc}$, where $C_{\gamma,d}$ denote the constants appearing in Lemma~\ref{lem: small energy improved Kroger}. If $\lambda \leq \frac{c_1^2}{r_{\rm in}(\Omega)^2}$, then by Lemma~\ref{lem: small energy improved Kroger} and \eqref{eq: inradius bound}
	\begin{equation}\label{eq: improved bound small lambda critical}
	\begin{aligned}
		\Tr(-\Delta_\Omega^{\rm N}-\lambda)_\limminus^{\gamma_d^{\rm N}} 
		&\geq 
			C_{\gamma_d^{\rm N},d} \frac{|\Omega|}{r_{\rm in}(\Omega)}\lambda^{\gamma_d^{\rm N} + \frac{d-1}{2}}\\
		&= 
			L_{\gamma_d^{\rm N}, d}^{\rm sc} |\Omega|\lambda^{\gamma_d^{\rm N}+ \frac{d}{2}} + \Bigl(\frac{C_{\gamma_d^{\rm N},d}|\Omega|}{r_{\rm in}(\Omega)}- L_{\gamma_d^{\rm N},d}^{\rm sc}|\Omega|\sqrt{\lambda}\Bigr)\lambda^{\gamma_d^{\rm N} + \frac{d-1}{2}}\\
		&\geq
			L_{\gamma_d^{\rm N}, d}^{\rm sc} |\Omega|\lambda^{\gamma_d^{\rm N}+ \frac{d}{2}} + \frac{C_{\gamma_d^{\rm N},d}|\Omega|}{2r_{\rm in}(\Omega)}\lambda^{\gamma_d^{\rm N} + \frac{d-1}{2}}\\
		&\geq
			L_{\gamma_d^{\rm N}, d}^{\rm sc} |\Omega|\lambda^{\gamma_d^{\rm N}+ \frac{d}{2}} + \frac{C_{\gamma_d^{\rm N},d}}{2d}\Haus^{d-1}(\partial\Omega)\lambda^{\gamma_d^{\rm N} + \frac{d-1}{2}}\,.
	\end{aligned}
	\end{equation}

	Now assume that $0\leq \lambda \leq \frac{B}{r_{\rm in}(\Omega)^2}$. Then, by~\eqref{eq: Aizenman-Lieb} with $\gamma'=\gamma_d^{\rm N}$, the inequality that holds by the definition of $\gamma_d^{\rm N}$ for $\tau \geq \frac{c_1^2}{r_{\rm in}(\Omega)^2}$, and using \eqref{eq: improved bound small lambda critical} for $\tau \leq \frac{c_1^2}{r_{\rm in}(\Omega)^2}$, we obtain
	\begin{align*}
		\Tr(-\Delta_\Omega^{\rm N} -\lambda)_\limminus^\gamma 
		&\geq 
		B(1+\gamma_d^{\rm N}, \gamma-\gamma_d^{\rm N})^{-1}\int_{0}^\lambda (\lambda-\tau)^{\gamma-\gamma_d^{\rm N}-1}L_{\gamma_{d}^{\rm N},d}|\Omega|\tau^{\gamma_d^{\rm N}+\frac{d}{2}}\,d\tau\\
		&\quad +\frac{C_{\gamma_d^{\rm N},d}\Haus^{d-1}(\partial\Omega)}{2d B(1+\gamma_d^{\rm N}, \gamma-\gamma_d^{\rm N})} \int_{0}^{\min\{\lambda, c_1^2/r_{\rm in}(\Omega)^2\}} (\lambda-\tau)^{\gamma-\gamma_d^{\rm N}-1}\tau^{\gamma_d^{\rm N}+ \frac{d-1}{2}}\,d\tau\\
		&=
		L_{\gamma, d}^{\rm sc}|\Omega|\lambda^{\gamma+\frac{d}{2}}\\
		&\quad +\frac{C_{\gamma_d^{\rm N},d}\Haus^{d-1}(\partial\Omega)\lambda^{\gamma+ \frac{d-1}{2}}}{2dB(1+\gamma_d^{\rm N}, \gamma-\gamma_d^{\rm N})} \int_{0}^{\min\{1, c_1^2/(\lambda r_{\rm in}(\Omega)^2)\}} (1-s)^{\gamma-\gamma_d^{\rm N}-1}s^{\gamma_d^{\rm N}+ \frac{d-1}{2}}\,ds\\
		&\geq
		L_{\gamma, d}^{\rm sc}|\Omega|\lambda^{\gamma+\frac{d}{2}}\\
		&\quad +\frac{C_{\gamma_d^{\rm N},d}\Haus^{d-1}(\partial\Omega)\lambda^{\gamma+ \frac{d-1}{2}}}{2d B(1+\gamma_d^{\rm N}, \gamma-\gamma_d^{\rm N})} \int_{0}^{\min\{1, c_1^2/B\}} (1-s)^{\gamma-\gamma_d^{\rm N}-1}s^{\gamma_d^{\rm N}+ \frac{d-1}{2}}\,ds\,.
	\end{align*}

	Therefore, we conclude that for all $\lambda \geq 0$ the inequality in \ref{itm: equiv 3N} of Theorem~\ref{thm: improved inequality above critical gamma Neu} holds with
	\begin{align*}
		c_{\gamma,d} &= \min\biggl\{A, \frac{C_{\gamma_d^{\rm N},d}}{2d B(1+\gamma_d^{\rm N}, \gamma-\gamma_d^{\rm N})} \int_{0}^{\min\{1,c_1^2/B\}} (1-s)^{\gamma-\gamma_d^{\rm N}-1}s^{\gamma_d^{\rm N}+ \frac{d-1}{2}}\,ds\biggr\}\,,
	\end{align*}
    which completes the proof.
\end{proof}


\section{The optimization problems}

In this section we mostly consider the case $\gamma_d^\sharp>0$, that is, the case when P\'olya's conjecture fails among convex sets. We study the optimization problems associated to the inequalities in \eqref{eq: Polya conjecture gamma} for exponents $0\leq\gamma\leq\gamma_d^\sharp$. In particular, we will prove Proposition \ref{dimred} as well as Theorem~\ref{thm: Main conclusions gamma_d}. The latter concerns the case $\gamma=\gamma_d^\sharp$, and in Theorem \ref{thm: compactness via r} we complement this with a corresponding result for $\gamma<\gamma_d^\sharp$.

Our proofs are formulated in terms of the so-called excess factors which are defined as the optimal constants $r_{\gamma,d}^\sharp$ so that, for all $\Omega\in \mathcal{C}_d$ and all $\lambda\geq 0$,
\begin{equation}\label{eq: sharp ineq.s in terms of r}
    \Tr(-\Delta_\Omega^{\rm D}-\lambda)_\limminus^\gamma  \leq r_{\gamma,d}^{\rm D}L_{\gamma, d}^{\rm sc}|\Omega|\lambda^{\gamma+\frac{d}2}\quad  \mbox{and} \quad \Tr(-\Delta_\Omega^{\rm N}-\lambda)_\limminus^\gamma  \geq r_{\gamma,d}^{\rm N}L_{\gamma, d}^{\rm sc}|\Omega|\lambda^{\gamma+\frac{d}2}\,.
\end{equation}
Let us summarize some facts about these constants. All these facts are valid also for the sharp constants in the corresponding inequalities without the convexity assumption on the underlying domain.
\begin{enumerate}
    \item[(a)] $r_{\gamma,d}^{\rm D} <\infty$ and $r_{\gamma,d}^{\rm N}>0$; this is well known, see \cite[Corollaries 3.30 and 3.39]{FrankLaptevWeidl}.
    \item[(b)] $r_{\gamma, d}^{\rm D}\geq 1 \geq r_{\gamma,d}^{\rm N}$; this is a consequence of Weyl's law.
    \item[(c)] $\gamma\mapsto r_{\gamma,d}^\sharp$ is non-increasing (for $\sharp=$ D) and non-decreasing (for $\sharp=$ N); this is a consequence of the Aizenman--Lieb argument.
    \item[(d)] $\gamma\mapsto r_{\gamma,d}^\sharp$ is continuous; this follows from Lemmas~\ref{lem: extrapolating improved inequality small lambda Dir} and~\ref{lem: extrapolating improved inequality small lambda Neu}.
    \item[(e)] $r_{\gamma, d}^{\rm D} = 1 = r_{\gamma,d}^{\rm N}$ for $\gamma\geq 1$; this is a consequence of the Berezin--Li--Yau and Kr\"oger inequalities and the monotonicity in (c).
\end{enumerate}

The critical exponent $\gamma_d^\sharp$ can be characterized as the unique smallest number so that
$$
r_{\gamma,d}^\sharp =1
\qquad\text{for all}\
\gamma\geq \gamma^\sharp_d \,.
$$
In fact, Proposition \ref{dimred} will be deduced as a corollary of the following result for the excess factors.

\begin{proposition}\label{dimred r}
    Fix $d\geq 2$ and $\sharp\in\{ \rm D, \rm N\}$. Then, for all $\gamma\geq 0$,
    $$
    r_{\gamma, d}^{\rm D}\geq r_{\gamma+ \frac{1}2,d-1}^{\rm D}
    \qquad\text{and}\qquad
    r_{\gamma, d}^{\rm N}\leq r_{\gamma+ \frac{1}2,d-1}^{\rm N} \,.
    $$
\end{proposition}


The proofs in this section will rely, among other things, on Theorem \ref{thm: Asymptotics degenerating convex sets}. We take this for granted and will prove it in Section \ref{sec:asymptoticscollapsing} below. (The proof there is independent of the remaining sections of this paper.)

Besides the asymptotic result for collapsing convex sets (Theorem \ref{thm: Asymptotics degenerating convex sets}) an important role is played by quantitative error estimates in Weyl asymptotics for non-collapsing convex sets from our previous paper~\cite{FrankLarson_24} for Riesz exponents $\gamma>0$. In addition we need the following result for $\gamma=0$, which provides the correct leading order term in the semiclassical limit, together with a remainder term. The remainder term is (probably) off by a logarithm, but the point of the result is that the remainder depends on the underlying set only through its inradius.

\begin{theorem}\label{thm: quantitative Weyl law}
    Let $d\geq 2$. There exists a constant $C>0$ so that for all open, bounded, and convex $\Omega \subset \R^d$ and all $\lambda > 0$ we have
    \begin{equation*}
        \Bigl|\Tr(-\Delta_\Omega^{\rm D}-\lambda)_\limminus^0 - L_{0,d}^{\rm sc}|\Omega|\lambda^{\frac{d}2}\Bigr| \leq C \Haus^{d-1}(\partial\Omega) \lambda^{\frac{d-1}2}\bigl(1+\ln_\limplus\bigl(r_{\rm in}(\Omega)\sqrt{\lambda}\bigr)\hspace{-1.5pt}\bigr) 
    \end{equation*}
    and
    \begin{equation*}
        \Bigl|\Tr(-\Delta_\Omega^{\rm N}-\lambda)_\limminus^0 - L_{0,d}^{\rm sc}|\Omega|\lambda^{\frac{d}2}\Bigr| \!\leq \!C \Haus^{d-1}(\partial\Omega) \lambda^{\frac{d-1}2}\!\bigl(\max\bigl\{\hspace{-1.7pt}1, \hspace{-2pt}\bigl(r_{\rm in}(\Omega)\sqrt{\lambda}\bigr)^{\!1-d}\bigr\}+\ln_\limplus\!\bigl(r_{\rm in}(\Omega)\sqrt{\lambda}\bigr)\hspace{-1.5pt}\bigr). 
    \end{equation*}
\end{theorem}

We omit the proof of this theorem. It follows by repeating the arguments used to prove \cite[Theorems 3.1 \& 3.3]{FrankLarson_24} in the case when $\Omega\subset \R^d$ is bounded and convex and $\gamma=0$.


\subsection{Comparison between dimensions {\it d} and {\it d} {\rm -- 1}}

Our goal in this subsection is to prove Propositions \ref{dimred} and \ref{dimred r}. The main step is contained in the following lemma, which will be useful again later on.

\begin{lemma}\label{lem: cylinder lift}
    Fix $d\geq 2, \gamma \geq 0$. If $\omega \subset \R^{d-1}$ is a bounded open set and if $\lambda, \ell >0$, then
    \begin{align*}
        0\geq \frac{\Tr(-\Delta_{\omega \times (0, \ell)}^{\rm D}-\lambda)_\limminus^\gamma}{L_{\gamma, d}^{\rm sc}\ell \Haus^{d-1}(\omega)\lambda^{\gamma+ \frac{d}2}}- \frac{\Tr(-\Delta_{\omega}^{\rm D}-\lambda)_\limminus^{\gamma+\frac{1}2}}{L_{\gamma+\frac{1}2, d-1}^{\rm sc}\Haus^{d-1}(\omega)\lambda^{\gamma+ \frac{d}2}}
        &\geq
        -\frac{\Tr(-\Delta_\omega^{\rm D}-\lambda)_\limminus^\gamma}{L_{\gamma, d}^{\rm sc}\ell \Haus^{d-1}(\omega)\lambda^{\gamma+ \frac{d}2}}\,,
    \end{align*}
    and if, in addition, the spectrum of $-\Delta_\omega^{\rm N}$ in $[0, \lambda]$ is discrete then
    \begin{align*}
        0\leq \frac{\Tr(-\Delta_{\omega \times (0, \ell)}^{\rm N}-\lambda)_\limminus^\gamma}{L_{\gamma, d}^{\rm sc}\ell \Haus^{d-1}(\omega)\lambda^{\gamma+ \frac{d}2}}-\frac{\Tr(-\Delta_{\omega}^{\rm N}-\lambda)_\limminus^{\gamma+\frac{1}2}}{L_{\gamma+\frac{1}2, d-1}^{\rm sc}\Haus^{d-1}(\omega)\lambda^{\gamma+ \frac{d}2}}
        &\leq
        \frac{\Tr(-\Delta_\omega^{\rm N}-\lambda)_\limminus^\gamma}{L_{\gamma, d}^{\rm sc}\ell \Haus^{d-1}(\omega)\lambda^{\gamma+ \frac{d}2}}\,.
    \end{align*}
\end{lemma}
\begin{proof}
    By the product structure of $\omega \times (0, \ell)$ the eigenvalues of $-\Delta_{\omega \times (0, \ell)}^\sharp$ are given by $\lambda_k^{\sharp}(\omega) +\lambda_j^\sharp((0, \ell))$ with corresponding multiplicities. Recall that the spectrum of $-\Delta_{(0, \ell)}^{\rm N}$ differs from the spectrum of $-\Delta_{(0, \ell)}^{\rm D}$ only by the simple eigenvalue at $0$.

    Then for any $\ell>0$ using the validity of P\'olya's conjecture for the Neumann Laplacian on an interval and the fact that $L_{\gamma,d}^{\rm sc}=L_{\gamma+\frac{1}2,d-1}^{\rm sc}L_{\gamma,1}^{\rm sc}$,
    \begin{align*} 
    \frac{\Tr(-\Delta_{\omega \times (0, \ell)}^{\rm D}-\lambda)_\limminus^\gamma}{L_{\gamma, d}^{\rm sc}\ell \Haus^{d-1}(\omega)\lambda^{\gamma+ \frac{d}2}}
    &=
    \frac{\sum_{k\geq 1}\Bigl(\Tr(-\Delta_{(0, \ell)}^{\rm N}-\lambda+\lambda_k^{\rm D}(\omega))_\limminus^\gamma-(\lambda-\lambda_k^{\rm D}(\omega))_\limplus^\gamma\Bigr)}{L_{\gamma, d}^{\rm sc}\ell \Haus^{d-1}(\omega)\lambda^{\gamma+ \frac{d}2}}\\
    &\geq
    \frac{\Tr(-\Delta_{\omega}^{\rm D}-\lambda)_\limminus^{\gamma+\frac{1}2}}{L_{\gamma+\frac{1}2, d-1}^{\rm sc}\Haus^{d-1}(\omega)\lambda^{\gamma+\frac{d}2}}-\frac{\Tr(-\Delta_\omega^{\rm D}-\lambda)_\limminus^\gamma}{L_{\gamma, d}^{\rm sc}\ell \Haus^{d-1}(\omega)\lambda^{\gamma+ \frac{d}2}}\,,
    \end{align*}
    and 
    \begin{align*} 
    \frac{\Tr(-\Delta_{\omega \times (0, \ell)}^{\rm N}-\lambda)_\limminus^\gamma}{L_{\gamma, d}^{\rm sc}\ell \Haus^{d-1}(\omega)\lambda^{\gamma+ \frac{d}2}}
    &=
    \frac{\sum_{k\geq 1}\Tr(-\Delta_{(0, \ell)}^{\rm N}-\lambda+\lambda_k^{\rm N}(\omega))_\limminus^\gamma}{L_{\gamma, d}^{\rm sc}\ell \Haus^{d-1}(\omega)\lambda^{\gamma+ \frac{d}2}}\geq
    \frac{\Tr(-\Delta_{\omega}^{\rm N}-\lambda)_\limminus^{\gamma+\frac{1}2}}{L_{\gamma+\frac{1}2, d-1}^{\rm sc}\Haus^{d-1}(\omega)\lambda^{\gamma+\frac{d}2}}\,.
    \end{align*}
    
    The upper bounds are proved similarly, but using instead the validity of P\'olya's conjecture for the Dirichlet Laplacian on an interval. This completes the proof of the lemma.
\end{proof}

\begin{proof}[Proof of Proposition \ref{dimred r}]
    By the definition of $r_{\gamma, d}^{\sharp}$ together with Lemma~\ref{lem: cylinder lift}, for any open, convex, bounded $\omega \subset \R^{d-1}$ and any $\lambda , \ell >0$ we have that
\begin{align*}
     r_{\gamma, d}^{\rm D} 
    &\geq 
    \frac{\Tr(-\Delta_{\omega}^{\rm D}-\lambda)_\limminus^{\gamma+\frac{1}2}}{L_{\gamma+\frac{1}2, d-1}^{\rm sc}\Haus^{d-1}(\omega)\lambda^{\gamma+\frac{d}2}}-\frac{\Tr(-\Delta_\omega^{\rm D}-\lambda)_\limminus^\gamma}{L_{\gamma, d}^{\rm sc}\ell \Haus^{d-1}(\omega)\lambda^{\gamma+ \frac{d}2}}\,,
    \\
    r_{\gamma, d}^{\rm N} 
    &\leq 
    \frac{\Tr(-\Delta_{\omega}^{\rm N}-\lambda)_\limminus^{\gamma+\frac{1}2}}{L_{\gamma+\frac{1}2, d-1}^{\rm sc}\Haus^{d-1}(\omega)\lambda^{\gamma+\frac{d}2}}+\frac{\Tr(-\Delta_\omega^{\rm N}-\lambda)_\limminus^\gamma}{L_{\gamma, d}^{\rm sc}\ell \Haus^{d-1}(\omega)\lambda^{\gamma+ \frac{d}2}}\,.
\end{align*}
By taking the limit $\ell \to \infty$ we deduce that 
\begin{align*}
    \Tr(-\Delta_{\omega}^{\rm D}-\lambda)_\limminus^{\gamma + \frac12} &\leq r_{\gamma, d}^{\rm D} L_{\gamma+\frac{1}2, d-1}^{\rm sc}\Haus^{d-1}(\omega)\lambda^{\gamma+\frac{d}2}\,,\\
    \Tr(-\Delta_{\omega}^{\rm N}-\lambda)_\limminus^{\gamma + \frac12} &\geq r_{\gamma, d}^{\rm N} L_{\gamma+\frac{1}2, d-1}^{\rm sc}\Haus^{d-1}(\omega)\lambda^{\gamma+\frac{d}2}\,.
\end{align*}
As $(\omega, \lambda) \in \mathcal{C}_{d-1}\times (0, \infty)$ were arbitrary, this proves the desired inequalities.
\end{proof}

\begin{proof}[Proof of Proposition \ref{dimred}]
To deduce the inequality $\gamma_{d}^\sharp \geq (\gamma_{d-1}^\sharp-\frac{1}2)_\limplus$ we may assume that $\gamma_{d-1}^\sharp >\frac{1}2$, for otherwise the claimed inequality is trivial. If $\gamma+\frac{1}2<\gamma_{d-1}^\sharp$ then the fact that $r_{\gamma,d}^{\sharp}\neq 1$ if and only if $\gamma<\gamma_d^{\sharp}$, combined with the inequalities in Proposition~\ref{dimred r}, implies that $r_{\gamma,d}^{\rm D}>1$ if $\sharp = \rm D$ and $r_{\gamma, d}^{\rm N}<1$ if $\sharp =\rm N$. Using again the fact that $r_{\gamma,d}^{\sharp}\neq 1$ if and only if $\gamma<\gamma_d^{\sharp}$ we deduce that if $\gamma + \frac{1}2<\gamma_{d-1}^\sharp$ then $\gamma <\gamma_{d}^\sharp$. This proves the claimed inequality.
\end{proof}


\subsection{Riesz means below the critical exponent}
\label{sec: below critical}

Before turning to Theorem \ref{thm: Main conclusions gamma_d}, which concerns the optimization problem defining $r_{\gamma,d}^\sharp$ for $\gamma=\gamma_d^\sharp$, we consider the slightly simpler case $\gamma<\gamma_d^\sharp$. 

We shall say that a pair $(\Omega, \lambda) \in \mathcal{C}_d\times (0, \infty)$ is an \emph{equality case} or \emph{extremal} for one of the two inequalities in~\eqref{eq: sharp ineq.s in terms of r} if
\begin{equation*}
    \Tr(-\Delta_\Omega^{\sharp}-\lambda)_\limminus^\gamma  = r_{\gamma,d}^{\sharp}L_{\gamma, d}^{\rm sc}|\Omega|\lambda^{\gamma+\frac{d}2}\,.
\end{equation*}
However, in the case $\gamma=0, \sharp = \rm D$ the above equality can never be valid. Indeed, both the left- and right-hand sides of the inequality $\Tr(-\Delta_\Omega^{\rm D}-\lambda)_\limminus^0  \leq r_{0,d}^{\rm D}L_{0, d}^{\rm sc}|\Omega|\lambda^{\frac{d}2}$ are increasing functions of $\lambda$ but the left-hand side is an integer valued lower semi-continuous step function and the right-hand side is strictly increasing. Therefore, we shall in the case $\gamma=0, \sharp ={\rm D}$ adopt the convention that $(\Omega, \lambda)\in \mathcal{C}_d\times (0, \infty)$ is an equality case or extremal if instead
\begin{equation*}
    \lim_{\epsilon \to 0^\limplus}\Tr(-\Delta_\Omega^{\rm D}-\lambda-\epsilon)_\limminus^0  = r_{0,d}^{\rm D}L_{\gamma, d}^{\rm sc}|\Omega|\lambda^{\frac{d}2}\,.
\end{equation*}
In other words, for equality cases for P\'olya's conjecture it is in the Dirichlet case most natural to consider the upper semi-continuous realization of the counting function.

Before stating and proving our next result, we recall that in Proposition \ref{dimred r} we have shown the inequalities $r_{\gamma,d}^{\rm D}\geq r_{\gamma+\frac12,d-1}^{\rm D}$ and $r_{\gamma,d}^{\rm N}\leq r_{\gamma+\frac12,d-1}^{\rm N}$. The nature of the optimization problem for Riesz means will depend on whether these inequalities are strict or not.

\begin{theorem}\label{thm: compactness via r}
    Fix $d\geq 2$, $\sharp\in\{ \rm D, \rm N\}$ and $0\leq\gamma<\gamma_d^\sharp$. It holds that 
	\begin{enumerate}[label=\textup{(}\hspace{-0.4pt}\alph*\textup{)}]
        \item\label{itm: r conditional compactness} 
        $r_{\gamma, d}^\sharp \neq r_{\gamma+ \frac{1}2,d-1}^\sharp$ if and only if all sequences $\{(\Omega_j, \lambda_j)\}_{j\geq 1} \subset \mathcal{C}_d\times (0, \infty)$ such that
        \begin{equation*}
            \lim_{j\to \infty} \frac{\Tr(-\Delta_{\Omega_j}^{\rm \sharp}-\lambda_j)_\limminus^\gamma}{L_{\gamma,d}^{\rm sc}|\Omega_j|\lambda_j^{\gamma+\frac{d}2}} = r_{\gamma,d}^\sharp
        \end{equation*}
        are precompact in the topology induced by the Hausdorff distance on convex sets modulo the symmetries induced by the scaling $(\Omega, \lambda) \mapsto (t\Omega, t^{-2}\lambda)$, $t>0$, and translation $(\Omega, \lambda) \mapsto (\Omega+x_0, \lambda)$, $x_0 \in \R^d$. In particular, if $r_{\gamma, d}^\sharp \neq r_{\gamma+ \frac{1}2,d-1}^\sharp$, then the corresponding inequality in~\eqref{eq: sharp ineq.s in terms of r} admits an equality case.

        \item\label{itm: r conditional noncompactness} 
        if $r_{\gamma, d}^\sharp = r_{\gamma+ \frac{1}2,d-1}^\sharp$ then there is a bounded convex open non-empty $\omega_*\subset \R^{d-1}$ and a $\lambda_*>0$ such that
        \begin{equation*}
            \Tr(-\Delta_{\omega_*}^{\sharp} -\lambda_*)^{\gamma+\frac{1}2}_\limminus = r_{\gamma+ \frac{1}2,d-1}^\sharp L_{\gamma+ \frac{1}2,d-1}^{\rm sc}\Haus^{d-1}(\omega_*) \lambda_*^{\gamma + \frac{d}2}
        \end{equation*}
        and, if
        $
            \Omega_{j} := \omega_*\times (0, j) \subset \R^d,
        $ 
        then
        \begin{equation*}
            \lim_{j\to \infty} \frac{\Tr(-\Delta_{\Omega_j}^{\rm \sharp}-\lambda_*)_\limminus^\gamma}{L_{\gamma,d}^{\rm sc}|\Omega_j|\lambda^{\gamma+\frac{d}2}} = r_{\gamma,d}^\sharp\,.
        \end{equation*}
        In particular, the sequence $\{(\Omega_j, \lambda_*)\}_{j\geq 1}\subset \mathcal{C}_d\times (0, \infty)$ is an optimizing sequence for $r_{\gamma, d}^\sharp$ which fails to be precompact in the sense of \ref{itm: r conditional compactness}.
    \end{enumerate}
\end{theorem}

\begin{remark}
    We remark that the corresponding optimization problems in the regime $\gamma >\gamma_d^\sharp$ display drastically different behavior. Indeed, the two-term inequalities in Theorem~\ref{thm: improved inequality above critical gamma Dir} and \ref{thm: improved inequality above critical gamma Neu} together with \eqref{eq: inradius bound} and \cite[Theorem 1.2]{FrankLarson_24} show that in this case sequences $\{(\Omega_j, \lambda_j)\}_{j\geq 1}\subset \mathcal{C}_d\times (0, \infty)$ satisfy $\lim_{j\to \infty}r_{\rm in}(\Omega_j)\sqrt{\lambda_j}=\infty$ if and only if
    \begin{equation*}
        \lim_{j\to \infty} \frac{\Tr(-\Delta_{\Omega_j}^{\rm \sharp}-\lambda_j)_\limminus^\gamma}{L_{\gamma,d}^{\rm sc}|\Omega_j|\lambda_j^{\gamma+\frac{d}2}} = 1=r_{\gamma,d}^\sharp\,.
    \end{equation*}
    In particular, the collection of such sequences contains no sequence which is precompact in the sense of \ref{itm: r conditional compactness}.
\end{remark}

\begin{proof}
We split the proof into parts corresponding to the different statements of the theorem. 
\medskip

{\noindent\it Part 1: Proof of \ref{itm: r conditional compactness} for $\sharp = \rm D$ when $r_{\gamma, d}^{\rm D}>r_{\gamma+ \frac{1}2,d-1}^{\rm D}$}. Fix $0 \leq \gamma <\gamma_d^{\rm D}$ and let $\{(\Omega_j, \lambda_j)\}_{j \geq 1} \subset \mathcal{C}_d\times (0, \infty)$ be a maximizing sequence for $r_{\gamma,d}^{\rm D}$ (i.e.\ having the limiting property in the statement). Since the quotient 
\begin{equation*}
    \frac{\Tr(-\Delta_{\Omega_j}^{\rm D}-\lambda_j)_\limminus^\gamma}{L_{\gamma,d}^{\rm sc}|\Omega_j|\lambda_j^{\gamma+ \frac{d}2}}
\end{equation*}
is unaltered if we replace the pair $(\Omega_j, \lambda_j)$ by $(|\Omega_j|^{-\frac{1}d}\Omega_j + x_0, |\Omega_j|^{\frac{2}d}\lambda_j)$ for any $x_0\in \R$ we assume without loss of generality that $|\Omega_j|=1$ and $0\in \Omega_j$ for each $j$. 

By the Hersch--Protter inequality we have that $\lambda_1^{\rm D}(\Omega_j) \geq \frac{\pi^2}{4r_{\rm in}(\Omega_j)^2}$. Therefore, $\liminf_{j\to \infty}r_{\rm in}(\Omega_j)\sqrt{\lambda_j} =0$ would imply that $\lambda_1^{\rm D}(\Omega_j)>\lambda_j$ along a sequence of $j \to \infty$, thus
\begin{equation*}
    r_{\gamma, d}^{\rm D} =\liminf_{j\to \infty} \frac{\Tr(-\Delta_{\Omega_j}^{\rm D}-\lambda_j)_\limminus^\gamma}{L_{\gamma, d}^{\rm sc}|\Omega_j|\lambda_j^{\gamma+ \frac{d}2}} =0\,,
\end{equation*}
which contradicts that $r_{\gamma, d}^{\rm D}>1$. We conclude that
\begin{equation}\label{eq: a priori inradius bound}
    \liminf_{j\to \infty}r_{\rm in}(\Omega_j)\sqrt{\lambda_j}>0\,.
\end{equation}

We next aim to prove that $0 <\liminf_{j\to \infty}\lambda_j< \limsup_{j\to \infty}\lambda_j <\infty$.

\medskip

Assume towards a contradiction that $\liminf_{j\to \infty} \lambda_j =0$. By the Faber--Krahn inequality $\lambda_1^{\rm D}(\Omega_j) \geq |B|^{\frac{2}{d}}\lambda_1^{\rm D}(B)>0$ and so for all $j$ such that $\lambda_j<|B|^{\frac{2}{d}}\lambda_1^{\rm D}(B)$ we have $\Tr(-\Delta_{\Omega_j}^{\rm D}-\lambda_j)_\limminus^\gamma =0$. Therefore, if $\liminf_{j\to \infty} \lambda_j =0$ then
\begin{equation*}
   \liminf_{j\to \infty} \frac{\Tr(-\Delta_{\Omega_j}^{\rm D}-\lambda_j)_\limminus^\gamma}{L_{\gamma,d}^{\rm sc}|\Omega|\lambda_j^{\gamma+ \frac{d}2}} =0\,,
\end{equation*}
this contradicts that $\{(\Omega_j, \lambda_j)\}_{j\geq 1}$ was a maximizing sequence ($r_{\gamma,d}^{\rm D}\geq 1$). Consequently, we have reached the desired conclusion, $\liminf_{j\to \infty} \lambda_j>0$.

\medskip

Assume towards contradiction that $\limsup_{j\to \infty} \lambda_j =\infty$. We split into two cases.

\smallskip

{\it Case 1} $\limsup_{j\to \infty}r_{\rm in}(\Omega_j)\sqrt{\lambda_j}=\infty.$ By \eqref{eq: inradius bound} $\Haus^{d-1}(\partial\Omega_j) \leq \frac{d|\Omega_j|}{r_{\rm in}(\Omega_j)}$ and thus by \cite[Theorem 1.2]{FrankLarson_24} if $\gamma>0$ and Theorem~\ref{thm: quantitative Weyl law} if $\gamma=0$ we conclude that
\begin{equation*}
    \liminf_{j\to \infty} \frac{\Tr(-\Delta_{\Omega_j}^{\rm D}-\lambda_j)_\limminus^\gamma}{L_{\gamma,d}^{\rm sc}|\Omega_j|\lambda_j^{\gamma+ \frac{d}2}} = 1\,.
\end{equation*}
Since $r_{\gamma, d}^{\rm D}>1$ as $\gamma <\gamma_d^{\rm D}$ this contradicts the assumption that $\{(\Omega_j, \lambda_j)\}_{j\geq 1}$ was a maximizing sequence.

\smallskip

{\it Case 2} $\limsup_{j\to \infty} r_{\rm in}(\Omega_j)\sqrt{\lambda_j}<\infty$. By~\eqref{eq: a priori inradius bound}, the assumption, and since $|\Omega_j|=1$, Theorem~\ref{thm: Asymptotics degenerating convex sets} implies that there exists an open non-empty, bounded, and convex set $\Omega_*'\subset \R^d$ and an integer $1\leq m\leq d-1$ such that
	$$
	r_{\gamma, d}^{\rm D}=\limsup_{j\to\infty} \frac{\Tr(-\Delta_{\Omega_j}^{\rm D}-\lambda_j)_\limminus^{\gamma_d^{\rm D}}}{L_{\gamma_d^{\rm D}, d}^{\rm sc}\lambda_j^{\gamma_d^{\rm D}+\frac d2} |\Omega_j|} = \frac{1}{L^{\rm sc}_{\gamma_d^{\rm D}+ \frac{d-m}{2}, m}|\Omega_*'|}\int_{P^\perp \Omega_*'} \Tr(-\Delta_{\Omega_*'(y)}^{\rm D}-1)_\limminus^{\gamma_d^{\rm D}+ \frac{d-m}{2}}\,dy \,.
	$$
    By definition of $r_{\gamma, d}^{\rm D}$ the trace in the integral is bounded by $r_{\gamma+ \frac{d-m}2,m}^{\rm sc}L_{\gamma+ \frac{d-m}2,m}^{\rm sc}\Haus^{m}(\Omega_*'(y))$ for each $y \in P^\perp \Omega_*'$. Since
    \begin{equation*}
        \frac{1}{|\Omega_*'|}\int_{P^\perp \Omega_*'}\Haus^m(\Omega_*'(y))\,dy = 1 \,,
    \end{equation*}
    it follows that
    \begin{equation}\label{eq: reverse inequality rs}
        r_{\gamma, d}^{\rm D} \leq r_{\gamma+ \frac{d-m}2,m}^{\rm D}\,.
    \end{equation}
    Since $\gamma + \frac{d-m}2 \geq 1$ for $m\leq d-2$, Theorem~\ref{thm: Extended range of semiclassical ineq Dir} and the fact that $r_{\gamma',d'}^{\rm D}=1$ if $\gamma' \geq \gamma_{d'}^{\rm D}$ implies that $r_{\gamma+ \frac{d-m}2,m}^{\rm D}=1$ for $m \leq d-2$, in which case \eqref{eq: reverse inequality rs} contradicts that $r_{\gamma,d}^{\rm D}>1$. In the remaining case $m=d-1$ the inequality \eqref{eq: reverse inequality rs} is contradicted by the assumption $r_{\gamma, d}^{\rm D} > r_{\gamma+ \frac{1}2,d-1}^{\rm D}$.

    \medskip

    Combining the two cases we finally conclude that $\limsup_{j\to \infty}\lambda_j <\infty$.

    We have shown that $0 <\liminf_{j\to \infty}\lambda_j< \limsup_{j\to \infty}\lambda_j <\infty$. Therefore, by passing to a subsequence we assume without loss of generality that $\lim_{j \to \infty} \lambda_j = \lambda_*>0$. By~\eqref{eq: a priori inradius bound} it follows that $\liminf_{j\to \infty}r_{\rm in}(\Omega_j)>0$. Consequently, $\Omega_j$ satisfies that $|\Omega_j|=1$, $r_{\rm in}(\Omega_j)\geq \frac{1}2\liminf_{k\to \infty}r_{\rm in}(\Omega_k)>0$ for all $j$ sufficiently large. In particular, by~\eqref{eq: inradius bound} the sets $\Omega_j$ are uniformly bounded. The Blaschke selection theorem allows us to extract a subsequence along which $\Omega_j$ converges to an open, convex, non-empty and bounded set $\Omega_*$ with respect to the Hausdorff distance. This completes the proof of precompactness of maximizing sequences.

    \medskip

    For $\gamma>0$ the existence of $(\Omega, \lambda)\in \mathcal{C}_d\times(0, \infty)$ that realizes the supremum defining $r_{\gamma,d}^{\rm D}$ follows from  the precompactness of maximizing sequences above and the joint continuity $(\Omega, \lambda) \mapsto \frac{\Tr(-\Delta_{\Omega}^{\rm D}-\lambda)_\limminus^{\gamma}}{L_{\gamma,d}^{\rm sc}|\Omega|\lambda^{\gamma+ \frac{d}2}}$ (see Proposition~\ref{prop: DN continuity of eigenvalues}). For $\gamma=0$ the map $(\Omega, \lambda) \mapsto \frac{\Tr(-\Delta_{\Omega}^{\rm D}-\lambda)_\limminus^{0}}{L_{0,d}^{\rm sc}|\Omega|\lambda^{\frac{d}2}}$ is only semi-continuous (see Proposition~\ref{prop: DN continuity of eigenvalues}) and for $(\Omega_*, \lambda_*) \in \mathcal{C}_d\times (0, \infty)$ the limit of a maximizing sequence we only know that
    \begin{equation*}
        \lim_{\epsilon \to 0^+} \Tr(-\Delta_{\Omega_*}^{\rm D}-\lambda_*-\epsilon)_\limminus^0 = L_{0, d}^{\rm sc}|\Omega_*|\lambda_*^{\frac{d}2}\,.
    \end{equation*}
    Thus $(\Omega_*, \lambda_*)$ is an equality case in the sense discussed before the statement of the theorem. 

   \medskip
{\noindent\it Part 2: Proof of \ref{itm: r conditional compactness} for $\sharp = \rm N$ when $r_{\gamma, d}^{\rm N}<r_{\gamma+ \frac{1}2,d-1}^{\rm N}$}. Fix $0 \leq \gamma <\gamma_d^{\rm N}$ and let $\{(\Omega_j, \lambda_j)\}_{j \geq 1} \subset \mathcal{C}_d\times (0, \infty)$ be a minimizing sequence for $r_{\gamma,d}^{\rm N}$. As in the Dirichlet case we assume without loss of generality that $|\Omega_j|=1$ and $0\in \Omega_j$ for each $j$. 

By the Lemma~\ref{lem: small energy improved Kroger} we have that $\Tr(-\Delta_{\Omega_j}^{\rm N}-\lambda_j)_\limminus^\gamma \geq \frac{C_{\gamma,d}}{r_{\rm in}(\Omega_j)}|\Omega_j|\lambda_j^{\gamma+ \frac{d-1}{2}}$. Therefore, $\liminf_{j\to \infty}r_{\rm in}(\Omega_j)\sqrt{\lambda_j} =0$ would imply that
\begin{equation*}
    r_{\gamma, d}^{\rm N} =\lim_{j\to \infty} \frac{\Tr(-\Delta_{\Omega_j}^{\rm N}-\lambda_j)_\limminus^\gamma}{L_{\gamma, d}^{\rm sc}|\Omega_j|\lambda_j^{\gamma+ \frac{d}2}}  \geq \limsup_{j\to \infty} \frac{C_{\gamma,d}}{L_{\gamma, d}^{\rm sc}r_{\rm in}(\Omega_j)\sqrt{\lambda_j}}=\infty\,,
\end{equation*}
which contradicts that $r_{\gamma, d}^{\rm N}<1$. We conclude that
\begin{equation}\label{eq: a priori inradius bound N}
    \liminf_{j\to \infty}r_{\rm in}(\Omega_j)\sqrt{\lambda_j}>0\,.
\end{equation}

Again we want to prove that $0 <\liminf_{j\to \infty}\lambda_j< \limsup_{j\to \infty}\lambda_j <\infty$.

\medskip

Assume towards contradiction that $\liminf_{j\to \infty} \lambda_j =0$. Since $0$ is an eigenvalue of the Neumann Laplacian we have $\Tr(-\Delta_{\Omega_j}^{\rm N}-\lambda_j)_\limminus^\gamma \geq \lambda_j^{\gamma}$ and so, if $\liminf_{j\to \infty}\lambda_j=0$,
\begin{equation*}
   \limsup_{j\to \infty} \frac{\Tr(-\Delta_{\Omega_j}^{\rm N}-\lambda_j)_\limminus^\gamma}{L_{\gamma,d}^{\rm sc}|\Omega|\lambda_j^{\gamma+ \frac{d}2}} =\infty\,,
\end{equation*}
this contradicts that $\{(\Omega_j, \lambda_j)\}_{j\geq 1}$ was a minimizing sequence ($r_{\gamma,d}^{\rm D}\leq 1$). Therefore, we have reached the desired conclusion, $\liminf_{j\to \infty} \lambda_j>0$.

\medskip

Assume towards contradiction that $\limsup_{j\to \infty} \lambda_j =\infty$. We split into two cases.

\smallskip

{\it Case 1} $\limsup_{j\to \infty}r_{\rm in}(\Omega_j)\sqrt{\lambda_j}=\infty.$ By \eqref{eq: inradius bound} $\Haus^{d-1}(\partial\Omega_j) \leq \frac{d|\Omega_j|}{r_{\rm in}(\Omega_j)}$ and thus by \cite[Theorem 1.2]{FrankLarson_24} if $\gamma>0$ and Theorem~\ref{thm: quantitative Weyl law} if $\gamma=0$ we conclude that
\begin{equation*}
    \limsup_{j\to \infty} \frac{\Tr(-\Delta_{\Omega_j}^{\rm N}-\lambda_j)_\limminus^\gamma}{L_{\gamma,d}^{\rm sc}|\Omega_j|\lambda_j^{\gamma+ \frac{d}2}} = 1\,.
\end{equation*}
Since $r_{\gamma, d}^{\rm N}<1$ as $\gamma <\gamma_d^{\rm N}$ this contradicts the assumption that $\{(\Omega_j, \lambda_j)\}_{j\geq 1}$ was a minimizing sequence.

\smallskip

{\it Case 2} $\limsup_{j\to \infty} r_{\rm in}(\Omega_j)\sqrt{\lambda_j}<\infty$. By~\eqref{eq: a priori inradius bound N}, the assumption, and since $|\Omega_j|=1$, Theorem~\ref{thm: Asymptotics degenerating convex sets} implies that there exists an open non-empty, bounded, and convex set $\Omega_*'\subset \R^d$ and an integer $1\leq m\leq d-1$ such that
	$$
	r_{\gamma, d}^{\rm N}=\limsup_{j\to\infty} \frac{\Tr(-\Delta_{\Omega_j}^{\rm N}-\lambda_j)_\limminus^{\gamma_d^{\rm N}}}{L_{\gamma_d^{\rm N}, d}^{\rm sc}\lambda_j^{\gamma_d^{\rm N}+\frac d2} |\Omega_j|} = \frac{1}{L^{\rm sc}_{\gamma_d^{\rm N}+ \frac{d-m}{2}, m}|\Omega_*'|}\int_{P^\perp \Omega_*'} \Tr(-\Delta_{\Omega_*'(y)}^{\rm N}-1)_\limminus^{\gamma_d^{\rm N}+ \frac{d-m}{2}}\,dy \,.
	$$
    By arguing as in the Dirichlet case this leads to a contradiction. 

    \medskip

    Combining the two cases we finally conclude that $\limsup_{j\to \infty}\lambda_j <\infty$.

    We have shown that $0 <\liminf_{j\to \infty}\lambda_j< \limsup_{j\to \infty}\lambda_j <\infty$. Therefore, by passing to a subsequence we assume without loss of generality that $\lim_{j \to \infty} \lambda_j = \lambda_*>0$. By~\eqref{eq: a priori inradius bound N}, we thus have $\liminf_{j\to \infty}r_{\rm in}(\Omega_j)>0$.
    Consequently, $\Omega_j$ satisfies that $|\Omega_j|=1$, $r_{\rm in}(\Omega_j)\geq \frac{1}2\liminf_{k\to \infty}r_{\rm in}(\Omega_k)>0$ for all $j$ sufficiently large. As in the Dirichlet case, the Blaschke selection theorem allows us to extract a subsequence along which $\Omega_j$ converges to a convex non-empty and bounded set $\Omega_*$ with respect to the Hausdorff distance. 
    This completes the proof of precompactness of minimizing sequences.
\medskip

    The existence of $(\Omega, \lambda)\in \mathcal{C}_d\times(0, \infty)$ that realizes the infimum defining $r_{\gamma,d}^{\rm N}$ follows from  the precompactness of minimizing sequences above and the joint lower semi-continuity of $(\Omega, \lambda) \mapsto \frac{\Tr(-\Delta_{\Omega}^{\rm N}-\lambda)_\limminus^{\gamma}}{L_{\gamma,d}^{\rm sc}|\Omega|\lambda^{\gamma+ \frac{d}2}}$ (see Proposition~\ref{prop: DN continuity of eigenvalues}).

\medskip
{\noindent\it Part 3: Proof of \ref{itm: r conditional noncompactness}}. Next, we verify the existence of non-compact optimizing sequences in the case $\gamma<\gamma_d^{\sharp}$ and $r_{\gamma,d}^\sharp=r_{\gamma+ \frac{1}2,d-1}^\sharp$. (Note that this will also prove the remaining implication in part \ref{itm: r conditional compactness}.)

Since $\gamma_1^\sharp=0$ there is nothing to prove for $d\leq 2$, as $r_{0,1}^{\sharp}=1$ and $\gamma <\gamma_{d}^{\sharp}$ implies $r_{\gamma, d}^{\rm D}>1$ and $r_{\gamma, d}^{\rm N}<1$. Thus, let $d\geq 3$. Since $\gamma_d^{\sharp}<1$ for all $d$ by Theorems \ref{thm: Extended range of semiclassical ineq Dir} and \ref{thm: Extended range of semiclassical ineq Neu}, we conclude that under the assumptions on $\gamma$ it holds by the inequality in Proposition~\ref{dimred r} that
\begin{equation*}
    r_{\gamma, d}^{\rm D} = r_{\gamma+ \frac{1}2,d-1}^{\rm D} >r_{\gamma +1, d-2}^{\rm D} =1 \quad \mbox{and}\quad r_{\gamma, d}^{\rm N} = r_{\gamma+ \frac{1}2,d-1}^{\rm N} <r_{\gamma +1, d-2}^{\rm N} =1\,.
\end{equation*}
By what we just proved there exists an optimizing pair $(\omega_*, \lambda_*)\in \mathcal{C}_{d-1}\times (0, \infty)$ for the variational problem defining $r_{\gamma+ \frac{1}2,d-1}^{\sharp}$. The sequence $\{(\omega_* \times (0, j), \lambda_*)\}_{j\geq 1} \subset \mathcal{C}_d\times (0, \infty)$ is an optimizing sequence for $r_{\gamma,d}^{\sharp}$ since by Lemma~\ref{lem: cylinder lift} and the assumption $r_{\gamma,d}^\sharp=r_{\gamma+\frac12,d-1}^\sharp$ we have
\begin{align*}
    \Biggl| \frac{\Tr(-\Delta_{\omega_*\times (0,j)}^{\rm \sharp}-\lambda_*)_\limminus^\gamma}{L_{\gamma,d}^{\rm sc}|\omega_*\times(0,j)|\lambda_*^{\gamma+\frac{d}2}}- r_{\gamma,d}^{\sharp}\Biggr| \leq j^{-1}\frac{\Tr(-\Delta_{\omega_*}^{\sharp}-\lambda_*)_\limminus^\gamma}{L_{\gamma,d}^{\rm sc}\Haus^{d-1}(\omega_*)\lambda_*^{\gamma + \frac{d}2}}\,.
\end{align*}
Furthermore, as $|\omega_*\times (0, j)| \to \infty$ but $r_{\rm in}(\omega_*\times (0, j)) \to r_{\rm in}(\omega_*)<\infty$, this optimizing sequence has no converging subsequence modulo the symmetries of the problem.
\end{proof}


\subsection{Riesz means at the critical exponent}

Our goal in this subsection is to prove Theorem \ref{thm: Main conclusions gamma_d}. There are two main ingredients in the proof: First, the characterization of $\gamma_d^\sharp$ in terms of semiclassical two-term inequalities (Theorems \ref{thm: improved inequality above critical gamma Dir} and \ref{thm: improved inequality above critical gamma Neu}) and second, the (partial) Weyl asymptotics on collapsing convex sets (Theorem \ref{thm: Asymptotics degenerating convex sets}).

The overall strategy of the proof of Theorem \ref{thm: Main conclusions gamma_d} is similar in the Dirichlet and Neumann cases, but there are differences in carrying out the individual steps. We will provide full details in the Dirichlet case and only focus on the differences in the Neumann case. 

\begin{proof}[Proof of Theorem~\ref{thm: Main conclusions gamma_d}]
	The proof is split into parts according to the two statements in the theorem. We begin by proving \ref{itm: Main thm existence sup} first for $\sharp = \rm D$ and then for $\sharp = \rm N$. After \ref{itm: Main thm existence sup} is established, we prove that the validity of \ref{itm: Main thm cylinder sequence} in dimension $d$ follows from the validity of \ref{itm: Main thm existence sup} in dimension $d-1$ by applying Lemma~\ref{lem: cylinder lift}.
 
    \medskip
 
     {\noindent \it Proof of \ref{itm: Main thm existence sup} for $\sharp = {\rm D}$}: By assumption we have $\gamma_d^{\rm D} > (\gamma_{d-1}^{\rm D}-\frac{1}2)_\limplus \geq 0$. In particular, according to Theorem \ref{thm: improved inequality above critical gamma Dir} there is no two-term inequality at $\gamma_d^{\rm D}$. Consequently, there exists a sequence $\{\lambda_j\}_{j\geq 1}\subset (0, \infty)$ and a sequence of non-empty, bounded, convex sets $\{\Omega_j\}_{j\geq 1}$ such that
	\begin{equation}\label{eq: failure of improved ineq}
		\Tr(-\Delta_{\Omega_j}^{\rm D}-\lambda_j)_\limminus^{\gamma_d^{\rm D}} > \Bigl(L_{\gamma_d^{\rm D}, d}^{\rm sc} |\Omega_j|\lambda_j^{\gamma_d^{\rm D}+\frac{d}{2}}- j^{-1}\Haus^{d-1}(\partial\Omega_j)\lambda_j^{\gamma_d^{\rm D}+\frac{d-1}{2}}\Bigr)_\limplus \,.
	\end{equation}
	By scaling and translation we may without loss of generality assume $|\Omega_j|=1$ and $0 \in \Omega_j$ for each $j$.
	Our aim is to show that both these sequences must have well-defined limits $\lambda_*, \Omega_*$, and that these limits constitute an equality case as in the statement of the theorem. Keys to proving that these limits exist are bounds for $r_{\rm in}(\Omega_j)$ and $\lambda_j$, which we now prove.
    
    \medskip
    
    \emph{Step 1. Bounds on the inradius.} Let us show that
    \begin{equation}
        \label{eq:inradd}
        0< \inf_j r_{\rm in}(\Omega_j) \sqrt{\lambda_j} \leq \sup_j r_{\rm in}(\Omega_j) \sqrt{\lambda_j} <\infty \,.
    \end{equation}
    Since the right-hand side in~\eqref{eq: failure of improved ineq} is non-negative, the left-hand side is positive. Therefore, by the Hersch--Protter inequality we have $r_{\rm in}(\Omega_j) \sqrt{\lambda_j} \geq \frac{\pi}{2}$ for all $j$, which proves the left inequality in \eqref{eq:inradd}.
    By~\eqref{eq: inradius bound}, $|\Omega| \geq d^{-1} r_{\rm in}(\Omega) \mathcal H^{d-1}(\partial\Omega)$ for any convex set $\Omega\subset \R^d$ and so
	\begin{equation}\label{eq: Dir bound on second term}
	 |\Omega_j| \lambda_j^{\gamma_d^{\rm D}+\frac d2}
	\geq d^{-1} r_{\rm in}(\Omega_j) \mathcal H^{d-1}(\partial\Omega_j) \lambda_j^{\gamma_d^{\rm D}+\frac d2} \geq d^{-1}  \frac\pi 2 \mathcal H^{d-1}(\partial\Omega_j) \lambda_j^{\gamma_d^{\rm D}+\frac {d-1}2} \,,
	\end{equation}
	and therefore, for all sufficiently large $j$, 
	$$
		L_{\gamma_d^{\rm D},d}^{{\rm sc}} |\Omega_j| \lambda_j^{\gamma_d^{\rm D}+\frac d2}\geq  j^{-1} \mathcal H^{d-1}(\partial\Omega_j) \lambda_j^{\gamma+\frac{d-1}{2}}\,.
	$$
	It follows that the positive part in~\eqref{eq: failure of improved ineq} can be dropped for $j$ large enough, so that
	$$
	\Tr(-\Delta_{\Omega_j}^{\rm D}-\lambda_j)_\limminus^{\gamma_d^{\rm D}} > L_{\gamma_d^{\rm D},d}^{{\rm sc}} |\Omega_j| \lambda_j^{\gamma_d^{\rm D}+\frac d2} - j^{-1} \mathcal H^{d-1}(\partial\Omega_j) \lambda_j^{\gamma_d^{\rm D}+\frac{d-1}{2}} \,.
	$$
	Meanwhile, by \cite[Theorem~1.2]{FrankLarson_24}
	\begin{align*}
		\Tr(-\Delta_{\Omega_j}^{\rm D}-\lambda_j)_\limminus^{\gamma_d^{\rm D}} &\leq L_{\gamma_d^{\rm D},d}^{{\rm sc}} |\Omega_j| \lambda_j^{\gamma_d^{\rm D}+\frac d2} - \tfrac14 L_{\gamma_d^{\rm D},d-1}^{{\rm sc}} \mathcal H^{d-1}(\partial\Omega_j) \lambda_j^{\gamma_d^{\rm D}+\frac{d-1}{2}} \\
		&\quad 
	+ C_{\gamma_d^{\rm D},d} \mathcal H^{d-1}(\partial\Omega_j) \lambda_j^{\gamma_d^{\rm D}+\frac{d-1}{2}} ( r_{\rm in}(\Omega_j) \sqrt{\lambda_j})^{-\frac{\gamma_d^{\rm D}}{22}} \,.
	\end{align*}
	
	Combining this with the lower bound, we deduce that
	\begin{align*}
		- j^{-1} \mathcal H^{d-1}(\partial\Omega_j) \lambda_j^{\gamma_d^{\rm D}+\frac{d-1}{2}} &< - \tfrac14 L_{\gamma_d^{\rm D},d-1}^{{\rm sc}} \mathcal H^{d-1}(\partial\Omega_j) \lambda_j^{\gamma_d^{\rm D}+\frac{d-1}{2}} \\
		&\quad
	+ C_{\gamma_d^{\rm D},d} \mathcal H^{d-1}(\partial\Omega_j) \lambda_j^{\gamma_d^{\rm D}+\frac{d-1}{2}} ( r_{\rm in}(\Omega_j) \sqrt{\lambda_j})^{-\frac{\gamma_d^{\rm D}}{22}} \,,
	\end{align*}
	that is,
	$$
	- j^{-1} < - \tfrac14 L_{\gamma_d^{\rm D},d-1}^{{\rm sc}} + C_{\gamma_d^{\rm D},d} ( r_{\textup{in}}(\Omega_j) \sqrt{\lambda_j})^{-\frac{\gamma_d^{\rm D}}{22}} \,.
	$$
    This proves the right inequality in \eqref{eq:inradd}.

    Before proceeding, let us note that the above arguments imply that
    \begin{equation}\label{eq: limsup bound by def of gammad}
		\lim_{j\to \infty} \frac{\Tr(-\Delta_{\Omega_j}^{\rm D}-\lambda_j)_\limminus^{\gamma_d^{\rm D}}}{L_{\gamma_d^{\rm D}, d}^{\rm sc} |\Omega_j|\lambda_j^{\gamma_d^{\rm D}+\frac{d}{2}}} = 1\,.
	\end{equation}
    Indeed, the fact that the limit is $\leq 1$ follows from the definition of $\gamma_d^{\rm D}$ and the fact that the limit is $\geq 1$ follows from \eqref{eq: failure of improved ineq} and \eqref{eq: Dir bound on second term}.

    \medskip

    \emph{Step 2. Bounds on $\lambda_j$.} Let us show that
    \begin{equation}
        \label{eq:lambdad}
        0< \inf_j \lambda_j \leq \sup_j \lambda_j <\infty \,.
    \end{equation}
    Since the right-hand side in~\eqref{eq: failure of improved ineq} is non-negative, the left-hand side is positive. Therefore, by the Faber--Krahn inequality we have $\lambda_j \geq \lambda_1(B_1)|B_1|^{-2/d}$, which proves the left inequality in \eqref{eq:lambdad}. To prove the right inequality we argue by contradiction and assume that $\limsup_{j\to \infty} \lambda_j =\infty$. Our goal will be to reach a contradiction to \eqref{eq: limsup bound by def of gammad}.
	
    Because of $\limsup_{j\to \infty} |\Omega_j|\lambda_j^{d/2}=\infty$ (by the contradiction assumption) and \eqref{eq:inradd}, Theorem~\ref{thm: Asymptotics degenerating convex sets} implies that there exists an open non-empty, bounded, and convex set $\Omega_*'\subset \R^d$ and an integer $1\leq m\leq d-1$ such that
	$$
	\limsup_{j\to\infty} \frac{\Tr(-\Delta_{\Omega_j}^{\rm D}-\lambda_j)_\limminus^{\gamma_d^{\rm D}}}{L_{\gamma_d^{\rm D}, d}^{\rm sc}\lambda_j^{\gamma_d^{\rm D}+\frac d2} |\Omega_j|} = \frac{1}{L^{\rm sc}_{\gamma_d^{\rm D}+ \frac{d-m}{2}, m}|\Omega_*'|}\int_{P^\perp \Omega_*'} \Tr(-\Delta_{\Omega_*'(y)}^{\rm D}-1)_\limminus^{\gamma_d^{\rm D}+ \frac{d-m}{2}}\,dy \,.
	$$
    Independently of the value of $m$, our assumptions imply that $\gamma_d^{\rm D} + \frac{d-m}{2}>\gamma_m^{\rm D}$. Indeed, if $1\leq m \leq d-2$ then $\gamma_{d}^{\rm D}+ \frac{d-m}{2}> 1$ and therefore $\gamma_{d}^{\rm D}+ \frac{d-m}{2}>\gamma_m^{\rm D}$ by the Berezin--Li--Yau inequality. If instead $m = d-1$ then the assumption $\gamma_d^{\rm D}>(\gamma_{d-1}^{\rm D}-\frac{1}{2})_\limplus$ entails that $\gamma_d^{\rm D} + \frac{1}{2}> \gamma_{d-1}^{\rm D}$. 
    
    By Theorem~\ref{thm: improved inequality above critical gamma Dir}, the inequality $\gamma_d^{\rm D} + \frac{d-m}{2}>\gamma_m^{\rm D}$ implies that for any $y\in P^\perp \Omega_*'$ we have
	\begin{equation*}
		\Tr(-\Delta_{\Omega_*'(y)}^{\rm D}-1)_\limminus^{\gamma_d^{\rm D} + \frac{d-m}{2}} < L_{\gamma_d^{\rm D}+ \frac{d-m}{2},m}^{\rm sc} \Haus^m(\Omega_*'(y))\,.
	\end{equation*}
	Consequently, using
	\begin{equation*}
		 \frac{1}{|\Omega_*'|}\int_{P^\perp \Omega_*'} \Haus^m(\Omega_*'(y))\,dy = 1 \,,
	\end{equation*}
	it follows that
	$$
	\limsup_{j\to\infty} \frac{\Tr(-\Delta^{\rm D}_{\Omega_j}-\lambda_j)_\limminus^{\gamma_d^{\rm D}}}{L_{\gamma_d^{\rm D}, d}^{\rm sc}\lambda_j^{\gamma_d^{\rm D}+\frac d2} |\Omega_j|}<1\,,
	$$
	which contradicts~\eqref{eq: limsup bound by def of gammad}. This proves the right inequality in \eqref{eq:lambdad}.

    \medskip

    We can now complete the proof of \ref{itm: Main thm existence sup} for $\sharp ={\rm D}$. By \eqref{eq:lambdad} we can pass to a subsequence along which $\lambda_j$ has a non-zero limit $\lambda_*$. Along this subsequence $|\Omega_j|=1$, $\inf_{j} r_{\rm in}(\Omega_j) > 0$ and $0 \in \Omega_j$. Therefore, by \eqref{eq: inradius bound} and the Blaschke selection theorem we can extract a further subsequence along which $\Omega_j$ converges to an open, non-empty, bounded, and convex set $\Omega_*$ with respect to the Hausdorff distance.

	Recall that the maps $\Omega \mapsto |\Omega|$ and $(\lambda, \Omega) \to \Tr(-\Delta_\Omega^{\rm D}-\lambda)_\limminus^{\gamma}$, for $\gamma>0$, are continuous with respect to Hausdorff convergence of convex sets (see \cite{HenrotPierre_18} and Proposition~\ref{prop: DN continuity of eigenvalues} below). By assumption $\gamma_d^{\rm D}>0$, therefore
    \begin{equation*}
        \lim_{j\to \infty}\Tr(-\Delta_{\Omega_j}^{\rm D}-\lambda_j)_\limminus^{\gamma_d^{\rm D}}= \Tr(-\Delta_{\Omega_*}^{\rm D}-\lambda_*)_\limminus^{\gamma_d^{\rm D}}\,.
    \end{equation*}
    The definition of $\gamma_d^{\rm D}$ together with~\eqref{eq: failure of improved ineq} yields
	\begin{equation*}
		\Tr(-\Delta_{\Omega_*}^{\rm D}-\lambda_*)_\limminus^{\gamma_d^{\rm D}} = L_{\gamma_d^{\rm D}, d}^{\rm sc} |\Omega_*|\lambda_*^{\gamma_d^{\rm D}+\frac{d}{2}}\,.
	\end{equation*}
    This completes the proof of \ref{itm: Main thm existence sup} for $\sharp ={\rm D}$ and any $d\geq 2$.

\medskip

 {\noindent \it Proof of \ref{itm: Main thm existence sup} for $\sharp ={\rm N}$}: We now turn to the Neumann case. The basic strategy is similar as in the Dirichlet case and we will focus on the differences. By assumption we have $\gamma_d^{\rm N} > (\gamma_{d-1}^{\rm N}-\frac{1}2)_\limplus \geq 0$. In particular, according to Theorem \ref{thm: improved inequality above critical gamma Neu} there is no two-term inequality at $\gamma_d^{\rm N}$. Consequently there exists a sequence $\{\lambda_j\}_{j\geq 1}\subset (0, \infty)$ and a sequence of non-empty, bounded, convex sets $\{\Omega_j\}_{j\geq 1}$ satisfying
	\begin{equation}\label{eq: failure of improved ineq Neu}
		\Tr(-\Delta_{\Omega_j}^{\rm N}-\lambda_j)_\limminus^{\gamma_d^{\rm N}} < L_{\gamma_d^{\rm N}, d}^{\rm sc} |\Omega_j|\lambda_j^{\gamma_d^{\rm N}+\frac{d}{2}}+ j^{-1}\Haus^{d-1}(\partial\Omega_j)\lambda_j^{\gamma_d^{\rm N}+\frac{d-1}{2}}
	\end{equation}
	After scaling and translating we may assume without loss of generality that $|\Omega_j|=1$ and $0 \in \Omega_j$ for each $j$.
	As in the Dirichlet case we shall show that these sequences have limits $\lambda_*, \Omega_*$ that attain the equality claimed in the theorem and again a key step will be to establish bounds on $r_{\rm in}(\Omega_j)$ and $\lambda_j$.

    \medskip

    \emph{Step 1. Bounds on the inradius.} Let us show that the bounds \eqref{eq:inradd} hold in the present situation as well. By Lemma~\ref{lem: small energy improved Kroger}, \eqref{eq: failure of improved ineq Neu}, and \eqref{eq: inradius bound} we have
    \begin{align*}
        C_{\gamma_d^{\rm N},d} \frac{|\Omega_j|}{r_{\rm in}(\Omega_j)}\lambda_j^{\gamma^{\rm N}_d+ \frac{d-1}2} &<L_{\gamma_d^{\rm N}, d}^{\rm sc} |\Omega_j|\lambda_j^{\gamma_d^{\rm N}+\frac{d}{2}}+ j^{-1}\Haus^{d-1}(\partial\Omega_j)\lambda_j^{\gamma_d^{\rm N}+\frac{d-1}{2}}\\
        &\leq L_{\gamma_d^{\rm N}, d}^{\rm sc} |\Omega_j|\lambda_j^{\gamma_d^{\rm N}+\frac{d}{2}}+ j^{-1} \frac{d|\Omega_j|}{r_{\rm in}(\Omega_j)}\lambda_j^{\gamma_d^{\rm N}+\frac{d-1}{2}} \,.
    \end{align*}
    After rearranging the terms we deduce the left inequality in \eqref{eq:inradd}. To prove the right inequality, we note that by \cite[Theorem~1.2]{FrankLarson_24}, and since $\gamma_d^{\rm N}<1$ by Theorem~\ref{thm: Extended range of semiclassical ineq Neu}, we have
	\begin{align*}
		\Tr(-&\Delta^{\rm N}_{\Omega_j}-\lambda_j)_\limminus^{\gamma_d^{\rm N}} \\
        &\geq L_{\gamma_d^{\rm N},d}^{{\rm sc}} |\Omega_j| \lambda_j^{\gamma_d^{\rm N}+\frac d2} 
        + \tfrac14 L_{\gamma_d^{\rm N},d-1}^{{\rm sc}} \mathcal H^{d-1}(\partial\Omega_j) \lambda_j^{\gamma_d^{\rm N}+\frac{d-1}{2}} \\
		&\quad 
	- C_{\gamma_d^{\rm N},d} \mathcal H^{d-1}(\partial\Omega_j) \lambda_j^{\gamma_d^{\rm N}+\frac{d-1}{2}} \Bigl[\bigl(1+\ln_\limplus\bigl(r_{\rm in}(\Omega)\sqrt{\lambda_j}\bigr)\bigr)^{-\frac{\gamma^{\rm N}_d}{2}}+\bigl(r_{\rm in}(\Omega_j) \sqrt{\lambda_j}\bigr)^{1-d}\Bigr]  \,.
	\end{align*}
	Combining this with the upper bound~\eqref{eq: failure of improved ineq Neu}, we deduce that
	\begin{align*}
		j^{-1} &> \tfrac14 L_{\gamma_d^{\rm N},d-1}^{{\rm sc}}  
	- C_{\gamma_d^{\rm N},d}  \Bigl[\bigl(1+\ln_\limplus\bigl(r_{\rm in}(\Omega)\sqrt{\lambda_j}\bigr)\bigr)^{-\frac{\gamma^{\rm N}_d}{2}}+\bigl(r_{\rm in}(\Omega_j) \sqrt{\lambda_j}\bigr)^{1-d}\Bigr] \,.
	\end{align*}
    This proves the right inequality in \eqref{eq:inradd}.
 
    Before proceeding, let us note that the above arguments imply that
    \begin{equation}\label{eq: Weyl law extremal sequence Neu}
        \lim_{j\to \infty} \frac{\Tr(-\Delta_{\Omega_j}^{\rm N}-\lambda_j)_\limminus^{\gamma_d^{\rm N}}}{L_{\gamma^{\rm N}_d,d}^{\rm sc}|\Omega_j|\lambda_j^{\gamma_d^{\rm N}+ \frac{d}2} }= 1\,.
    \end{equation}
    Indeed, the fact that the limit is $\geq 1$ follows from the definition of $\gamma_d^{\rm N}$ and the fact that the limit is $\leq 1$ follows from \eqref{eq: failure of improved ineq Neu}, \eqref{eq: inradius bound} and the left inequality in \eqref{eq:inradd}.

    \medskip
    
    \emph{Step 2. Bounds on $\lambda_j$.} Let us show that the bounds \eqref{eq:lambdad} hold in the present situation as well. We shall argue that $\liminf_{j\to \infty}\lambda_j>0$ and $\limsup_{j\to \infty} \lambda_j <\infty$ by reaching a contradiction to~\eqref{eq: Weyl law extremal sequence Neu} in both cases.
    
    If $\liminf_{j\to 0}\lambda_j=0$, then since $\Tr(-\Delta_{\Omega_j}^{\rm N}-\lambda_j)_\limminus^{\gamma_d^{\rm N}} \geq \lambda_j^{\gamma_d^{\rm N}}$ one finds
    \begin{equation*}
        \limsup_{j\to \infty} \frac{\Tr(-\Delta_{\Omega_j}^{\rm N}-\lambda_j)_\limminus^{\gamma_d^{\rm N}}}{L_{\gamma^{\rm N}_d,d}^{\rm sc}|\Omega_j|\lambda_j^{\gamma_d^{\rm N}+ \frac{d}2} }\geq \limsup_{j\to \infty} \frac{1}{L_{\gamma^{\rm N}_d,d}^{\rm sc}|\Omega_j|\lambda_j^{\frac{d}2} } = \infty\,.
    \end{equation*}
    This contradicts \eqref{eq: Weyl law extremal sequence Neu}.

	Assuming that $\limsup_{j\to \infty} \lambda_j=\infty$, we argue as in the Dirichlet case by invoking Theorem~\ref{thm: Asymptotics degenerating convex sets} and deducing
	$$
	\limsup_{j\to\infty} \frac{\Tr(-\Delta_{\Omega_j}^{\rm N}-\lambda_j)_\limminus^{\gamma_d^{\rm N}}}{L_{\gamma_d^{\rm N}, d}^{\rm sc}\lambda_j^{\gamma_d^{\rm N}+\frac d2} |\Omega_j|}>1\,,
	$$
	which contradicts~\eqref{eq: Weyl law extremal sequence Neu}. Thus, we have shown the validity of \eqref{eq:lambdad}.

    \medskip

    We can now complete the proof of \ref{itm: Main thm existence sup} for $\sharp =N$ in the same way as in the Dirichlet case, using the continuity of the Riesz means of the Neumann Laplacian with respect to Hausdorff convergence of convex sets (see Proposition~\ref{prop: DN continuity of eigenvalues}). We omit the details.

    \medskip

	{\noindent \it Proof of \ref{itm: Main thm cylinder sequence}}: By the validity of P\'olya's conjecture in one dimension $\gamma_1^\sharp=0$ and so there is nothing to prove when $d=2$. Assume that $\gamma_{d-1}^{\sharp}\geq \frac{1}{2}$ for some value of $d\geq 3$. By Theorems~\ref{thm: Extended range of semiclassical ineq Dir} and~\ref{thm: Extended range of semiclassical ineq Neu} we have $\gamma_{d-2}^{\sharp}<1$ and thus $(\gamma_{d-2}^{\sharp}-\frac{1}{2})_\limplus< \frac{1}2 \leq \gamma_{d-1}^{\sharp}$. Consequently, the validity of \ref{itm: Main thm existence sup} in dimension $d-1$ implies that there exists an open non-empty convex set $\omega_*\subset \R^{d-1}$ and a $\lambda_*>0$ such that
	\begin{equation}\label{eq: equality case in d-1}
		\Tr(-\Delta_{\omega_*}^{\sharp}-\lambda_*)_\limminus^{\gamma_{d-1}^{\sharp}} = L^{\rm sc}_{\gamma_{d-1}^{\sharp},d-1}\Haus^{d-1}(\omega_*)\lambda_*^{\gamma_{d-1}^{\sharp}+\frac{d-1}{2}}\,.
	\end{equation}
    We shall show that the pair $(\omega_*, \lambda_*)$ satisfies the desired conclusions. 

    Using Lemma~\ref{lem: cylinder lift}, \eqref{eq: equality case in d-1}, and the assumption that $\gamma_{d}^\sharp = \gamma_{d-1}^\sharp -\frac{1}2$ it follows that
	\begin{align*}
		\Biggl|\frac{\Tr(-\Delta_{\Omega(\lambda)}^{\sharp}-\lambda)_\limminus^{\gamma_d^{\sharp}}}{L_{\gamma_d^{\sharp}, d}^{\rm sc} |\Omega(\lambda)|\lambda^{\gamma_d^\sharp + \frac{d}2}}- 1 \Biggr| 
        &= \Biggl|\frac{\Tr(-\Delta_{\Omega(\lambda)}^{\sharp}-\lambda)_\limminus^{\gamma_d^{\sharp}}}{L_{\gamma_d^{\sharp}, d}^{\rm sc} |\Omega(\lambda)|\lambda^{\gamma_d^\sharp+ \frac{d}2}}- \frac{\Tr(-\Delta_{\omega_*}^{\sharp}-\lambda_*)_\limminus^{\gamma_{d-1}^{\sharp}}}{L_{\gamma_{d-1}^{\sharp}, d-1}^{\rm sc} \Haus^{d-1}(\omega_*)\lambda_*^{\gamma_{d-1}^\sharp+ \frac{d-1}2}}\Biggr| \\
        &\leq 
        \lambda^{-\frac{d}2}\frac{\Tr(-\Delta_{\omega_*}^\sharp-\lambda_*)_\limminus^{\gamma_d^\sharp}}{L_{\gamma_d^\sharp,d}^{\rm sc}\Haus^{d-1}(\omega_*)\lambda_*^{\gamma_d^\sharp}}\,,
	\end{align*}
	which completes the proof of \ref{itm: Main thm cylinder sequence}.

	\medskip

	This completes the proof of Theorem~\ref{thm: Main conclusions gamma_d}.
\end{proof}


\section{Asymptotic shape optimization problems}
\label{sec: proof shape opt}

We now turn our attention to the shape optimization problems and prove Proposition~\ref{prop:shapeoptasymp} and Theorem~\ref{thm: shape optimization convex}.

Theorem \ref{thm: Asymptotics degenerating convex sets} describing eigenvalue asymptotics for collapsing convex sets is an important ingredient in our proofs. As in the previous section, we take this theorem for granted and defer its proof to the following section.


We first turn to the proof of Proposition \ref{prop:shapeoptasymp} concerning the asymptotics of the shape optimization problems $M_\gamma^\sharp(\lambda)$ as $\lambda\to\infty$. In fact, we shall prove the following stronger result which implies Proposition \ref{prop:shapeoptasymp}.
\begin{proposition}\label{prop:shapeoptasymp r}
    Let $d\geq 2$, $\gamma\geq 0$ and $\sharp \in \{{\rm D}, {\rm N}\}$. Then, 
    $$
    \lim_{\lambda \to \infty}\frac{M_{\gamma}^{\sharp}(\lambda)}{L_{\gamma,d}^{\rm sc}\lambda^{\gamma+\frac{d}2}} = r_{\gamma+ \frac{1}2,d-1}^\sharp\,.
    $$
\end{proposition}

\begin{proof}[Proof of Proposition \ref{prop:shapeoptasymp r}]
    Taking a ball $B\subset\R^d$ of unit measure as a competitor and using the Weyl asymptotics, we find
    $$
    M^{\rm D}_\gamma(\lambda) \geq \Tr(-\Delta_{B}^{\rm D}-\lambda)_\limminus^\gamma  = L_{\gamma, d}^{\rm sc} \lambda^{\gamma+\frac{d}{2}} \left(1 + o(1) \right).
    $$
    The reverse inequality holds in the Neumann case. Thus,
    \begin{equation}
        \label{eq:shapoptasymproof1}
        \liminf_{\lambda\to\infty} \frac{M^{\rm D}_\gamma(\lambda)}{L_{\gamma, d}^{\rm sc}\lambda^{\gamma+ \frac{d}2}}\geq 1 \geq\limsup_{\lambda\to\infty} \frac{M^{\rm N}_\gamma(\lambda)}{L_{\gamma, d}^{\rm sc}\lambda^{\gamma+ \frac{d}2}} \,.
    \end{equation} 
    While we aim to show that these bounds are sharp for $\gamma \geq \gamma_{d-1}^\sharp -\frac{1}2$ we shall see that when $\gamma <\gamma_{d-1}^\sharp - \frac{1}2$ we can construct better trial sequences.

   We begin with the Dirichlet case. By the definition of $r_{\gamma+ \frac{1}2, d-1}^{\rm D}$, for any $\delta>0$ there exists a bounded convex set $\omega_*\subset \R^{d-1}$ and a $\lambda_* >0$ so that
   \begin{equation}\label{eq: almost equality case}
    \Tr(-\Delta_{\omega_*}^{\rm D}-\lambda_*)_\limminus^{\gamma + \frac{1}2} \geq (1-\delta)r_{\gamma+ \frac{1}2,d-1}^{\rm D}L_{\gamma+ \frac{1}2,d-1}^{\rm sc}\Haus^{d-1}(\omega_*)\lambda_*^{\gamma+ \frac{d}2}\,.
   \end{equation}
    By scaling we assume without loss of generality that $\Haus^{d-1}(\omega_*)=1$.

    Set $\kappa_\lambda := \bigl(\frac{\lambda_{*}}{\lambda}\bigr)^{\frac{1}2}$ and define a family of convex sets $\Omega'_\lambda \subset \R^d$ by
	 \begin{equation*}
	 	\Omega'_\lambda := (\kappa_\lambda \omega_*)\times (0, \kappa_\lambda^{1-d})\,.
	 \end{equation*}
	 Observe that $|\Omega'_\lambda|=1$. By Lemma~\ref{lem: cylinder lift} and~\eqref{eq: almost equality case}, it follows that
	 \begin{align*}
	 	\Tr(-\Delta_{\Omega'_\lambda}^{\rm D}-\lambda)_\limminus^\gamma 
	 	&\geq
	 	L^{\rm sc}_{\gamma, 1}\lambda^{\gamma+\frac{d}{2}} \frac{\Tr(-\Delta_{\omega_*}^{\rm D}-\lambda_{*})_\limminus^{\gamma+\frac{1}{2}}}{\lambda_{*}^{\gamma+\frac{d}{2}}} 
	 	- \lambda^{\gamma}\frac{\Tr(-\Delta_{\omega_*}^{\rm D}-\lambda_{*})_\limminus^{\gamma}}{\lambda_{*}^{\gamma}}\\
	 	&\geq 
	 	(1-\delta)r_{\gamma+ \frac{1}2,d-1}^{\rm D}L^{\rm sc}_{\gamma, d}\lambda^{\gamma+\frac{d}{2}}
	 	- \lambda^{\gamma}\frac{\Tr(-\Delta_{\omega_*}^{\rm D}-\lambda_{*})_\limminus^{\gamma}}{\lambda_{*}^{\gamma}}\,.
	 \end{align*}
    This is the desired asymptotic lower bound that shows that
    $$
    \liminf_{\lambda\to \infty}\frac{M_\gamma^{\rm D}(\lambda)}{L_{\gamma,d}^{\rm sc}\lambda^{\gamma+ \frac{d}2}} \geq \liminf_{\lambda\to \infty}\frac{\Tr(-\Delta_{\Omega'_\lambda}^{\rm D}-\lambda)_\limminus^\gamma}{L_{\gamma,d}^{\rm sc}\lambda^{\gamma+ \frac{d}2}} \geq  (1-\delta)r_{\gamma+ \frac{1}2,d-1}^{\rm D}\,.
    $$
    Since $\delta>0$ was arbitrary this shows that
    \begin{equation}\label{eq: optimal collapsing sequence Dir}
    \liminf_{\lambda\to \infty}\frac{M_\gamma^{\rm D}(\lambda)}{L_{\gamma,d}^{\rm sc}\lambda^{\gamma+ \frac{d}2}}  \geq r_{\gamma+ \frac{1}2,d-1}^{\rm D}\,.
    \end{equation}
    For the Neumann case it can be proved in an analogous manner that
    \begin{equation}\label{eq: optimal collapsing sequence Neu}
    \limsup_{\lambda\to \infty}\frac{M_\gamma^{\rm N}(\lambda)}{L_{\gamma,d}^{\rm sc}\lambda^{\gamma+ \frac{d}2}}  \leq r_{\gamma+ \frac{1}2,d-1}^{\rm N}\,.
    \end{equation}
    Note that since $r_{\gamma_d}^{\rm D}\geq 1$ and $r_{\gamma,d}^{\rm N}\leq 1$ for all $\gamma, d$ the inequalities in \eqref{eq:shapoptasymproof1} can be deduced from \eqref{eq: optimal collapsing sequence Dir} and \eqref{eq: optimal collapsing sequence Neu}.

    We now show that the reverse of the inequalities \eqref{eq: optimal collapsing sequence Dir} and \eqref{eq: optimal collapsing sequence Neu} hold. Let $\{\lambda_j\}_{j\geq 1}\subset (0,\infty)$ be a sequence with $\lambda_j\to\infty$ and let $\{\Omega_j\}_{j\geq 1}$ be a sequence of bounded, convex, open sets with $|\Omega_j|=1$ such that
    $$
    \lim_{j\to\infty} \frac{\Tr(-\Delta_{\Omega_j}^\sharp-\lambda_j)_\limminus^\gamma}{M_\gamma^\sharp(\lambda_j)} = 1 \,.
    $$

    We begin by showing that
    \begin{equation}
        \label{eq:inradshaped0}
        \liminf_{j\to\infty} r_{\rm in}(\Omega_j)\sqrt{\lambda_j} > 0 \,.
    \end{equation}
    For $\sharp=$ D we note that by the Hersch--Protter inequality, $\Tr(-\Delta_{\Omega}^{\rm D}-\lambda)_\limminus^\gamma=0$ if $\lambda\leq\frac{\pi^2}{4 r_{\rm in}(\Omega)^2}$ and any convex set $\Omega\subset\R^d$. Thus, $r_{\rm in}(\Omega_j) \sqrt{\lambda_j} \leq\frac{\pi}{2}$ for all sufficiently large $j$, proving \eqref{eq:inradshaped0} for $\sharp=$ D. For $\sharp=$ N, by Lemma~\ref{lem: small energy improved Kroger} $\Tr(-\Delta_\Omega^{\rm N}-\lambda)_\limminus^\gamma \geq C_{\gamma,d} \frac{|\Omega|}{r_{\rm in}(\Omega)}\lambda^{\gamma+ \frac{d-1}{2}}$ for any convex set $\Omega \subset \R^d, \lambda \geq 0$. From this and \eqref{eq:shapoptasymproof1} we deduce that $r_{\rm in}(\Omega_j)\sqrt{\lambda_j}\geq \frac{2 C_{\gamma,d}}{L_{\gamma,d}^{\rm sc}}$ for all $j$ sufficiently large, proving \eqref{eq:inradshaped0} for $\sharp=$ N.
    
    Passing to a subsequence, we may assume that
    \begin{equation}
        \label{eq:inradalternative}
        \text{either}\quad \lim_{j\to\infty} r_{\rm in}(\Omega_j)\sqrt{\lambda_j} = \infty
        \quad
        \text{or}\quad
        \limsup_{j\to\infty} r_{\rm in}(\Omega_j)\sqrt{\lambda_j}<\infty \,.
    \end{equation}
    In the first case we can apply Theorem \ref{thm: quantitative Weyl law} (when $\gamma=0$) or \cite[Theorem 1.2]{FrankLarson_24} (when $\gamma>0$) and deduce that
    $$
    \Tr(-\Delta_{\Omega_j}^\sharp-\lambda_j)_\limminus^\gamma = L_{\gamma, d}^{\rm sc} \lambda_j^{\gamma+\frac{d}{2}} \left(1 + o(1) \right).
    $$
    (Here we also used $\Haus^{d-1}(\partial\Omega)\leq \frac{d|\Omega|}{r_{\rm in}(\Omega)}$ by~\eqref{eq: inradius bound}.) Therefore, in this case we conclude that
    \begin{equation*}
        \lim_{j\to \infty} \frac{M_\gamma^\sharp(\lambda_j)}{L_{\gamma,d}^{\rm sc}\lambda_j^{\gamma+ \frac{d}2}} = 1\,.
    \end{equation*}
    This is the desired reverse of \eqref{eq: optimal collapsing sequence Dir} and \eqref{eq: optimal collapsing sequence Neu} if $r_{\gamma+ \frac{1}2,d-1}^\sharp = 1$, that is if $\gamma + \frac{1}2\geq \gamma_{d-1}^\sharp$. If $\gamma + \frac{1}2< \gamma_{d-1}^\sharp$ then $r_{\gamma+ \frac{1}2,d-1}^{\sharp} > 1$ when $\sharp ={\rm D}$ and $r_{\gamma+ \frac{1}2,d-1}^{\sharp} < 1$ when $\sharp ={\rm N}$. Therefore the above limit contradicts \eqref{eq: optimal collapsing sequence Dir} resp.\ \eqref{eq: optimal collapsing sequence Neu} and so for $\gamma <\gamma_{d-1}^\sharp-\frac{1}{2}$ we must have $\limsup_{j\to\infty} r_{\rm in}(\Omega_j)\sqrt{\lambda_j} < \infty$.
    
    In the second case in \eqref{eq:inradalternative}, recalling \eqref{eq:inradshaped0}, we can apply Theorem~\ref{thm: Asymptotics degenerating convex sets} and we deduce that there exists an integer $m \in [1, d-1]$ and a non-empty bounded open convex set $\Omega_*\subset \R^d$ such that, along a suitably chosen subsequence (specified below),
	 \begin{equation}\label{eq: collapsing asymptotics shape optasymp}
	 	\lim_{j\to \infty} \frac{\Tr(-\Delta_{\Omega_j}^\sharp-\lambda_j)_\limminus^\gamma}{L^{{\rm sc}}_{\gamma, d}\lambda_j^{\gamma+\frac{d}2}}  = \frac{1}{L^{\rm sc}_{\gamma + \frac{d-m}{2},m}|\Omega_*|}\int_{P^\perp \Omega_*}\Tr(-\Delta_{\Omega_*(y)}^\sharp-1)_\limminus^{\gamma+\frac{d-m}{2}}\,dy\,.
	 \end{equation}
    When $\sharp ={\rm D}$ we choose the subsequence above so that the limit on the left side of \eqref{eq: collapsing asymptotics shape optasymp} coincides with the corresponding $\limsup$ along the original sequence, and for $\sharp ={\rm N}$ we instead choose a subsequence along which the corresponding $\liminf$ is realized.

    By the definition of $r_{\gamma,d}^\sharp$, the trace under the integral in \eqref{eq: collapsing asymptotics shape optasymp} can be bounded from above (for $\sharp=\rm D$) or below (for $\sharp=\rm N$) by 
    $$
    r_{\gamma+\frac{d-m}2,m}^\sharp L_{\gamma+\frac{d-m}2,m}^{\rm sc} \mathcal H^m(\Omega_*(y)) \,.
    $$
    Thus, bounding the trace under the integral in \eqref{eq: collapsing asymptotics shape optasymp} by the claimed quantity and using
    $$
    \frac{1}{L^{\rm sc}_{\gamma + \frac{d-m}{2},m}|\Omega_*|}\int_{P^\perp \Omega_*} L_{\gamma+\frac{d-m}2,m}^{\rm sc} \mathcal H^m(\Omega_*(y)) \,dy = 1 \,,
    $$
    we deduce that
    \begin{align*}
    \limsup_{j\to\infty} \frac{\Tr(-\Delta_{\Omega_j}^{\rm D} -\lambda_j)_\limminus^\gamma}{L^{{\rm sc}}_{\gamma, d}\lambda_j^{\gamma+\frac{d}2}} \leq \max_{m=1, \ldots, d-1}\Bigl\{r_{\gamma+ \frac{d-m}2,m}^{\rm D}\Bigr\}\,,\\
    \liminf_{j\to\infty} \frac{\Tr(-\Delta_{\Omega_j}^{\rm N}-\lambda_j)_\limminus^\gamma}{L^{{\rm sc}}_{\gamma, d}\lambda_j^{\gamma+\frac{d}2}} \geq \min_{m=1, \ldots, d-1}\Bigl\{r_{\gamma+ \frac{d-m}2,m}^{\rm N}\Bigr\}\,.
    \end{align*}
    From the Berezin--Li--Yau and Kr\"oger inequalities it follows that for any $m\leq d-2$ we have $r_{\gamma+ \frac{d-m}2,m}^\sharp =1$ as $\gamma+\frac{d-m}{2} \geq \gamma +1 \geq 1\geq \gamma_m^\sharp$. Therefore, by the monotonicity of $\gamma \mapsto r_{\gamma,d}^\sharp$,
    \begin{align*}
        \limsup_{j\to\infty} \frac{M_\gamma^{\rm D}(\lambda_j)}{L^{{\rm sc}}_{\gamma, d}\lambda_j^{\gamma+\frac{d}2}} =\limsup_{j\to\infty} \frac{\Tr(-\Delta_{\Omega_j}^{\rm D} -\lambda_j)_\limminus^\gamma}{L^{{\rm sc}}_{\gamma, d}\lambda_j^{\gamma+\frac{d}2}} \leq r_{\gamma+ \frac{1}2,d-1}^{\rm D}\,,\\
        \liminf_{j\to\infty} \frac{M_\gamma^{\rm N}(\lambda_j)}{L^{{\rm sc}}_{\gamma, d}\lambda_j^{\gamma+\frac{d}2}}=\liminf_{j\to\infty} \frac{\Tr(-\Delta_{\Omega_j}^{\rm N}-\lambda_j)_\limminus^\gamma}{L^{{\rm sc}}_{\gamma, d}\lambda_j^{\gamma+\frac{d}2}} \geq r_{\gamma+ \frac{1}2,d-1}^{\rm N}\,.
    \end{align*}
    As $\{\lambda_j\}_{j\geq 1}$ was arbitrary, we have proved the desired converse of inequalities \eqref{eq: optimal collapsing sequence Dir} and \eqref{eq: optimal collapsing sequence Neu} and therefore completed the proof of the proposition.
\end{proof}

\begin{proof}[Proof of Proposition \ref{prop:shapeoptasymp}]
    Since $r_{\gamma+ \frac{1}2,d-1}^{\rm N}\leq 1 \leq r_{\gamma+ \frac{1}2,d-1}^{\rm D}$ for all $\gamma$, and $r_{\gamma+\frac{1}2,d-1}^{\sharp}$ is equal to $1$ if and only if $\gamma+\tfrac{1}2 \geq \gamma_{d-1}^\sharp$ Proposition \ref{prop:shapeoptasymp} follows from Proposition \ref{prop:shapeoptasymp r}.
\end{proof}

We now describe the behavior of optimizing sequences for the shape optimization problems $M_\gamma^\sharp(\lambda)$. To do so, we carefully revisit the argument leading to Proposition~\ref{prop:shapeoptasymp r}.

\begin{proof}[Proof of Theorem~\ref{thm: shape optimization convex}]
    {\noindent \it Cases \ref{itm: Shape opt thm super critical} and \ref{itm: Shape opt thm counting super critical}}:
    We show that in both cases we have
    \begin{equation}
        \label{eq:inradinfty}
        \lim_{j\to\infty} r_{\rm in}(\Omega_j)\sqrt{\lambda_j}=\infty \,.
    \end{equation}
    
    Similarly as in the proof of Proposition \ref{prop:shapeoptasymp r} we argue by contradiction and consider a subsequence along which $r_{\rm in}(\Omega_j)\sqrt{\lambda_j}$ remains finite. Applying Theorem~\ref{thm: Asymptotics degenerating convex sets} we obtain \eqref{eq: collapsing asymptotics shape optasymp} with an integer $m \in [1, d-1]$ and a non-empty bounded open convex set $\Omega_*\subset \R^d$. Here we used the fact that \eqref{eq:inradshaped0} is valid in the present setting by the same argument as in the proof of Proposition \ref{prop:shapeoptasymp r}. As before, the limit in \eqref{eq: collapsing asymptotics shape optasymp} is computed along a subsequence realizing the $\limsup$ for $\sharp=\rm D$ and the $\liminf$ for $\sharp=\rm N$.

    We claim that 
    \begin{equation}
        \label{eq:strictineq}
        \gamma + \frac{d-m}{2} > \gamma_{m}^\sharp \,.
    \end{equation}
    Indeed, this is clear when $m\leq d-2$, because then, by Theorems~\ref{thm: Extended range of semiclassical ineq Dir} and \ref{thm: Extended range of semiclassical ineq Neu}, we have $\gamma + \frac{d-m}{2}\geq 1 > \gamma_{m}^\sharp$. When $m=d-1$, we have \eqref{eq:strictineq} since the assumptions of cases \ref{itm: Shape opt thm super critical} and \ref{itm: Shape opt thm counting super critical} imply that $\gamma + \frac{1}{2} > \gamma_{d-1}^\sharp$. 
    
    As a consequence of \eqref{eq:strictineq}, Theorem~\ref{thm: improved inequality above critical gamma Dir} implies that the traces under the integral in \eqref{eq: collapsing asymptotics shape optasymp} satisfy a two-term semiclassical inequality, and so
	\begin{align*}
	 	\limsup_{j\to \infty} \frac{\Tr(-\Delta_{\Omega_j}^{\rm D}-\lambda_j)_\limminus^\gamma}{L^{{\rm sc}}_{\gamma, d}\lambda_j^{\gamma+\frac{d}2}}  &\leq \frac{1}{|\Omega_*|}\int_{P^\perp \Omega_*}\Bigl(\mathcal H^{m}(\Omega_*(y))-c_{\gamma,m,d}\mathcal H^{m-1}(\partial \Omega_*(y))\Bigr)_\limplus\,dy\\
	 	&<1\,.
	\end{align*}
	This contradicts the trivial inequality \eqref{eq:shapoptasymproof1} obtained by comparison with a ball. Thus, we have proved \eqref{eq:inradinfty} for $\sharp=\rm D$. The proof in the case $\sharp=\rm N$ is similar, using Theorem~\ref{thm: improved inequality above critical gamma Neu} instead of Theorem~\ref{thm: improved inequality above critical gamma Dir}. This completes the proof of \ref{itm: Shape opt thm counting super critical}. 
 
    To complete the proof of \ref{itm: Shape opt thm super critical}, we note that by \cite[Theorem~1.2]{FrankLarson_24} and \eqref{eq:inradinfty}
    $$
    \Tr(-\Delta_{\Omega_j}^{\rm D}-\lambda_j)_\limminus^\gamma = L_{\gamma,d}^{\rm sc} \lambda_j^{\gamma+\frac d2} - \frac14 L_{\gamma,d-1}^{\rm sc} \mathcal H^{d-1}(\partial\Omega_j) \lambda_j^{\gamma+\frac{d-1}{2}} + o(\lambda_j^{\gamma+\frac{d-1}2})
    \quad\text{as}\ j\to\infty \,.
    $$
    Similar asymptotics hold when $\Omega_j$ is replaced by $B$, a ball of unit measure. Since $M^{\rm D}_\gamma(\lambda)\geq \Tr(-\Delta_B^{\rm D}-\lambda)_\limminus^\gamma$, we deduce that
	 \begin{equation*}
	 	\liminf_{j\to \infty} \frac{\Tr(-\Delta_{\Omega_j}^{\rm D}-\lambda_j)_\limminus^\gamma-M^{\rm D}_\gamma(\lambda_j)}{\lambda_j^{\gamma+ \frac{d-1}{2}}} \leq \frac{L_{\gamma,d-1}^{\rm sc}}{4}\liminf_{j\to \infty}\bigl(\Haus^{d-1}(\partial B)-\Haus^{d-1}(\partial\Omega_j)\bigr)\,.
	 \end{equation*}
	 Therefore, the assumption in the theorem implies that $$\limsup_{j \to \infty} \Haus^{d-1}(\partial\Omega_j) \leq \Haus^{d-1}(\partial B)\,.$$
	 In particular, by~\eqref{eq: inradius bound} we conclude that $\{\Omega_j\}_{j\geq 1}$ is a uniformly bounded sequence. By the Blaschke selection theorem we can extract a non-empty, open, convex limit (up to translations) with respect to the Hausdorff distance. Since volume and perimeter are continuous functions with respect to this convergence, the isoperimetric inequality (including characterisation of the equality case) implies that the limiting set must be a ball of unit volume. Therefore, the proof shows that $M_\gamma^{\rm D}(\lambda_j)= \Tr(-\Delta_B^{\rm D}-\lambda_j)_\limminus^\gamma+o(\lambda_j^{\gamma+ \frac{d-1}2})$. As the sequence $\{\lambda_j\}_{j \geq 1}$ was arbitrary, this proves statement \ref{itm: Shape opt thm super critical} of the theorem for $\sharp = \rm D$. The proof of \ref{itm: Shape opt thm super critical} for $\sharp = \rm N$ is analogous.

\medskip
    {\noindent \it Case \ref{itm: Shape opt thm sub critical case}}:
    We show that in both cases we have
    \begin{equation}
        \label{eq:inradfinite}
        \limsup_{j\to\infty} r_{\rm in}(\Omega_j)\sqrt{\lambda_j}<\infty
    \end{equation}
    and that the sequence $\{\Omega_j\}_{j\geq 1}$ collapses in a spaghetti-like manner.

    We argue by contradiction assuming the $\limsup$ in \eqref{eq:inradfinite} is infinite. Then by Theorem~\ref{thm: quantitative Weyl law} (for $\gamma=0$) and \cite[Theorem 1.2]{FrankLarson_24} (for $\gamma>0$) we have
    $$
    \Tr(-\Delta_{\Omega_j}^\sharp -\lambda_j)_\limminus^\gamma = L^{{\rm sc}}_{\gamma, d}\lambda_j^{\gamma+\frac{d}2} \left( 1+ o(\lambda_j) \right)
    \qquad\text{as}\ j\to\infty \,.
    $$
    This contradicts the assertion of Proposition \ref{prop:shapeoptasymp} if $\gamma<\gamma_{d-1}^\sharp-\frac12$. This proves \eqref{eq:inradfinite} in case~\ref{itm: Shape opt thm sub critical case} under the assumption that $\gamma<\gamma_{d-1}^\sharp-\frac12$.

    Note that here we only used the leading order term in the asymptotics for Riesz means. To deal with the case~\ref{itm: Shape opt thm sub critical case} when $0<\gamma = \gamma_{d-1}^\sharp-\frac12$ we use the two-term asymptotics in \cite[Theorem 1.2]{FrankLarson_24}. Then, by arguing as in our proof of \ref{itm: Shape opt thm super critical}, one finds that $\{\Omega_j\}_{j\geq 1}$ converges (along a subsequence and up to translations) to a ball $B\subset\R^d$ of unit volume in Hausdorff distance and that
	 \begin{equation}\label{eq: assumed asymptotics of Mgamma}
     \begin{aligned}
	 		M^{\rm D}_\gamma(\lambda_j) &= \Tr(-\Delta_B^{\rm D}-\lambda_j)_\limminus^\gamma+o(\lambda_j^{\gamma+ \frac{d-1}2})\\
            & =L_{\gamma, d}^{\rm sc} \lambda_j^{\gamma +\frac{d}{2}}- \frac{L_{\gamma, d-1}^{\rm sc}}{4}\Haus^{d-1}(\partial B)\lambda_j^{\gamma+ \frac{d-1}{2}}+ o(\lambda_j^{\gamma+ \frac{d-1}{2}})\,.
     \end{aligned}
	 \end{equation}
    Meanwhile, by definition we have $M^{\rm D}_\gamma(\lambda_j)\geq \Tr(-\Delta_{\Omega}^{\rm D}-\lambda_j)_\limminus^\gamma$ for any convex $\Omega\subset \R^d$ with $|\Omega|=1$. By choosing $\Omega = \Omega(\lambda_j)$ as in Theorem \ref{thm: Main conclusions gamma_d} \ref{itm: Main thm cylinder sequence}, the lower bound in Lemma~\ref{lem: cylinder lift} contradicts~\eqref{eq: assumed asymptotics of Mgamma} in the limit $j \to \infty$. The argument in the Neumann case is similar. This proves \eqref{eq:inradfinite} in case~\ref{itm: Shape opt thm sub critical case} and $0<\gamma = \gamma_{d-1}^\sharp-\frac12$.

   To complete the proof of the theorem, we need to show that $\{\Omega_j\}_{j\geq 1}$ collapses in a spaghetti-like manner. To do so, we argue as in the proof of Proposition \ref{prop:shapeoptasymp r} and use \eqref{eq:inradshaped0} (which is valid in the present setting by the same argument as in the proof of Proposition \ref{prop:shapeoptasymp r}) and \eqref{eq:inradfinite}. By Theorem \ref{thm: Asymptotics degenerating convex sets} and along any subsequence along which the limit exists we obtain \eqref{eq: collapsing asymptotics shape optasymp} with an integer $m \in [1, d-1]$ and a non-empty bounded open convex set $\Omega_*'\subset \R^d$. The case $m\leq d-2$ is excluded in the same way as in the proof of cases~\ref{itm: Shape opt thm super critical} and \ref{itm: Shape opt thm counting super critical}. Consequently, $m=d-1$ and looking at the proof of Theorem \ref{thm: Asymptotics degenerating convex sets} we obtain the claimed spaghetti-like degeneration.
\end{proof}


\begin{remark}
By following the same strategy, using the fact that by Theorems~\ref{thm: improved inequality above critical gamma Dir} and~\ref{thm: improved inequality above critical gamma Neu} two-term semiclassical inequalities hold for $\gamma >\gamma_d^{\sharp}$, we are also able to prove similar results regarding the corresponding shape optimization problem among sets that are disjoint unions of convex sets. The obtained results show that for $\gamma >\gamma_d^{\sharp}$ sequences of (almost) optimizing sets converge to a single ball as $\lambda \to \infty$ and that the corresponding convergence fails if $\gamma <\gamma_d^\sharp$. Remarkably, in this setting the critical value at which the asymptotic nature of the problem changes is $\gamma_d^\sharp$ rather than $\gamma_{d-1}^\sharp-\frac12$ as in Theorem \ref{thm: shape optimization convex}. That is, the dimensional reduction that is observed in Propositions~\ref{prop:shapeoptasymp}, \ref{prop:shapeoptasymp r} and Theorem \ref{thm: shape optimization convex} disappears. For detailed statements and proofs the reader is referred to the preprint version of this paper available on arXiv \cite{FrankLarson_24b}.

\end{remark}


\section{Quasiclassical asymptotics for collapsing convex domains}
\label{sec:asymptoticscollapsing}

\subsection{Proof of Theorem~\ref{thm: Asymptotics degenerating convex sets}}

In this subsection we want to compute the asymptotics of $\Tr(-\Delta^\sharp_{\Omega_j}-\lambda_j)_\limminus^\gamma$ for a sequence $\{\Omega_j\}_{j\geq 1}$ of bounded open convex subsets of $\R^d$ and a corresponding sequence $\{\lambda_j\}_{j\geq 1}$ of positive numbers such that
\begin{equation*}
	0< \inf_{j\geq 1}r_{\rm in}(\Omega_j)\sqrt{\lambda_j}\leq \sup_{j\geq 1}r_{\rm in}(\Omega_j)\sqrt{\lambda_j} <\infty
	\quad \mbox{and}\quad
 	\lim_{j\to \infty}|\Omega_j|\lambda_j^{\frac{d}2}=\infty\,.
\end{equation*} 
Before embarking into the rather lengthy proof, let us give a rough description of how we will proceed.

The strategy is perhaps best illustrated by the example
$$
\Omega_j = \omega^{(1)} \times (\ell_j \omega^{(2)} ) \,,
\qquad
\lambda_j = 1 \,, 
$$
where $1\leq m\leq d-1$, where $\omega^{(1)}\subset\R^m$, $\omega^{(2)}\subset\R^{d-m}$ are bounded open convex sets and where $\lim_{j\to\infty} \ell_j=\infty$. Here $m$ directions, corresponding to variables $x\in\R^m$, remain of order 1, while $d-m$ directions, corresponding to variables $y\in\R^{d-m}$ tend to infinity. It is the $y$-directions that become semiclassical, while the other directions retain their `quantum' character. In the example this means concretely that the partial trace with respect to the semiclassical directions is replaced by the corresponding phase space integral, that is,
\begin{align*}
    \Tr(-\Delta^\sharp_{\Omega_j}-1)_\limminus^\gamma & \approx
\iint_{(\ell_j\omega^{(2)})\times\R^{d-m}} \Tr(-\Delta_{\omega^{(1)}}^\sharp + |\xi|^2 -1)_\limminus^\gamma \ \frac{dx\,d\xi}{(2\pi)^{d-m}} \\
& = L_{\gamma,d-m}^{\rm sc} \ell_j^{d-m} |\omega^{(2)}| \Tr(-\Delta_{\omega^{(1)}}^\sharp -1)_\limminus^{\gamma+\frac{d-m}2} \,.
\end{align*}
Setting $\Omega_* = \omega^{(1)}\times\omega^{(2)}$, the right side can be written as
$$
L_{\gamma,d}^{\rm sc} |\Omega_j| \lambda_j^{\gamma+\frac d2} \times \frac{1}{L_{\gamma+\frac{d-m}2,m}^{\rm sc} |\Omega_*|} \int_{P^\bot \Omega_*} \Tr(-\Delta_{\Omega_*(y)}^\sharp -1)_\limminus^{\gamma+\frac{d-m}2} \,dy \,,
$$
which is what Theorem \ref{thm: Asymptotics degenerating convex sets} claims.

Needless to say, the case of general convex sets $\Omega_j$ is quite a bit more involved, but the basic mechanism is the same as in the example. A first step is to find, using John's theorem, an integer $m$ and $m$ directions in which the domain is of size $\sim 1/\!\sqrt{\lambda_j}$, while in the remaining $d-m$ directions the domain is of size $\gg 1/\!\sqrt{\lambda_j}$. (In contrast to the above example, the sizes in the latter directions need not all be the same.)

To prove the claimed asymptotics we use the technique of Dirichlet--Neumann bracketing. More precisely, we will impose additional Dirichlet and Neumann boundary conditions in the long directions. In the above example these boundary conditions will be imposed on the sets $\Omega_j \cap \{ (x,z):\ x\in\R^m\}$ where $z$ runs through $M \Z^{d-m}$ with a parameter $M$ that will eventually be sent to infinity. This allows us to pass to the semiclassical limit in the $y$-variables in the example.

In the general case, $M/\!\sqrt{\lambda_j}$ is an auxiliary length scale that is chosen much larger than the short directions, but much shorter than the large directions.  In this manner the spectral problem after bracketing retains both the semiclassical nature of the expanding directions and the 'quantum nature' in the cross-sections orthogonal to the expanding directions.

After this motivation we now turn to the details.

\begin{proof}[Proof of Theorem~\ref{thm: Asymptotics degenerating convex sets}]
    We want to compute
    \begin{equation*}
		\lim_{j\to \infty} \frac{\Tr(-\Delta^\sharp_{\Omega_j}-\lambda_j)_\limminus^\gamma}{L^{\rm sc}_{\gamma, d}|\Omega_j|\lambda_j^{\frac{d}2}}\,,
	\end{equation*}
    whose existence we are assuming.
	
	Define $l_j^{(1)} \leq l^{(2)}_j \leq \ldots \leq l_j^{(d)}$ to be the axes of the John ellipsoid $\mathcal{E}_j$ of $\Omega_j$. After a translation and an orthogonal transformation we may assume that $\mathcal{E}_j$ is centered at the origin and has semiaxes parallel to the coordinate axes with
	\begin{equation*}
		\mathcal{E}_j \subset \prod_{k=1}^d [-l_j^{(k)}, l_j^{(k)}]\,.
	\end{equation*}
	We know that $l_j^{(1)} \sim r(\Omega_j)$, $l_j^{(d)} \sim \mathrm{diam}(\Omega_j)$, and $\prod_{k=1}^d l_j^{(k)} \sim |\Omega_j|$. 
	
    By assumption, we have $l_j^{(1)}\sqrt{\lambda_j} \sim r(\Omega_j)\sqrt{\lambda_j}\lesssim 1$ and
    $$
    \prod_{k=2}^d (l_j^{(k)}\sqrt{\lambda_j}) \sim \lambda_j^{\frac{d}2}|\Omega_j|/(l_j^{(1)}\sqrt{\lambda_j}) \to \infty.
    $$
    The latter implies $\limsup_{j \to \infty} l_j^{(d)}\sqrt{\lambda_j} = \infty$. Thus there is an integer $1\leq m< d$ such that
    \begin{equation*}
		\limsup_{j\to \infty} l_j^{(m)}\sqrt{\lambda_j}  < \infty
	\end{equation*}
    and
	\begin{equation*}
		\limsup_{j\to \infty} l_j^{(k)}\sqrt{\lambda_j}  = \infty
        \qquad\text{for all}\ k>m \,.
	\end{equation*}
	
	By passing to a subsequence we can without loss of generality assume that
	\begin{equation}\label{eq: limit assumptions}
		\lim_{j\to \infty} l_j^{(m+1)}\sqrt{\lambda_j}=\infty\quad \mbox{and} \quad \lim_{j\to \infty}\{(l_j^{(1)}\sqrt{\lambda_j}, \ldots, l_j^{(m)}\sqrt{\lambda_j})\}_{j\geq 1} = (a_1, \ldots, a_{m})
	\end{equation}
	for some $(a_1, \ldots, a_{m})\in (0, \infty)^{m}$.

	Define the diagonal matrices
    \begin{align*}
		T_j^\parallel &= \mathrm{diag}\bigl(1/l_j^{(1)}, \ldots,  1/l^{(m)}_j\bigr)
		\,, & \hspace{-30pt} 
		T^\perp_j &= \mathrm{diag}\bigl(1/l_j^{(m+1)}, \ldots, 1/l^{(d)}_j\bigr)\,,\\
		T_j &= \left(\begin{matrix}
			T_j^\parallel & 0\\
			0 & T_j^\perp
		\end{matrix}\right)\,, & \hspace{-30pt} 
		\tilde T_j &= \left(\begin{matrix}
			\sqrt{\lambda_j}\, \mathrm{I}_{m} & 0\\
			 0& T_j^\perp
		\end{matrix}\right)\,,
	\end{align*}
	where $\mathrm{I}_m$ is the $m\times m$ identity matrix. Note that by~\eqref{eq: limit assumptions} the matrices $T_j, \tilde T_j$ are comparable in the limit in the sense that 
	\begin{equation}\label{eq: comparability of renormalization maps}
		\tilde T_j T_j^{-1} \to A:=
        \mathrm{diag}(a_1, \ldots, a_m, 1, \ldots, 1) 
        \quad \mbox{as }j \to \infty\,.
	\end{equation}

	By construction the convex sets $T_j\Omega_j$ all contain $B_1$ and are contained in $B_d$ by~\cite[Theorem 10.12.2]{Schneider_book14}. By the Blaschke selection theorem $T_j \Omega_j$ has a subsequence that converges with respect to the Hausdorff distance to a non-empty open convex set $\Omega_\infty$. From here on we work with such a subsequence. By~\eqref{eq: comparability of renormalization maps} the sequence $\tilde T_j \Omega_j$ also converges, we denote its limit by
    $$
    \Omega_*= A\Omega_\infty \,.
    $$ 
    (This is the set that appears in the statement of the theorem).

	We say that the sequence $\Omega_j$ becomes long in the directions corresponding to the indexes $\{m+1, \ldots, d\}$ and collapses in the directions $\{1, \ldots, m\}$.

	For a convex set $U \subset \R^d$ and with $m$ as defined above we set
	\begin{align*}
		P^\perp U &= \{y \in \R^{d-m}: (x, y) \in U\quad \mbox{for some }x \in \R^{m}\}\,,\\
		P U &= \{x \in \R^{m}: (x, y) \in U\quad \mbox{for some }y \in \R^{d-m}\}\,.
	\end{align*}
	Note that $P^\perp \tilde T_j = P^\perp T_j= T_j^\perp P^\perp$ and by continuity of the projection map we note that $T_j^\perp P^\perp \Omega_j \to P^\perp \Omega_\infty=P^\perp \Omega_*$ in the Hausdorff distance. Here we used the fact that $A$ acts as the identity in the last $d-m$ coordinates.
	
	For $\ell>0$ to be chosen and $z \in \R^{d-m}$ set $Q_{z,\ell} := z + (-\ell/2, \ell/2)^{d-m}$. For a convex set $U \subset \R^d$ define
	\begin{align*}
		U[z, \ell] &:=  U \cap (\R^{m}\times Q_{z,\ell})\,.
	\end{align*}
	Note that $U[z, \ell]$ is itself a convex set and
	note that 
	\begin{equation*}
		P^\perp(U[z, \ell]) = (P^\perp U) \cap Q_{z,\ell}\,.
	\end{equation*}
	The boundary of $U[z, \ell]$ is naturally split into two parts:
	\begin{align*}
		\partial U[z, \ell]_\perp & := \partial U[z, \ell] \setminus U\,,\\
		\partial U[z, \ell]_\parallel & := \partial U[z, \ell] \cap U\,.
	\end{align*}
    Note that $\partial U[z, \ell]_\parallel$ is a subset of $\R^m \times \partial Q_{z, \ell}$.
	
	In addition to the Laplace operators with Dirichlet resp.\ Neumann boundary conditions on $U[z, \ell]$ we shall need to consider the following operators with mixed boundary conditions. We write 
	\begin{equation*}
		-\Delta^{{\rm D}_\perp, {\rm N}_\parallel}_{U[z, \ell]}\,,\ -\Delta^{{\rm N}_\perp, {\rm D}_\parallel}_{U[z, \ell]} 
	\end{equation*}
	for the Laplace operators in $L^2(U[z, \ell])$ with Dirichlet boundary conditions on the $\perp$-part of the boundary and Neumann on the $\parallel$-part, or vice versa. 
	
	For any $\Omega\subset \R^d$ the variational principle implies that
	\begin{equation}\label{eq: bracketing bound}
		\sum_{z \in \ell\Z^{d-m}} \Tr(-\Delta_{\Omega[z, \ell]}^{\sharp_\perp, {\rm D}_\parallel}-\lambda)_\limminus^\gamma\leq \Tr(-\Delta^{\sharp}_{\Omega}-\lambda)_\limminus^\gamma \leq \sum_{z \in \ell\Z^{d-m}} \Tr(-\Delta_{\Omega[z, \ell]}^{\sharp_\perp, {\rm N}_\parallel}-\lambda)_\limminus^\gamma\,,
	\end{equation}
	for all $\lambda\geq 0, \gamma \geq 0$.

	For $M>0, j\geq 1$, and $\sharp, \flat \in \{{\rm D}, {\rm N}\}$ define $F^{\sharp, \flat}_{j, M}\colon \R^{d-m}\to [0, \infty)$ by
	\begin{align*}
		F^{\sharp, \flat}_{j, M}(y) &= \frac{\prod_{k=m+1}^d l_j^{(k)}}{|\Omega_j|\lambda_j^{\gamma+m/2}M^{d-m}}\sum_{z \in M\lambda_j^{-1/2}\Z^{d-m}} \Tr(-\Delta_{\Omega_j[z, M\lambda_j^{-1/2}]}^{\sharp_\perp, {\flat}_\parallel}-\lambda_j)_\limminus^\gamma \1_{T_j^\perp Q_{z, M\lambda_j^{-1/2}}}(y)\,.
	\end{align*}
	Since the functions $F_{j,M}^{\sharp, \flat}$ are supported in the set
    $$\bigl\{y\in \R^{d-m}: \dist((T^{\perp}_j)^{-1}y, P^\perp\Omega_j)\leq \sqrt{d}M/(2\sqrt{\lambda_j})\bigr\}$$ and since $P^\perp T_j\Omega_j \to P^\perp \Omega_*$ in Hausdorff distance, all of the functions are zero outside of a sufficiently large ball as long as $M$ remains fixed. Indeed, if $y_j\in\R^{d-m}$ satisfies $\dist((T^{\perp}_j)^{-1}y_j, P^\perp\Omega_j)\leq \sqrt{d}M/(2\lambda_j^{1/2})$, then
	\begin{align*}
		\frac{d M^2}{4\lambda_j}&\geq \dist((T_j^\perp)^{-1}y_j, P^\perp \Omega_j)^2\\
		& = \inf_{y\in P^\perp \Omega_j} \sum_{k=1}^{d-m}|l_j^{({m+k})}y_j^k-y^k|^2\\
		& = \inf_{y\in P^\perp T_j\Omega_j} \sum_{k=1}^{d-m}(l_j^{({m+k})})^2|y_j^k-y^k|^2\\
		&\geq (l_j^{(m+1)})^2\dist(y_j, P^\perp T_j\Omega_j)^2\,.
	\end{align*}
	Consequently, $\dist(y_j, P^\perp T_j\Omega_j) \leq \frac{\sqrt{d}M}{2l_j^{(m+1)}\sqrt{\lambda_j}}$. Since $P^\perp T_j \Omega_j \to P^\perp \Omega_*$ (which is a bounded set) and $l_j^{(m+1)}\sqrt{\lambda_j}\to \infty$ by~\eqref{eq: limit assumptions}, we conclude that $\{y_j\}_{j\geq 1}$ is a bounded sequence. Moreover, we see that
	\begin{equation*}
		\bigcap_{j\geq 1}\supp F_{j,M}^{\sharp, \flat} = \overline{P^\perp \Omega_*}\,.
	\end{equation*}

	By a change of variables we see that
	\begin{align*}
		\int_{\R^{d-m}}F^{\sharp, \flat}_{j, M}(y)\,dy = \frac{1}{|\Omega_j|\lambda_j^{\gamma+\frac{d}2}}\sum_{z \in M\lambda_j^{-1/2}\Z^{d-m}} \Tr(-\Delta_{\Omega_j[z, M\lambda_j^{-1/2}]}^{\sharp_\perp, {\flat}_\parallel}-\lambda_j)_\limminus^\gamma 
	\end{align*}
	so by~\eqref{eq: bracketing bound}
	\begin{equation}\label{eq: int bracketing bound}
		\int_{\R^{d-m}}F^{\sharp, \rm D}_{j, M}(y)\,dy \leq \frac{\Tr(-\Delta_{\Omega_j}^\sharp-\lambda_j)_\limminus^\gamma}{|\Omega_j|\lambda_j^{\gamma+\frac{d}2}}\leq \int_{\R^{d-m}}F^{\sharp, \rm N}_{j, M}(y)\,dy\,.
	\end{equation}
	We claim that the integrands are uniformly bounded in $j$ and $y$. As the integrands are non-negative it suffices to prove a uniform bound from above.
	
	By scaling of Laplacian eigenvalues, we have
	\begin{align*}
		\Tr(-\Delta_{\Omega_j[z, M\lambda_j^{-1/2}]}^{\sharp_\perp, {\flat}_\parallel}-\lambda_j)_\limminus^\gamma = \lambda_j^\gamma \Tr(-\Delta_{\lambda_j^{1/2}\Omega_j[z, M\lambda_j^{-1/2}]}^{\sharp_\perp, {\flat}_\parallel}-1)_\limminus^\gamma \,.
	\end{align*}
	Using this together with the fact that $$|\Omega_j| = \det(\tilde T_j)^{-1}|\tilde T_j\Omega| = \lambda_j^{-\frac{m}2}\Bigl(\prod_{k=m+1}^dl_j^{(k)}\Bigr)|\tilde T_j\Omega_j|$$ we can equivalently write
	\begin{align*}
		F^{\sharp, \flat}_{j, M}(y) &= \frac{1}{|\tilde T_j\Omega_j|M^{d-m}}\sum_{z \in M\lambda_j^{-1/2}\Z^{d-m}} \Tr(-\Delta_{\lambda_j^{1/2}\Omega_j[z, M\lambda_j^{-1/2}]}^{\sharp_\perp, {\flat}_\parallel}-1)_\limminus^\gamma \1_{T_j^\perp Q_{z, M\lambda_j^{-1/2}}}(y)\,.
	\end{align*}
	Since
	\begin{equation*}
		\lim_{j\to \infty}|\tilde T_j \Omega_j| = |\Omega_*|>0 \,,
	\end{equation*}
	the factor in front of the sum in the above expression for the $F$'s remains uniformly bounded in $j$ as long as $M$ is fixed.

    To see that the traces in the sum remain bounded from above we first observe that by the variational principle
	\begin{equation*}
		F^{\sharp, \flat}_{j, M}(y) \leq F^{{\rm N}, {\rm N}}_{j, M}(y)
	\end{equation*}
	for all $y, j, M$. By construction the set $\sqrt{\lambda_j}\Omega_j[z, M/\!\sqrt{\lambda_j}]$ is, up to translation, contained in the cuboid 
	\begin{equation*}
		\biggl(\prod_{k=1}^m(-dl_j^{(k)}\sqrt{\lambda_j}, dl_j^{(k)}\sqrt{\lambda_j})\biggr) \times (-M/2, M/2)^{d-m} \,. 
	\end{equation*}
	Since $\sup_{j\geq 1}l_j^{(k)}\sqrt{\lambda_j}<\infty$ for all $k\leq m$ there exist $R>0$ and a sequence $\{a_j\}_{j\geq 1}\subset \R^d$ so that for all $j\geq 1$
	\begin{equation*}
		\sqrt{\lambda_j}\Omega_j[z, M/\!\sqrt{\lambda_j}]+ a_j \subset (-R, R)^{m} \times (-M/2, M/2)^{d-m} =: \mathcal{Q}\,.
	\end{equation*}
	By Lemma~\ref{lem: upper bound N counting},
	\begin{align*}
		\Tr(-\Delta_{\lambda_j^{1/2}\Omega_j[z, M\lambda_j^{-1/2}]}^{\rm N}-1)_\limminus^\gamma
		&\leq 
        \Tr(-\Delta_{\lambda_j^{1/2}\Omega_j[z, M\lambda_j^{-1/2}]}^{\rm N}-1)_\limminus^0\\
        &\leq C_d \diam(\sqrt{\lambda_j}\Omega_j[z, M/\!\sqrt{\lambda_j}])^d + 1\\
        &\leq  C_d \diam(\mathcal{Q})^d + 1\,.
	\end{align*}
	This proves the claim that the $F^{\sharp, \flat}_{j,M}$ are uniformly bounded with respect to $j$.

	By construction $\sqrt{\lambda_j}\Omega_j[z, M/\!\sqrt{\lambda_j}]= (\sqrt{\lambda_j}\Omega_j)[\sqrt{\lambda_j}z, M]$, and
	\begin{align*}
		\sqrt{\lambda_j}\Omega_j[z, M/\!\sqrt{\lambda_j}]
		&=
		\sqrt{\lambda_j}\Bigl\{ (x,y) \in \Omega_j: y\in Q_{z, M\lambda_j^{-1/2}}\Bigr\}\\
		&=
		\Bigl\{(x',y')\in \R^d: x' \in \sqrt{\lambda_j}\Omega_j(y'/\!\sqrt{\lambda_j}),\, y'\in Q_{z\lambda_j^{1/2}, M}\Bigr\}\,.
	\end{align*}
	In particular, since the cross-sections $\sqrt{\lambda_j}\Omega_j(\cdot)$ are uniformly bounded, also the sets $\sqrt{\lambda_j}\Omega_j[z, M/\!\sqrt{\lambda_j}]$ are uniformly bounded convex sets. Thus after translation and passing to a subsequence these sets converge in the Hausdorff distance to some (possibly empty) convex set.

	Fix $y \in P^\perp \Omega_\infty$. Let $\{z_j\}_{j\geq 1}\subset P^\perp \Omega_j$ with $z_j \in M/\!\sqrt{\lambda_j}\,\Z^{d-m}$ be a sequence such that $T^\perp_jz_j \to y$. We first note that the sequence of sets $T^\perp_j Q_{z_j, M\lambda_j^{-1/2}}$ concentrates at $y$. Indeed, for any sequence $\{z'_j\}_{j\geq 1}$ with $z_j' \in Q_{z_j, M\lambda_j^{-1/2}}$ for each $j$ we have
	\begin{align*}
	 	|y-T^\perp_jz_j'| 
	 	&= |y-T^\perp z_j-T^\perp_j(z_j'-z_j)| \\
	 	&\leq |y-T^\perp_jz_j| + \|T^\perp_j\||z_j'-z_j|\\
	 	&\leq |y-T^\perp_jz_j| + \frac{\sqrt{d} M}{2l_j^{(m+1)}\sqrt{\lambda_j}} \to 0\,.
	\end{align*} 
	So, in particular, $P^\perp \Omega_j[z_j, M/\!\sqrt{\lambda_j}]\subset P^\perp \Omega_\infty$ for all $j$ large enough. This implies that $P^\perp \Omega_j[z_j, M/\!\sqrt{\lambda_j}]= Q_{z_j, M\lambda_j^{-1/2}}$ for all $j$ sufficiently large. 
 
    We next aim to prove that this implies that any cross-section of $\sqrt{\lambda_j}\Omega_j[z_j,M/\!\sqrt{\lambda_j}]$ in the first $m$-coordinates converges to the same, non-empty, open, convex set. To see this, we note that if $\{z_j'\}_{j\geq 1}$ is a sequence as before then
	\begin{align*}
		(\sqrt{\lambda_j}\Omega_j[z_j, M/\!\sqrt{\lambda_j}])&(\sqrt{\lambda_j}z_j')\\
		&=
		\Bigl\{x\in \R^{m}: x \in (\sqrt{\lambda_j}\Omega_j)(\sqrt{\lambda_j} z_j'), \sqrt{\lambda_j} z_j' \in \sqrt{\lambda_j}Q_{z_j, M\lambda_j^{-1/2}}\Bigr\}\\
		&=
		\Bigl\{x\in \R^{m}: x \in \sqrt{\lambda_j}(\Omega_j(z_j'))\Bigr\}\\
		&= \sqrt{\lambda_j}\Omega_j(z_j')\\
		&= (\tilde T_j \Omega_j)(T_j^\perp z_j')\,.
	\end{align*}

	As $j \to \infty$ it holds that
	\begin{equation*}
		(\tilde T_j \Omega_j)(T_j^\perp z_j')\to \Omega_*(y)\,,
	\end{equation*}
	with respect to the Hausdorff distance. Indeed, this follows from Lemma~\ref{lem: convergence of sections of sequence of convex sets} in the next subsection by noting that the claim can equivalently be written as
	\begin{equation*}
		(\tilde T_j \Omega_j-(0, y-T_j^\perp z_j'))(y)\to \Omega_*(y)\,,
	\end{equation*}
	and since $\tilde T_j \Omega_j \to \Omega_*$ and $T_j^\perp z_j'\to y$ it follows that $\tilde T_j \Omega_j-(0, y-T_j^\perp z_j')\to \Omega_*$.

	Consequently, we conclude that up to a translation
	\begin{equation*}
		\sqrt{\lambda_j}\Omega_j[z_j, M/\!\sqrt{\lambda_j}] \to \Omega_*(y) \times Q_{0, M}
	\end{equation*}
	with respect to the Hausdorff distance. Note that this convergence holds for any fixed value of the parameter $M$. Therefore, by Propositions \ref{prop: DN continuity of eigenvalues} and \ref{prop: continuity of mixed eigenvalues} for any fixed $\gamma >0$ and $\sharp, \flat\in\{{\rm D}, {\rm N}\}$ it holds that
	\begin{align*}
		\lim_{j\to \infty}\Tr(-\Delta_{\lambda_j^{1/2}\Omega_j[z_j, M\lambda_j^{-1/2}]}^{\sharp_\perp, \flat_\parallel}-1)_\limminus^\gamma &= \Tr(-\Delta_{\Omega_*(y)\times Q_{0,M}}^{\sharp_\perp, \flat_\parallel}-1)_\limminus^\gamma \,.
	\end{align*}
	If $\gamma= 0$ and $\epsilon>0$, the same argument combined with the monotonicity of the counting functions yields
	\begin{align*}
		\limsup_{j\to \infty}\Tr(-\Delta_{\lambda_j^{1/2}\Omega_j[z_j, M\lambda_j^{-1/2}]}^{\sharp_\perp, {\rm N}_\parallel}-1)_\limminus^0 &\leq \Tr(-\Delta_{\Omega_*(x)\times Q_{0,M}}^{\sharp_\perp, {\rm N}_\parallel}-1-\epsilon)_\limminus^0 \,, \\
		\liminf_{j\to \infty}\Tr(-\Delta_{\lambda_j^{1/2}\Omega_j[z_j, M\lambda_j^{-1/2}]}^{\sharp_\perp, {\rm D}_\parallel}-1)_\limminus^0 &\geq \Tr(-\Delta_{\Omega_*(x)\times Q_{0,M}}^{\sharp_\perp, {\rm D}_\parallel}-1+\epsilon)_\limminus^0\,.
	\end{align*}

	Using again montonicity we conclude that, for all $\gamma \geq 0$, any fixed $M,\epsilon>0$, and almost every $y \in P^\perp\Omega_*$,
	\begin{align*}
		\liminf_{j\to \infty} F_{j, M}^{\sharp, \rm D}(y) &\geq \frac{\Tr(-\Delta_{\Omega_*(y)\times Q_{0,M}}^{\sharp_\perp, {\rm D}_\parallel}-1+\epsilon)_\limminus^\gamma}{|\Omega_*|M^{d-m}}\,,\\ 
        \limsup_{j\to \infty} F_{j, M}^{\sharp, \rm N}(y) &\leq \frac{\Tr(-\Delta_{\Omega_*(y)\times Q_{0,M}}^{\sharp_\perp, {\rm N}_\parallel}-1-\epsilon)_\limminus^\gamma}{|\Omega_*|M^{d-m}}\,.
	\end{align*}

	By the dominated convergence theorem we conclude from~\eqref{eq: int bracketing bound} that for every fixed $M, \epsilon>0$
	\begin{equation}\label{eq: dominated convergence 1}
    \begin{aligned}
		\lim_{j\to \infty} \frac{\Tr(-\Delta_{\Omega_j}^\sharp-\lambda_j)_\limminus^\gamma}{|\Omega_j|\lambda_j^{\gamma+\frac{d}2}} &\geq \frac{1}{|\Omega_*|} \int_{P^\perp\Omega_*} \frac{\Tr(-\Delta_{\Omega_*(y)\times Q_{0,M}}^{\sharp_\perp, {\rm D}_\parallel}-1+\epsilon)_\limminus^\gamma}{M^{d-m}}\,dy\,,\\
    \lim_{j\to \infty} \frac{\Tr(-\Delta_{\Omega_j}^\sharp-\lambda_j)_\limminus^\gamma}{|\Omega_j|\lambda_j^{\gamma+\frac{d}2}} &\leq \frac{1}{|\Omega_*|} \int_{P^\perp\Omega_*} \frac{\Tr(-\Delta_{\Omega_*(y)\times Q_{0,M}}^{\sharp_\perp, {\rm N}_\parallel}-1-\epsilon)_\limminus^\gamma}{M^{d-m}}\,dy \,.
	\end{aligned}
    \end{equation}
    
	To get to the expressions in the theorem we need to study the integrands and get rid of the parameters $M, \epsilon$. By the product structure of $\Omega_*(y)\times Q_{0, M}$ we have that for any combination of $\sharp, \flat \in \{{\rm D}, {\rm N}\}$,
	\begin{align*}
		\Tr(-\Delta_{\Omega_*(y)\times Q_{0,M}}^{{\sharp}_\perp, {\flat}_\parallel}-1 \pm \epsilon)_\limminus^\gamma
		&= \sum_{n\geq 1}\Tr(-\Delta_{Q_{0,M}}^\flat-1 \pm \epsilon+\lambda_n(-\Delta_{\Omega_*(y)}^\sharp))_\limminus^\gamma\\
		&= \sum_{n\geq 1}\Bigl(M^{d-m}L_{\gamma,d-m}^{\rm sc}(1\mp \epsilon-\lambda_n(-\Delta^{\sharp}_{\Omega_*(y)})_\limplus)^{\gamma+ \frac{d-m}{2}}\\
		&\quad + O\bigl(M^{d-m-1}(1\mp\epsilon-\lambda_n(-\Delta^{\sharp}_{\Omega_*(y)})_\limplus)^{\gamma+ \frac{d-m-1}{2}}\bigr)\Bigr)\\
		&= M^{d-m}L_{\gamma,d-m}^{\rm sc}\Tr(-\Delta^{\sharp}_{\Omega_*(y)}-1 \pm \epsilon)_\limminus^{\gamma+ \frac{d-m}{2}}\\
		&\quad + O\bigl(M^{d-m-1}\Tr(-\Delta^{\sharp}_{\Omega_*(y)}-1 \pm \epsilon)_\limminus^{\gamma+ \frac{d-m-1}{2}}\bigr)\,.
	\end{align*}
	
	Using again Lemma~\ref{lem: upper bound N counting} as above one deduces that the integrands in the right-hand side of~\eqref{eq: dominated convergence 1} are uniformly bounded with respect to $M$. Therefore, using again the dominated convergence theorem with $M \to \infty$ we deduce
	\begin{align*}
		\lim_{j\to \infty} \frac{\Tr(-\Delta_{\Omega_j}^\sharp-\lambda_j)_\limminus^\gamma}{|\Omega_j|\lambda_j^{\gamma+\frac{d}2}} &\geq \frac{1}{|\Omega_*|} \int_{P^\perp\Omega_*} L_{\gamma,d-m}^{\rm sc}\Tr(-\Delta^{\sharp}_{\Omega_*(y)}-1+\epsilon)_\limminus^{\gamma+ \frac{d-m}{2}}\,dy\,,\\
        \lim_{j\to \infty} \frac{\Tr(-\Delta_{\Omega_j}^\sharp-\lambda_j)_\limminus^\gamma}{|\Omega_j|\lambda_j^{\gamma+\frac{d}2}} &\leq \frac{1}{|\Omega_*|} \int_{P^\perp\Omega_*} L_{\gamma,d-m}^{\rm sc}\Tr(-\Delta^{\sharp}_{\Omega_*(y)}-1-\epsilon)_\limminus^{\gamma+ \frac{d-m}{2}}\,dy\,.
	\end{align*}
	Since $\gamma + \frac{d-m}{2}>0$, continuity of the Riesz means combined with the monotone convergence theorem now allows us to send $\epsilon \to 0$. Dividing by $L_{\gamma,d}^{\rm sc}$ completes the proof.
\end{proof}


\subsection{Technical ingredients}

In this subsection we discuss three technical results needed in the proof of Theorem \ref{thm: Asymptotics degenerating convex sets}. The first concerns cross-sections of convex sets and the latter two concern the continuity of Riesz means with respect to changing the underlying geometry.

In the first lemma we use the same notation as in the statement of Theorem \ref{thm: Asymptotics degenerating convex sets}, except that for clarity we now write $P_m^\perp\Omega$ instead of $P^\perp\Omega$.

\begin{lemma}\label{lem: convergence of sections of sequence of convex sets}
	Let $\Omega_j, \Omega \subset \R^d$ be non-empty, bounded, and open convex sets such that $\Omega_j \to \Omega$ in the Hausdorff topology, and let $1\leq m \leq d-1$. Then for any $y \in P_m^\perp \Omega$ it holds that $\Omega_j(y) \to \Omega(y)$ with respect to the Hausdorff distance.
\end{lemma}

\begin{proof}
	Since $\{\Omega_j(y)\}_{j\geq 1}$ is a sequence of non-empty, bounded convex sets we can, by the Blaschke's selection theorem, extract a subsequence so that $\Omega_j(y)$ converges in the Hausdorff distance to an open convex set $\Omega_\infty$. The proof is complete if we can identify $\Omega_\infty$ as $\Omega(y)$.

	We use that for $y \in P_m^\perp\Omega$,
	\begin{equation*}
		\Omega(y) = \{x \in \R^m: d_\Omega((x,y))>0\}\,.
	\end{equation*}
	Since $\Omega_j \to \Omega$ it follows that $d_{\Omega_j} \to d_\Omega$ uniformly. In particular, if $y \in P^\perp_m \Omega$ then there exists an $x \in \R^m$ such that $\epsilon:= d_\Omega((x,y))>0$ and so for $j$ large enough $d_{\Omega_j}(( x,y)) \geq \epsilon/2 >0$ and so $y \in P^\perp_m\Omega_j$ and $x \in \Omega_j(y)$. Consequently, for every $\delta>0$, $\{x\in \R^m: d_\Omega((x,y))>\delta\} \subset \Omega_\infty$. Since $d_\Omega$ is continuous and $|\nabla d_\Omega|=1$ almost everywhere in $\Omega$ we conclude that $\Omega(y) \subset \Omega_\infty$.
 
    Note that if $U, U'$ are open sets and $x_0 \in U' \setminus U$ and $K$ is a compact set with $U, U' \subset K$, then
    \begin{equation*}
        d^H(U, U') \geq \sup_{x\in K\setminus U}\inf_{x'\in K\setminus U'} |x-x'| \geq \inf_{x'\in K\setminus U'} |x_0-x'| = d_{U'}(x_0)\,.
    \end{equation*}
    In particular, if $x \in \Omega_\infty$ then $x \in \Omega_j(y)$ for any $j$ so large that $d^H(\Omega_\infty, \Omega_j(y))<d_{\Omega_\infty}(x)$. Therefore, $(x, y) \in \Omega_j$ for large enough $j$. By uniform convergence $d_{\Omega_j}\to d_{\Omega}$ there exists $x_0\in \Omega$ so that $d_{\Omega_j}(x_0)\geq r_{\rm in}(\Omega)/2$ for $j$ large enough. Therefore, both $(x, y) \in \Omega_j$ and $B_{r_{\rm in}(\Omega)/2}(x_0)$ are contained in $\Omega_j$. Consequently, by convexity of $\Omega_j$ the convex hull of $B_{r_{\rm in}(\Omega)/2}(x_0)\cup\{(x,y)\}$ is also contained in $\Omega_j$ for each $j$ sufficiently large. As $d_\Omega$ is the uniform limit of the functions $d_{\Omega_j}$ we conclude that $d_\Omega$ is positive in this convex hull, therefore, $(x, y)$ belongs to $\overline{\Omega}=\overline{\{z\in \R^d: d_\Omega(z)>0\}}$. Since $y \in P^\perp\Omega$ it follows that $x \in \overline{\Omega(y)}$. We conclude that $\Omega_\infty \subset \overline{\Omega(y)}$.
 
    The two inclusions $\Omega(y)\subset \Omega_\infty, \Omega_\infty\subset \overline{\Omega(y)}$ implies (by convexity) that $\Omega_\infty = \Omega(y)$ and completes the proof.
\end{proof}

Next we state two results concerning sufficient conditions on a (convergent, in some sense) sequence of open sets to have convergence of the corresponding Riesz means. For complete proofs of the statements the reader is referred to the preprint version of this paper available on arXiv \cite{FrankLarson_24b}. The first result concerns semi-continuity of Riesz means of Dirichlet and Neumann Laplace operators along sequences of convex sets converging in the Hausdorff sense. The argument is based on the continuity of individual eigenvalues, which is well known; see, for instance, \cite{Ross_04,HenrotPierre_18}.

\begin{proposition}\label{prop: DN continuity of eigenvalues}
    Let $\{\Omega_j\}_{j\geq 0}$ be a collection of open, non-empty, and bounded convex subsets of $\R^d$ such that $\Omega_j \to \Omega_0$ with respect to the Hausdorff distance. 
    If $\{\lambda_j\}_{j\geq 1}\subset (0, \infty)$ satisfies $\lim_{j\to \infty}\lambda_j=\lambda_*$, then
    \begin{equation*}
        \lim_{j\to \infty}\Tr(-\Delta_{\Omega_j}^\sharp-\lambda_j)_\limminus^\gamma = \Tr(-\Delta_{\Omega_*}^\sharp-\lambda_*)_\limminus^\gamma \quad \mbox{for any }\gamma>0\,,
    \end{equation*}
    and
    \begin{align*}
       \Tr(-\Delta_{\Omega_*}^{\sharp}-\lambda_*)_\limminus^0
       &\leq \liminf_{j\to \infty}\Tr(-\Delta_{\Omega_j}^{\sharp}-\lambda_j)_\limminus^0\\
       &\leq \limsup_{j\to \infty}\Tr(-\Delta_{\Omega_j}^{\sharp}-\lambda_j)_\limminus^0 \leq \lim_{\epsilon\to 0^\limplus}\Tr(-\Delta_{\Omega_*}^{\sharp}-\lambda_*-\epsilon)_\limminus^0 \,.
    \end{align*}
     
\end{proposition}

The second result concerns semi-continuity of Riesz means of mixed eigenvalue problems in almost cylindrical domains. Here we use the notation introduced in the proof of Theorem~\ref{thm: Asymptotics degenerating convex sets}. 

\begin{proposition}\label{prop: continuity of mixed eigenvalues}
    Let $\{\Omega_j\}_{j\geq 1}$ be a collection of open, non-empty, and bounded convex subsets of $\R^d$. Fix an integer $1 \leq m\leq d-1$ and $\omega \subset \R^m$ an open non-empty, bounded, convex set. Let $Q_l:=(-\ell/2, \ell/2)^{d-m}$ and assume that for any fixed $\ell>0$ 
    \begin{equation*}
        \Omega_j[\ell]=\Omega_j \cap (\R^m \times Q_\ell) \to \omega\times   Q_\ell
    \end{equation*}
    with respect to the Hausdorff distance. 
    Fix $\sharp, \flat \in \{{\rm D}, {\rm N}\}$ with $\sharp \neq \flat$ and $\ell>0$. If $\{\lambda_j\}_{j\geq 1}\subset (0, \infty)$ satisfies $\lim_{j\to \infty}\lambda_j=\lambda_*$, then
    \begin{equation*}
        \lim_{j\to \infty}\Tr(-\Delta_{\Omega_j[\ell]}^{\sharp_\perp, \flat_\parallel}-\lambda_j)_\limminus^\gamma = \Tr(-\Delta_{\omega \times Q_\ell}^{\sharp_\perp, \flat_\parallel}-\lambda_*)_\limminus^\gamma \quad \mbox{for any }\gamma>0\,,
    \end{equation*}
    and
    \begin{align*}
       \Tr(-\Delta_{\omega \times Q_\ell}^{\sharp_\perp, \flat_\parallel}-\lambda_*)_\limminus^0
       &\leq \liminf_{j\to \infty}\Tr(-\Delta_{\Omega_j[\ell]}^{\sharp_\perp, \flat_\parallel}-\lambda_j)_\limminus^0\\
       &\leq \limsup_{j\to \infty}\Tr(-\Delta_{\Omega_j[\ell]}^{\sharp_\perp, \flat_\parallel}-\lambda_j)_\limminus^0 \leq \lim_{\epsilon\to 0^\limplus}\Tr(-\Delta_{\omega \times Q_\ell}^{\sharp_\perp, \flat_\parallel}-\lambda_*-\epsilon)_\limminus^0 \,.
    \end{align*}
\end{proposition}


\end{document}